\documentclass[11pt]{article}
\usepackage{latexsym}

\usepackage{amsmath,amssymb,amsthm,epsfig,epstopdf,url,array}
\usepackage[left=2.6cm,right=2.5cm,top=2.5cm,bottom=3cm]{geometry}
\usepackage{caption}
\usepackage{subcaption}
\DeclareGraphicsExtensions{.pdf,.jpeg,.png}
\newtheorem{theorem}{Theorem}[section]
\newtheorem{lemma}{Lemma}[section]
\newtheorem{proposition}{Proposition}[section]

\newtheorem{remark}{Remark}[section]

\newcommand{\R}{I\! \! R}

\newcommand{\dem}{{\it Proof:} }
\begin{document}
\title{Repulsive chemotaxis and predator evasion in  predator prey models with diffusion and prey taxis.} 

\author{Purnedu  Mishra$^a$, Dariusz Wrzosek$^a$\\
$^a$ Institute of Applied Mathematics and Mechanics,\\
 University of Warsaw, Warszawa, Poland}
\date{}

\maketitle

\begin{abstract}
The role of predator evasion mediated by chemical signaling is studied in a diffusive prey-predator model when prey-taxis is taken into account (model A) or not (model B) with taxis strength coefficients $\chi$ and $\xi$ respectively. In the  kinetic part of the models it is assumed that the rate of prey consumption include functional responses of Holling, Bedington-DeAngelis or Crowley–Martin. Existence of  global-in-time classical solutions to model A is proved  in  space dimension $n=1$  while to model B for any $n\geq 1$. The Crowley-Martin response combined with bounded rate of signal production  preclude  blow-up of solution in model A for $n\leq 3$. Local and global stability of a constant coexistence  steady state which is stable for ODE and purely diffusive model are studied along with mechanism of Hopf bifurcation for Model B when $\chi$ exceeds some critical value. In model A it is shown that  prey taxis may destabilize the coexistence steady state provided $\chi$ and $\xi$ are big enough. Numerical simulation depict emergence of complex space-time patterns for both models and indicate existence of solutions to model A which blow-up  in finite time for $n=2$. 
\end{abstract}
\textbf{Keywords:} Predator-prey model; Chemo-repulsion; Direct taxis; Taxis-driven instability; Pattern formation.

\section{Introduction}
We study the effect of predator evasion mediated by chemical signaling described as chemorepulsion in an extended classical diffusive prey-predator model. It is well known that many chemicals (e.g. pheromones, kairomones)  released  by plants and animals are used as means of inter and intraspecific communication. Olfaction is a primary means by which prey animals detect predators \cite{Nolte} and trigger anti-predator responses. In the present paper we consider the case when the chemical signal  is diffusive and plays the role of alarm signal stimulating the  antipredator  response. Many types of anti-predator responses to chemical cues are  described  in the literature \cite{Banks,Ferrari,Hay}. It is enough to mention  induced morphological defense and behavioral responses. Among many behavioral prey strategies \cite{Rojas}  to the threat of  predation, each of them worth of modeling attempts, we concentrate  in this paper on escape (evasion)  in response to the gradient of chemical signal indicating the spot of high predator concentration (the long list of possible antipredator responses of prey  including the escape caused by chemical signal is provided in \cite{Connover,HB,Kats,ZB}). One of our goals  is to verify if   classical diffusive predator prey models enriched by terms accounting for  chemical signaling can describe the tendency to   spatiotemporal separation between prey and predators, by either avoiding areas inhabited by potential predators or using those areas at different times than the predators. Denoting the densities  of the prey, predator and the chemical by $N, P, W:\Omega\mapsto \R$, respectively, the model reads 
\begin{align} 
\label{G1}\left\{
\begin{aligned}
N_t&= D_1 \Delta N +\nabla \cdot (\chi N \nabla W) +   f(N) - F(N,P,W)P\,,\\
P_t&= D_2 \Delta P -\nabla \cdot (\xi P \nabla N)-\delta P+ bF(N,P,W)P\,,\\
W_t&=D_3 \Delta W-\mu W+g(N,P,W)\,,
\end{aligned}
\right.
\end{align}
defined in a  bounded domain $\Omega\subset\R^n$  with smooth boundary and outer normal $\nu$, supplemented  with  initial conditions 
\begin{equation}
 \label{IC} N(\cdot ,0)=N_0,\ P(\cdot ,0)=P_0, W(\cdot ,0)=W_0
\end{equation}
and  homogeneous Neumann boundary conditions 
\begin{equation}
\label{BC}
\langle\nabla N\,,\nu\rangle=\langle\nabla P\,,\nu\rangle=\langle\nabla W\,,\nu\rangle=0, \quad\mbox{on}\quad  \partial\Omega, \ t>0\,. 
\end{equation}
The function $f=f(N)$ describes the prey population growth while $F=F(N\,,P\,,W)$ is the functional response which describes the rate of prey consumption per unite predator density while $g=g(N,P,W)$  describes the rate of chemical signal production. The diffusion constants are denoted by   $D_i>0, i=1\,,\,2\,,3$, $\delta$ is predator's  death rate coefficient, $\mu$ is a chemical degradation rate  and  $b$    is a coefficient related to the conversion efficiency  of food into offspring. We consider for the sake of generality a hypothetical  situation when the functional response may be affected by  the chemical. 
The avoidance of predator by prey is upon  detection of chemical released by predator (e.g. predator odor)  which stimulates migration outward  the gradient of the  chemical concentration (chemorepulsion). The corresponding sensitivity coefficient is denoted by $\chi>0$.
System (\ref{G1}) is general enough to grasp many models known from  the literature with different   prey consumption rates per predator  i.e.   functional responses \cite{Hol} as well as different mechanisms of chemical production. 

We assume the following assumptions on functions $f$,  $F$ and $g$ which comprise many models used in the biomathematical literature.
Denoting $ [0,+\infty):=\R_+$ we assume the following restrictions on functions $f, F$ and $g$ 
\begin{enumerate}
\item[(H1)] The function $f: \R_+ \mapsto \R$ is a $C^2$-function such that there exist constants $r_1\,,r_2>0 $ and $K$ such that $f(0)=0,f(K)=0$ and 
\[f(N)\leq r_1 N -r_2 N^2 \quad \mbox{ for any} \;N\geq 0\,.
\]
\item[(H2)] The function $F :\R_+^3 \mapsto \R_+$ is a $C^2$-function such that for some  constants $C_F>0$ 
\[F(N,P,W)\leq C_F  \quad \mbox{for any} \; N, P, W \geq  0 \,.
\]
\item[(H3)] The function $g : \R_+^3 \mapsto \R_+$ is a $C^2$-function such that for some  constants $C_g>0$ 
\[g(N,P,W)\leq C_g P \quad \mbox{for any} \;\; N, P, W \geq 0 \,.
\]
\end{enumerate} 
The typical rate of population growth  which satisfies (H1) is of course the logistic function
\[f(N)=rN\left(1-\frac{N}{K}\right)
\] 
where $r$  and $K$ are the growth rate coefficient  and the  carrying capacity, respectively. Among models of prey consumption rate which satisfy  (H2) we  may point  the  Holling type II ($d=1$) and Holling type III ($d>1$) functional responses \cite{Hol} 
\begin{equation}
\label{FH}
F=F_H(N)=\frac{aN^d}{1+bN^d}\quad a\,,b>0\,,
\end {equation}
as well as  Beddington-DeAngelis functional response (c.f. \cite{BD})   
\begin{equation}
\label{FB}
F=F_H(N)=\frac{aN}{1+bN+cP},\quad a\,,b\,,c>0\,,
\end {equation}
or that of Crowley–Martin (c.f. \cite{CM}) 
\begin{equation} \label{CM}
F(N,P)=\frac{aN}{1+bN+cP+dNP}\,,\quad a\,,b\,,c\,, d>0\,.
\end{equation}
The last two  functions incorporate  mutual interference of  predators (see e.g. \cite{Skalski} for a survey and comparison with experimental data). 
As far as the chemical signal production rate is concerned  we  consider  two cases which fall into (H3), the simplest one when the rate of chemical production (odor of prey)  is  proportional to the predator density  
\begin{equation}
\label{odor} 
g(N, P,W)= \gamma P\,,\quad \gamma>0\,,
\end{equation}
and the case when the release of signal is due to damage of prey captured (chemicals from injured conspecific c.f. \cite{Kats} ) 
\begin{equation}
\label{demage}
g(N, P,W)= g_2(N,P)= \gamma_1F(N,P,W) P\,,\quad \gamma_1>0\,.
\end{equation}

We shall distinguish the following  two  different predator-prey models describing  evasion of predator by prey due to the chemical alarm signaling:

\begin{itemize}
\item
In  model A  ($\chi>0\,,\xi> 0$) it is assumed that (H1)-(H3) are satisfied and predator's searching strategy combines random spread (diffusion) and prey taxis ($\xi>0$) which  amounts to migration toward gradient  of prey density:
\begin{align}\label{modelA}
\text{MODEL A}\left\{
\begin{aligned}
N_t&=D_1 \Delta N+\nabla \cdot (\chi N \nabla W)+ f(N)- PF(N,P,W), \\
P_t&=D_2 \Delta P -\nabla \cdot (\xi P \nabla N) -\delta P+bPF(N,P,W), \\
W_t&=D_3 \Delta W -\mu W+ g(N, P,W)\,,
\end{aligned}
\right.
\end{align}
with initial and Neumann boundary conditions defined in (\ref{IC})-(\ref{BC}).
\item
In model B  ($\xi=0$) it is assumed that (H1)-(H3) are satisfied and predator's searching strategy  is merely restricted  to the random search described by the diffusion operator: 
\begin{align}\label{modelB}
\text{MODEL B}\left\{
\begin{aligned}
N_t&=D_1 \Delta N+\nabla \cdot (\chi N \nabla W)+ f(N)- PF(N,P,W), \\
P_t&=D_2 \Delta P  -\delta P+bPF(N,P,W), \\
W_t&=D_3 \Delta W -\mu W+ g(N,P,W)\,,
\end{aligned}
\right.
\end{align}
with initial and Neumann boundary conditions defined in (\ref{IC})-(\ref{BC}).
\end{itemize}
\medskip 

\noindent
{\bf Description of our  results and the related literature.} 

Many mathematical models describing complex interactions between components of biological systems have the structure of systems of nonlinear partial differential equations of parabolic type which describe changes in time and space of densities of biological system components. Such models, capable to describe complex space-time patterns reflect local or nonlocal in space interactions between systems components including diffusive transport and taxis. We refer the reader to most recent survey papers   which contain mathematical methods and modeling perspectives for chemotaxis systems \cite{Bellomo1,Bellomo2} as well as to the survey of various patterning mechanism  in this kind of models \cite{Painter}. Our investigations contribute to the series of recent papers on population interactions taking into account indirect mechanisms of taxis associated with chemical signaling \cite{Ahnand,Amorim,BaiWinkler,TaoWinkler,Tyutyunov,TWrz,WangWang}.
We proceed to describe the content of the paper.

After the introduction in Section \ref{SectionGlobalexistence} in Theorem \ref{exist}  the existence of global uniformly bounded classical solutions to Model B is proved for space dimension $n\geq 1$ and to Model A only for $n=1$. The latter turns out to be much more difficult to study because the only available estimate is just $L^1$ -estimate and the comparison method cannot be immediately applied neither to prey nor predator equation. It is worth noticing that numerical simulations  of model A (see Fig. \ref{blowup}) in space dimension $n=2$ indicate formation of blow-up  solution in finite time an effect related to the Keller-Segel model of chemotaxis (see \cite{K-S} and \cite{Bellomo1}). This observation is particularly interesting because for space dimension $n=2$ the formation of blow-up solutions is precluded for both the predator-prey  model with prey taxis ($\chi=0$) as proved in \cite{JinWang} and for predator prey model without prey taxis term (i.e Model B where $\xi=0$) proved in Theorem \ref{exist}.  This effect may appear only in model A when there is a cumulative effect of both taxis forces and initial densities of species are suitably chosen. Neither of them can alone  lead to such an effect.
 Model A may be viewed as predator prey model with pursuit (prey taxis)  and indirect repulsive predator taxis (evasion). Therefore it  is closely related to predator-prey model with  pursuit and evasion with  chemical sensing  studied in \cite{Tyutyunov} and recently  in \cite{WaWuShi} where only local in time existence of solutions was proved, so far,  provided some smallness condition on the taxis coefficients was satisfied. The  existence of global solutions was  shown  in \cite{Amorim} at least for space dimension $n\leq 2$ in a parabolic-elliptic case when the   distribution of chemical is governed  by elliptic equation \cite{Amorim} which amounts to assume that the diffusion of the chemical  happens in a much faster time scale than the movement of individuals.   On the other hand it was recently proved in \cite{Telch} that  global boundedness of solutions also holds for the   predator-prey system with pursuit-evasion and chemical signaling   under additional assumptions on highly nonlinear diffusion of species which turns out to preclude blow-up formation in finite time. 
It is worth adding that in Theorem \ref{exist2} we prove the existence of global solutions to model A for space dimension $n\leq 3$  assuming sufficiently strong  inhibitory effect in prey consumption for high predator densities and signal production linked directly with prey consumption. This effect may be attained by taking the functional response of Crowley-Martin type (\ref{CM}).  

Section \ref{stabB} and Section \ref{stabA}, related to model B and  model A respectively, concern the stability of the coexistence steady state $\bar{E}$ which stems from the ODE part of the system. The classical Rosenzweig-MacArthur prey-predator model \cite{RM}  may  serve as a  benchmark. It is  an ODE describing the densities of predator  and prey  accounting for  a concave functional response of Holling  type II c.f. (\ref{FH}) and logistic growth for  prey. The extended Rosenzweig-MacArthur model which  accounts for the chemical signaling  reads as follows 
\begin{align}\label{RM1}
\left\{
\begin{aligned}
N_t&=D_1 \Delta N+\nabla \cdot (\chi N \nabla W)+ rN \Big(1-\frac{N}{K}\Big)- \frac{aNP}{(1+aT_h N)}, \\
P_t&=D_2 \Delta P -\nabla \cdot (\xi \nabla N) -\delta P+\frac{abNP}{(1+aT_h N)}, \\
W_t&=D_3 \Delta W -\mu W+ \gamma P\,,
\end{aligned}
\right.
\end{align}
with the initial and boundary conditions (\ref{IC})-(\ref{BC}) where  $a$ is the prey  encounter rate,  $T_h$ is the handling time and $\delta$ counts death rate of predators. Positive parameter $\gamma$ is the production rate of the  chemoattractant and $\mu$ measures  its vanishing rate. It is known that for some range of parameters there exists the unique  coexistence  steady state $E_{RM}=(\bar{N}\,,\bar{P})$ for the Rosenzweig-MacArthur  model ($D_i=W\equiv0$ in ( \ref{RM1}))  which is a global attractor provided it is linearly stable. If the carring capacity $K$ is chosen as a bifurcation parameter then for some $K=K_b$ the Hopf bifurcation arises and then the steady state loses stability and  a stable limit cycle emerges for $K>K_b$ (see e.g. \cite{Wrz}) in the ODE case or \cite{Yi} in the case of reaction-diffusion system  (see also the literature given there). The coexistence steady state for (\ref{RM1}) is of form
\[\bar{E}=(\bar{N}\,,\bar{P}\,, \bar{W})\,, \quad\mbox{where}\quad \bar{W}=\frac{\gamma}{\mu}\bar{P}\,.
\]
Typically the coexistence steady state preserves stability for some range of parameters also for models accounting for other functional responses then Holling II and our goal is to find conditions under which taxis effects may destabilize the homogeneous coexistence steady state. The linear stability analysis (Theorem \ref{ThmstabB})  of the steady state  indicates that the parameter $\chi$ plays  a crucial role as its sufficiently high values  can destabilize the steady state and trigger  the  Hopf bifurcations which give rise to spatio-temporal patterns proved to hold for Model B (see Theorem \ref{hopfthm}). The emergence of periodic and quasi-periodic space-time patterns is depicted in Section \ref{NumSim} which is mostly devoted to  numerical simulations related to extensions of the Rosenzweig-MacArthur in the frame of model A and model B.
It is worth  noticing  that this scenario is in contrast with direct repulsive predator-taxis studied recently in \cite{WWS} where it was proved that such a repulsive predator taxis does not affect linear stability of the coexistence steady state for  the Rosenzweig-MacArthur prey-predator model and in particular formation of patterns is then excluded. This is yet another example showing that direct and indirect taxis associated with chemical signaling lead to essentially different properties of models having the same predator-prey kinetics. This difference is evident   for the case of direct \cite{Ahnand,Bendamane,JinWang,Lee,Tao} and indirect prey-taxis models \cite{Ahnand,MiWrz,TWrz}.  

A natural question which arises in the analysis is whether local stability of the steady state may be extended to  global one. It turns out that for   Model B with Holling II  or Beddington-deAngelis functional responses under additional assumption of logistic competition among predators a suitable Lyapunov  functional can be found which enables to prove in Theorem \ref{GlobStab} that  there is a threshold value of $\chi=\chi_0$ such that for $\chi<\chi_0$ the coexistence steady state $\bar{E}$ is indeed globally stable. Just before submitting the paper the authors have  learned about the paper \cite{Ahnand2} in which global existence of solutions and its long time behavior are studied to a system similar to Model B in which Lotka-Volterra kinetics  was assumed ($T_h=0$ in  (\ref{RM1})) along with   intraspecific  logistic  competition among predators.

Model A  may be also viewed  as a perturbation of predator- prey model with prey taxis for which it is well known that for the class of functional responses studied in this paper prey taxis enhances stability for any $\xi>0$ (see e.g. \cite{Lee}). The same is true (see Theorem \ref{ThmstabA}) in the case of our extended model  which accounts for chemorepulsive escape of prey provided  the repulsive force is not too strong i.e. $\chi$ is less then some threshold value. Otherwise  the steady state $\bar{E}$ may loose  or gain the stability depending  on the  relation between model parameters.

In Section \ref{NumSim} various complex  space-time patterns are shown which emerge in time starting from  initial conditions in the vicinity of the staedy state $\bar{E}$. In particular we show periodic and quasi periodic patterns as well as formation of singular solutions to model A.


\section{Existence of global-in-time solutions}\label{SectionGlobalexistence}
In this section we concentrate on showing  the existence of global in time solutions to model A and model B.

\begin{theorem} \label{exist}
Suppose that hypotheses (H1) -(H3) are satisfied and $N_0\,,P_0\,, W_0\in W^{1,r}(\Omega)$, $r>n$ are non-negative functions. 

\noindent
For model A ($\chi>0, \xi>0 $) in the case of $n=1$ and for Model B ($\chi>0, \xi=0 $) for all  $n\geq 1$    there exists the  unique uniformly $L^\infty$- bounded global classical solution $(N,P,W)$  to  system (\ref{G1}) defined in $\bar{\Omega}\times [0\,,\infty)$  satisfying initial and  boundary  conditions (\ref{IC})-(\ref{BC})   such that 
\[(N,P,W)\in (C([0\,,T):W^{1,r}(\Omega))\cap 
C^{2,1}(\bar{\Omega}\times (0\,,T)))^3\quad\mbox{for any}\;\; T>0 \,. 
\]
\end{theorem}

\noindent
Before proving the theorem we state two lemmata. Consider first an auxiliary initial boundary value problem 
\begin{equation}\label{aux}
u_t+ A u +\eta u =\nabla \cdot Q +\varphi \,,\quad u(0)=u_0\in W^{1,r}(\Omega)\,, \; r>n
\end{equation}
where $\Omega$  is a regular domain, $W^{k,r}(\Omega)\,,\;k\in\{0,1,2\}$, $r\geq 1$,  is the  Sobolev space with the norm denoted by 
$\Vert\cdot\Vert_{k,r}$. For short the norm in the space $L^q(\Omega)$, $\Omega\subset \R^n$,   will be  denoted by $\Vert\cdot\Vert_q$. Notice that by the Sobolev embedding theorem 
\begin{equation}\label{Sobembed}
 W^{1,r}(\Omega)\subset L^\infty(\Omega)\quad\mbox{for}\;\; r>n\,.
\end{equation}
The operator 
\[ Au=-\Delta u\;\; \mbox{for}\;\; u\in D(A)=\{v\in W^{2,q}(\Omega) : \frac{\partial v}{\partial \nu}=0 \;\;\mbox{on}\;\; \partial \Omega\}
\] 
is a $L^q(\Omega)$-realization, $q\in(1\,,\infty)$,   of the Laplace operator with homogeneous Neumann boundary condition and 
\begin{align}\label{Q}
Q&\in X_q := C([0\,,T):(L^q(\Omega))^n)\,, \\ \label{phi}
\varphi &\in C([0\,,T):L^{q_0}(\Omega))\,.
\end{align}
The  Duhamel formula for (\ref{aux}) reads
\begin{equation} \label{Duh}
u(t)= e^{-(A+\eta)(t-\tau)}u(\tau)+ \int_\tau^te^{-(t-s)(A+\eta)}\nabla\cdot Q(s) ds + \int_\tau^t e^{-(t-s)(A+\eta)}\varphi(s) ds
\end{equation}
where $\tau\geq 0$ and $\eta>0$. We shall use the Gagliardo-Nirenberg interpolation inequality (see e.g. \cite{Henry}) which is quoted below for the reader's convenience.
\begin{proposition}\label{G-N} 
There exists a constant $C_{G-N}$ such that  for all $u\in W^{1,q}(\Omega)$  
\[\Vert u \Vert_p\leq C_{G-N} \Vert u \Vert_{1,q}^\alpha \Vert u\Vert_m^{1-\alpha}
\]
where $p\geq q\geq 1$, $p\geq m$  satisfy 
\begin{equation}\label{pq}
1\geq\alpha\geq\frac{\frac{n}{m}-\frac{n}{p}}{\frac{n}{m}+ 1 -\frac{n}{q}}\in(0,1)
\end{equation}
with sharp inequality when $m=1$ or $q=1$.
\end{proposition}

\noindent
The following consequence of this proposition  will be also helpful.
\begin{proposition}\label{additive G-N} 
For any  $u\in W^{1,2}(\Omega)$ and $\varepsilon_1>0$  
\[\int_\Omega u^2 dx\leq \varepsilon_1 \int_\Omega |\nabla u|^2dx + C_{\varepsilon_1} \left(\int_\Omega u dx \right)^2\quad \mbox{with some constant }\;\;\;C_{\varepsilon_1}>0\,.
\]
\end{proposition} 
The lemma below which will be used several times in the proof of Theorem 1 is based on well known semigroup estimates 
\begin{lemma}\label{lemma1}
Suppose that (\ref{Q})-(\ref{phi}) are satisfied with $q_0\geq 1$, $q\geq q_0$, $m\in\{0\,,1\}$ and for some $\theta \in (0\,,1)$ the  parameter $p\in [1,\infty]$ satisfies
\begin{equation}\label{p1}
\frac{m}{2}-\frac{n}{2p} +\frac{n}{2q}<\theta< 1-\frac{n}{2}\left(\frac{1}{q_0}-\frac{1}{q}\right)\,,
\end{equation} 
provided $Q\equiv 0$. Otherwise,  we assume that $p$  satisfies  (\ref{p1})  and in addition for some $\varepsilon> 0$ there holds
\begin{equation} \label{p3}
0<\theta<\frac{1}{2}-\varepsilon \,.
\end{equation} 
Then there exist constants $C_0$ and $\mu_0$ such that the solution $u\in C([0\,,T):W^{1,r}(\Omega))$ to (\ref{aux})  satisfies
\begin{eqnarray}\nonumber
\Vert u(t)\Vert_{m,p}\leq C_0((t-\tau)^{-\theta} \Vert u(\tau) \Vert_q &+&\Gamma(1-\alpha) \mu_0^{1-\alpha}\sup_{\tau\in [0\,,T)}\Vert Q(t)\Vert_{X_q} \\ \label{mp}
&+& \Gamma(1-\beta) \mu_0^{1-\beta}\sup_{t\in [\tau\,,T)}\Vert \varphi(t)\Vert_{q_0} )
\end{eqnarray}
where $\Gamma (\cdot)$ is Euler's gamma function and 
\begin{equation}\label{ab}
\alpha=\frac{1}{2}+ \theta+ \varepsilon\,,\quad \beta=\theta +\frac{n}{2}\left( \frac{1}{q_0} -\frac{1}{q}\right)\,.
\end{equation} 
\end{lemma}

\noindent
\dem
The proof is based on well known    estimates which may be found in \cite{HW} or  \cite{WWS} in a more compact form. For $u\in D(A+\eta)^\theta$ where $\eta >0$ and $\theta\in(0,1)$, $q_0\geq 1$, $q\geq q_0$, $m\in\{0\,,1\}$ and $p\in [1,\infty]$ such that 
\[\frac{m}{2}-\frac{n}{2p} < \theta +\frac{n}{2q}\]
there holds
\[ \Vert u\Vert_{m,p}\leq C_1\Vert (A+\eta I)^\theta u\Vert_q
\] 
for some constant $C_1$. Next for $u\in L^{q_0}(\Omega)$, $q\geq q_0$, there exist $C_2>0$ and $\mu_0$ such that 
\[ \Vert (A+\eta I)^\theta e^{-t(A+\eta I)}u\Vert_q\leq C_2 t^{-\theta -\frac{n}{2}(1/q_0-1/q)} e^{-\mu_0 t} \Vert u\Vert_{q_0}\,.
\]
Moreover, for any $q\in (1\,,\infty)$ and $\varepsilon$ there exist a constant $C_3>0$ such that 
\[\Vert (A+\eta I)^\theta e^{-t(A+\eta I)}\nabla \cdot u\Vert_q\leq C_3 t^{-\theta -\frac{1}{2} -\varepsilon} e^{-\mu_0 t} \Vert u\Vert_{q}\,.
\]
Thus, making use of (\ref{Duh}) and  (\ref{p1})-(\ref{p3}) we obtain for $C_0=\max\{C_1C_2\,,C_1 C_3\}$ 
\begin{align*}
&\Vert u(t)\Vert_{m,p} \leq  C_1  \Vert (A+\eta I)u(t)\Vert_q \leq C_1 (\Vert (A+\eta I)^\theta e^{-(t-\tau)(A+\eta I)}u(\tau) \Vert_q\\
& + C_2\int_\tau^t (t-s)^{-\alpha} e^{-\mu_0 (t-s)}\Vert Q(s)\Vert_{X_q} ds 
 + C_3\int_\tau^t(t-s)^{-\beta} e^{-\mu_0 (t-s)}\Vert \varphi(s)\Vert_{q} ds \\
&\leq  C_0\left( (t-\tau)^{-\theta} \Vert u_0\Vert_q +  \int_0^{\infty} \sigma^{-\alpha} e^{-\mu_0\sigma} d\sigma (\sup_{t\in [\tau\,,T_{max})}\Vert Q(t)\Vert_{X_q}) \right.\\
&+ \left.\int_\tau^{\infty} \sigma^{-\beta} e^{-\mu_0\sigma} d\sigma (\sup_{t\in [\tau\,,T_{max})}\Vert \varphi(t)\Vert_{q_0}\right) \,.
\end{align*}
where $\alpha $ and $\beta$ satisfy (\ref{ab}). Hence, using the definition of Euler's gamma function 
\[ \Gamma(a)=\frac{1}{x^{-a}} \int_0^{\infty} \sigma^{a-1} e^{-x\sigma} d\sigma\quad\mbox{for}\;\; a>0, x>0
\] 
(see e.g. \cite{Henry})  (\ref{mp}) follows. \qed  

\begin{lemma}\label{L1bound} For any solution  $(N,P,W)\in (C([0\,,T_{max}):W^{1,r}(\Omega))\cap 
C^{2,1}(\bar{\Omega}\times (0\,,T_{max})))^3$  to  system (\ref{G1}) satisfying  initial and boundary conditions (\ref{IC})-(\ref{BC}) there exist a constant $M>0$ such that 
\begin{equation}\label{M}
\sup_{t\in [0\,,T_{max})}(\Vert N(t)\Vert_{1}+\Vert P(t)\Vert_{1}+ \Vert W(t)\Vert_{1} )\leq M\,. 
\end{equation}
 \end{lemma}

\noindent
\dem Using (H1)-(H3) and the boundary conditions we obtain  after integration and summing up the  equations   that  
\begin{eqnarray*}
& &\frac{d}{dt}\left(\int_\Omega N(x,t)dx +\frac{1}{b}\int_\Omega P(x,t)dx + \frac{\delta}{2bC_g}\int_\Omega W(x,t)dx\right)\\
& & \leq -\frac{\delta}{b}\int_\Omega P(x,t)dx + \int_\Omega f(N(x,t))dx -\frac{\mu \delta}{2bC_g}\int_\Omega W(x,t)dx + 
\frac{\delta}{2bC_g}\int_\Omega g(N,P,W) dx\\
& &\leq  -\frac{\delta}{b}\int_\Omega P(x,t)dx  + \int_\Omega\left(r_1 N(x,t) dx - r_2  N(x,t)^2\right) dx +\frac{\delta}{2b}\int_\Omega P(x,t)dx \\
& &\;-\frac{\mu \delta}{2bC_g}\int_\Omega W(x,t)dx
\end{eqnarray*} 
It is easy to check that 
\[r_1 N -r_2 N^2< \frac{3r_1^2}{4r_2}- \frac{r_1}{2}N 
\]
and hence we obtain 
\begin{eqnarray*}
& & \frac{d}{dt}\left(\int_\Omega N(x,t)dx +\frac{1}{b}\int_\Omega P(x,t)dx +\frac{\delta}{2bC_G}\int_\Omega W(x,t)dx\right) \\
& &\leq-\frac{\delta}{2b}\int_\Omega P(x,t)dx -\frac{r_1}{2}\int_\Omega N(x,t)dx -\frac{\mu \delta}{2bC_G}\int_\Omega W(x,t)dx +\frac{3r_1^2}{4r_2}|\Omega|\\
& & \leq -\min\left\{\frac{\delta}{2}\,,\frac{r_1}{2}\,,\mu\right\}\left(\int_\Omega N(x,t)dx +\frac{1}{b}\int_\Omega P(x,t)dx+ \frac{\delta}{2bC_G}\int_\Omega W(x,t)dx \right) + \frac{3r_1^2}{4r_2}
|\Omega|\,.
\end{eqnarray*} 
Next, for $t\in [0,T_{max}) $ we use the inequality 
\[
\frac{dv}{dt}+ c_1 v(t) \leq c_0
\]
with $v(t)=\Vert N(t)\Vert_1 +\frac{1}{b}\Vert P(t)\Vert_1 +\frac{\delta}{2bC_g}\Vert W(t)\Vert_1$, $c_1=\min\left\{\frac{\delta}{2}\,,\frac{r_1}{2}\,,\mu\right\}$ and $c_0=\frac{3r_1^2}{4r_2}$.  Hence,
\[v(t)\leq \frac{c_0}{c_1}(1-e^{-c_1 t})+v(0)e^{-c_1 t}\leq \max\left\{v(0)\,, \frac{c_0}{c_1}\right\}
\]
and then 
\begin{eqnarray*}
\Vert N(t)\Vert_1 &+&\frac{1}{b}\Vert P(t)\Vert_1 +\frac{\delta}{2bC_G}\Vert W(t)\Vert_1 \leq \\
& &\max\left\{\Vert N_0\Vert_1 +\frac{1}{b}\Vert P_0\Vert_1 +\frac{\delta}{2bC_G}\Vert W_0\Vert_1\,,
\frac{\frac{3r_1^2}{4r_2}|\Omega|}{\min\{\delta\,,\frac{r_1}{2}\,,\mu\}}\right\}\,.
\end{eqnarray*}
whence (\ref{M}) immediately follows.\qed

\noindent
{\it Proof of Theorem \ref{exist}.}
The local in-time existence of solutions for similar problems have been considered  in many papers  therefore we present it in abbreviated form. We first notice that  in the case of Model B as well as in the case of model A upon exchange of the first and the second equation  the main  part of the quasilinear parabolic system is a normally elliptic operator with upper-triangular structure and the existence and uniqueness  of  maximal classical solution 
\[(N,P,W)\in (C([0\,,T_{max}):W^{1,r}(\Omega))\cap 
C^{2,1}(\bar{\Omega}\times (0\,,T_{max})))^3\] 
satisfying  initial and boundary conditions (\ref{IC})-(\ref{BC})
follows from  Amann's theory \cite[Theorems~14.4 \&~14.6]{Am93} (see e.g. \cite{Ahnand,JinWang,WWS} for details). Moreover in this case it is known that a uniform in time $L^\infty$-bound for the solution is enough to warrant that in fact $T_{max}=+\infty$.  The non-negativity of solutions easily follows from the maximum  principle. 

We first consider the case of model A for $n=1$. Owing to  Lemma \ref{L1bound} we may apply Lemma\ref{lemma1} to  $W$-equation with $Q\equiv 0$, $q_0=1$, $q>1$,  $ p\in (1,\infty)$ and $\theta\in (\frac{1}{2}-\frac{1}{2p}+\frac{1}{2q}, \frac{1}{2}+\frac{1}{2q})$  to obtain that  there is $\tau_0>0$ and constant $K_W(\tau_0)$  such that 
\begin{equation}\label{W1p}
\sup_{t\in [\tau_0 \,, T_{max})} \Vert W (t) \Vert_{1,p}\leq K_W(\tau_0)
\end{equation} 
Next, we turn to $N$-equation. On multiplying it by $N(\cdot,t)^{k-1} $, $k\geq 2$,  for $t\in[\tau_0 , T_{max})$ we arrive at
\begin{align*}
\frac{1}{k} \frac{d}{dt}\int_\Omega N^kdx 
&= -\frac{4(k-1)}{k^2}D_1 \int_\Omega |\nabla (N^{k/2})|^2dx +\frac{2(k-1)}{k}\chi \int _\Omega N^{k/2}\nabla (N^{k/2})\nabla W dx\\
& +\int_\Omega N^k f(N)dx - \int_\Omega N^kPF(N,P,W) dx\,.
\end{align*}
Using the H\"{o}lder inequality to the second term on the r.h.s., next (H1) and the non-negativity of solutions we obtain for 
$t\geq\tau_0$
\begin{align*}
&\frac{1}{k} \frac{d}{dt}\int_\Omega N^kdx + \frac{4(k-1)}{k^2}D_1 \int_\Omega |\nabla (N^{k/2})|^2dx \leq \\
& \leq \frac{2(k-1)}{k}\chi \left(\int _\Omega N^{k+1}dx\right)^{\frac{k}{2(k+1)}}\left(\int_\Omega |\nabla (N^{k/2})|^2 dx\right)^{1/2}\left(\int_\Omega|\nabla W|^{2(1+k)}dx\right)^{\frac{1}{2(1+k)}} \\
& + r_1\int_\Omega N^k dx - r_2\int_\Omega N^{k+1} dx\,.
\end{align*}
On account of  (\ref{W1p}) and Young inequality with $\varepsilon =\frac{2D_1}{\chi}$ we may write  
\[\frac{1}{k} \frac{d}{dt}\int_\Omega N^kdx\leq K_0 \left(\int _\Omega N^{k+1}dx\right)^{\frac{k}{k+1}}+ r_1\int_\Omega N^k dx - 
r_2\int_\Omega N^{k+1} dx.
\]
where $K_0$ depends on $k$ and $K_W(\tau_0)\,.$  
Since
\[K_0\left(\int _\Omega N^{k+1}dx\right)^{\frac{k}{k+1}}\leq \varepsilon \int _\Omega N^{k+1}dx + 
C_\varepsilon K_0^{\frac{k}{k+1}}
\]
taking $\varepsilon= \frac{r_2}{2}$ we get 
\[\frac{1}{k} \frac{d}{dt}\int_\Omega N^kdx   + \frac{r_2}{2}\leq \int_\Omega N^{k+1} dx \leq  r_1\int_\Omega N^k dx +C_\varepsilon K_0^{\frac{k}{k+1}}\,.
\]
The application of  H\"{o}lder's inequality yields
\[\left(\int _\Omega N^{k}dx\right)^{\frac{k+1}{k}}\leq K_1\int _\Omega N^{1+k}dx
\]
where $K_1$ depends on $k$ and $|\Omega|$  and finally  we obtain 
\[\frac{1}{k} \frac{d}{dt}\int_\Omega N(t)^kdx + \frac{r_2}{2K_1}\left(\int _\Omega N^{k}dx\right)^{\frac{k+1}{k}}\leq r_1\int_\Omega N^k dx +C_\varepsilon K_0^{\frac{k}{k+1}}\,.
\]
Whence, setting $\xi(t):=\int_\Omega N(\cdot, t)^kdx$ and making comparison with differential equation
\[\frac{d\xi}{dt} =-c_1\xi^\gamma +c_2\xi+c_0 \]
with $c_1=\frac{r_2k}{2K_1}\,, c_2=r_1\,, c_0=kC_\varepsilon K_0^{\frac{k}{k+1}}$ 
we have that 
\[\xi(t)\leq \max\{\xi (0)\,, \bar{\xi}\}
\]
where $\bar{\xi} $ solves the equation $-c_1\xi^\gamma +c_2\xi+c_0=0 \,.$
Thus for any $k\geq 2$ there exists $C(k)$  such that
\begin{equation}\label{Nk}
\Vert N(\cdot, t) \Vert_k\leq C(k)\quad\mbox{for}\;\; t\in(\tau_0\,,T_{max})\,.
\end{equation}
 Notice that from (\ref{W1p}) and (\ref{Nk}) it follows using H\"{o}lder's inequality that for any $q\in(1,\infty) $ there is a constant $C(q) $ such that 
\[\Vert N(\cdot, t)\nabla W(\cdot, t)\Vert_q\in C(q)\quad\mbox{for}\;\; t\in(\tau_0\,,T_{max})\,.
\]  
Now we are in a position to apply again Lemma \ref{lemma1} with $ Q=N\nabla W$,$\eta=1$, $\tau=\tau_0$,  and  $\varphi= N+ f(N)-PF(N,P,W)$. 
To this end we take  $n=1$ and $q_0=1$ because the only available estimate of $PF(N,P,W)$ comes from Lemma \ref{L1bound}. Thus  for any $p\in(1,\infty)$ we may choose $q>p$ such that for 
$\theta\in (\frac{1}{2}-\frac{1}{2p}+\frac{1}{2q}, \frac{1}{2}-\varepsilon)$ and 
$\varepsilon <\frac{1}{2} (1/p -1/q)$  conditions 
(\ref{p1}) -(\ref{p3}) are indeed satisfied and there is $\tau_1\in(\tau_0, T_{max})$ and  a constant $K_N(\tau_1,p)$ such that 
\begin{equation}\label{N1p}
\sup_{t\in(\tau_1\,,T_{max})} \Vert N(\cdot, t) \Vert_{1,p}\leq K_N(\tau_1, p)\,.
\end{equation}
Now we turn to $P$-equation. On multiplying it by $P(\cdot,t)$ for $t\in(\tau_1,T_{max})$, integrating on $\Omega$  and making use of (H2) we obtain 
\begin{equation}\label{Pt}
\frac{1}{2} \frac{d}{dt}\int_\Omega P^2dx + D_2\int_\Omega |\nabla P|^2dx +\mu \int_\Omega P^2dx
\leq \int_\Omega P\nabla N \nabla P dx + C_F\int_\Omega P^2dx\,.
\end{equation}
By Young's and H\"{o}lder's inequalities  we have
\begin{align}\nonumber
\int_\Omega P\nabla N \nabla P dx &\leq \varepsilon \int_\Omega |\nabla P|^2dx +C_\varepsilon \int_\Omega |\nabla N |^2  P^2 dx \\ \label{npn}
&\leq \varepsilon \int_\Omega |\nabla P|^2dx +C_\varepsilon \Vert P\Vert_4^2\vert \nabla N\Vert^2_4
\end{align}
and by the Gagliardo-Nirenberg inequality (Proposition \ref{G-N}) we obtain
\begin{equation}\label{npn1}
 \int_\Omega P\nabla N \nabla P dx \leq \varepsilon \int_\Omega |\nabla P|^2dx + C_\varepsilon C_{G-N} \Vert P\Vert_{1,2}\Vert P\Vert_{1} \Vert \nabla N\Vert^2_4 \,.
\end{equation} 
and yet another application of Young's inequality with $\varepsilon_0$ along with Lemma \ref{L1bound} and  (\ref{N1p}) with $p=4$ yields
\begin{equation}\label{K'}
\int_\Omega P\nabla N \nabla P dx \leq \varepsilon \int_\Omega |\nabla P|^2dx + \varepsilon_0 \left(\int_\Omega |\nabla P|^2dx + \int_\Omega P^2dx \right) + K'
\end{equation}
where $K'$ is a constant depending on $M$, $C_{G-N}$ and $K_N(\tau,4)$  stemming from (\ref{N1p}). On the other hand by  Proposition \ref{additive G-N} we have
\[ \int_\Omega P^2 dx\leq \varepsilon_1\int_\Omega |\nabla P|^2dx +C_{\varepsilon_1}\Vert P\Vert_1^2
\]
for any $\varepsilon_1>0$. Combining (\ref{K'}) with (\ref{Pt}) we arrive at 
\begin{eqnarray*}
& &\int_\Omega P\nabla N \nabla P dx + C_F\int_\Omega P^2dx \leq (\varepsilon +\varepsilon_0)\int_\Omega |\nabla P|^2dx+ (C_F+\varepsilon_0)\int_\Omega P^2dx +K'\\
& &\leq (\varepsilon +\varepsilon_0)\int_\Omega |\nabla P|^2dx+ (C_F+\varepsilon_0)\left(\varepsilon_1\int_\Omega |\nabla P|^2dx +C_{\varepsilon_1} \Vert P\Vert_1^2\right) +K'\,.
\end{eqnarray*}
Now choosing $\varepsilon=\varepsilon_0=\frac{1}{4}D_2$ and then $\varepsilon_1=\frac{D_2}{2(C_F+\varepsilon_0)}$ we obtain from (\ref{Pt})
the following differential inequality
\[ \frac{d}{dt}\int_\Omega P^2dx +2\mu \int_\Omega P^2dx \leq K''\]
where $K''$ is a positive constant depending on $C_F, M$ and $K'$.
It follows that 
\begin{equation}\label{P2}
\sup_{t\in(\tau_1, T_{max})}\Vert P(t)\Vert_2\leq \max\left\{\Vert P_0\Vert_2\,,\frac{K''}{2\mu}\right\}\,.
\end{equation}
Owing to this bound and (\ref{N1p})  we  apply Lemma \ref{lemma1} to $P$ equation with $Q=P\nabla N$, $\eta=\mu$  and $\varphi=PF(N)$ choosing $n=1$, $q=q_0=\frac{3}{2}$ and $m=0$. It is easy to check that then (\ref{p1})-(\ref{p3}) are satisfied for $p=\infty$ and $0<\varepsilon<\frac{1}{6}$. It follows that there is a constant $K_P(\tau_1)$ such that 
\begin{equation}\label{Pinf}
\sup_{t\in(\tau_1, T_{max})}\Vert P(\cdot,t)\Vert_\infty\leq K_P(\tau_1)\,.
\end{equation}
In the light of the embedding $ W^{1,p}(\Omega) \subset L^\infty(\Omega)$ for $n=1$ and any  $p>1$ as well as  (\ref{W1p}), (\ref{N1p}) and (\ref{Pinf}) we infer  that there is  a constant $K(\tau_1) $ such that 
\[\sup_{t\in(\tau_1, T_{max})}\left(\Vert N(t)\Vert_\infty+\Vert P(t)\Vert_\infty+\Vert W(t)\Vert_\infty\right)\leq K(\tau_1)\,.
\]
This is a crucial estimate which according to Amann's theory allows to deduce that  $T_{max}=\infty$ 
This statement completes the proof of the global existence for model A when $n=1$. 

\noindent
{\it Existence of global solutions to Model B.}

The existence proof for Model B ( $\xi=0$) is less complicated  since now the taxis term is absent in P-equation and in the light of \cite{Alikakos} from  the the secon equation in (\ref{G1}) and $L^1(\Omega)$-bound in  Lemma\ref{L1bound} we deduce that there is a constant $P_\infty$ such that 
\begin{equation} \label{PinfB}
\Vert P(\cdot, t)\Vert_\infty\leq P_\infty\quad\mbox{for}\;\;t\in [0,T_{max}) \,.
\end{equation} 
Since 
\[\Vert P(\cdot,t)F(N(\cdot,t),P(\cdot,t),W(\cdot,t))\Vert_\infty \leq P_\infty C_F\quad\mbox{for}\;\;t\in [0,T_{max}) 
\]
  it follows using  Lemma\ref{lemma1} with $q=q_0>1$, $p=\infty\,, \theta\in(\frac{1}{2}+\frac{n}{2q}\,,1)$ that   for any $n\geq 1$ there exists $\tau\in (0,T_{max})$ and a constant $W_\infty$ such  that 
\begin{equation} \label{WinfB}
\Vert W(\cdot, t)\Vert_{1,\infty}\leq  W_\infty \quad\mbox{for}\;\;t\in [\tau,T_{max}) \,.
\end{equation} 
Using equation (\ref{G1})  we are in a position to proceed in  essentially the same way as  in \cite[Lemma3.2]{WWS}  to conclude that there  is a constant $N_\infty$   such that
\begin{equation} \label{NinfB}
\Vert N(t)\Vert_\infty\leq N_\infty\quad\mbox{for}\;\;t\in [0,T_{max}) \,.
\end{equation} 
Due to the fact that the main part of the operator is upper-triangular,  it follows from \cite[Theorem~15.5]{Am93} that the uniform $L^\infty$bound for all components of the solution ensures the extensibility criterion for the  existence of maximal solution  to conclude that   $T_{max}=+\infty$. 
Then using  the  parabolic regularity theory for $t>0$ we infer that  in fact  $(N,P,W) $ is a  classical solution to system (\ref{G1}). Moreover, 
it follows from (\ref{WinfB}),(\ref{PinfB}), (\ref{NinfB}) and  Lemma \ref{lemma1} with $p=q_0=q=\infty$ that for some $\tau>0$ there is a constant $C_N$ such that the following inequality holds
\begin{equation}\label{r1}
\Vert N(\cdot, t)\Vert_{1,\infty}+ \Vert P(t)\Vert_{1,\infty}+ \Vert W(\cdot, t)\Vert_{1,\infty}\leq C_N\quad\mbox{for}\;\;t\in [\tau,\infty) 
\end{equation}
which completes the existence proof for model B. 
 \hfill $\Box$

\medskip 
 As we shall see in the last section some numerical simulations (see Figure \ref{blowup})  indicate that the blow-up of solutions in finite time is possible for model A in the case of space dimension $n=2$. From the view point of biological applicability of the model there arises a question of finding mechanism of a possible blow-up prevention. One way to achieve this effect is to consider the volume filling effect for prey and/or predator which was already taken into account  for prey taxis models (see e.g. \cite{Bendamane} or \cite{Tao}). The other way is to warrant that the consumption rate and the chemical production rate decrease sufficiently rapidly with the increase of predator density. This is the case when  sufficiently strong interference among predators is assumed so that  the term $PF(N,P,W) $ is bounded for all $(N,P,W)\in 
\R_+^3$. Notice that this requirement is satisfied  when the Crowley-Martin functional response (\ref{CM}) is accounted for. We shall also require   in addition  that  signal production  is   proportional to the rate of prey consumption  (c.f. (\ref{demage}) and \cite{Kats}):
\begin{equation} \label{CMB}
F= F_{C-M}(N,P)=\frac{aN}{1+bN+cP+dNP}\quad\mbox{and} \;\;g=\gamma F_{C-M}(N,P)\,.
\end{equation} 
 This situation falls into the following hypothesis;

\medskip
\noindent
(H4) $F,g:\R_+^3 \mapsto \R_+$ are  $C^2$-functions  such that there exist constants $C'_F$  and $C'_G$ such that  for some  constants $C_g>0$ 
\[ F(N,P,W)\leq \frac{C'}{P}\quad\mbox{and} \;\; g(N,P,W)\leq C'_G\quad \mbox{for all}\;\;  (N,P,W)\in 
\R_+^3\,.
\]

\begin{theorem} \label{exist2}
Suppose that hypotheses (H1)-(H4) are satisfied and $N_0\,,P_0\,, W_0\in W^{1,r}(\Omega)$, $r>n$ are non-negative functions. Then for $n\leq 3$ model A ($\chi>0, \xi>0 $)  has  the  unique uniformly $(L(\Omega)^\infty)^3$- bounded global classical solution $(N,P,W)$ defined on $\bar{\Omega}\times [0\,,\infty) $ satisfying  initial and boundary  conditions (\ref{IC})-(\ref{BC})   such that 
\[(N,P,W)\in (C([0\,,T):W^{1,r}(\Omega))\cap 
C^{2,1}(\bar{\Omega}\times (0\,,T)))^3\quad\mbox{for any}\;\; T>0 \,. 
\]
\end{theorem}
\dem 
We first observe that using  Lemma\ref{lemma1} with $q=q_0>1$, $p=\infty\,, \theta\in(\frac{1}{2}+\frac{n}{2q}\,,1)$ that   for any $n\geq 1$ there exists $\tau\in (0,T_{max})$ and a constant $W_\infty$ such  that 
\begin{equation} \label{Winf}
\Vert W(\cdot, t)\Vert_{1,\infty}\leq  W_\infty \quad\mbox{for}\;\;t\in [\tau,T_{max}) \,.
\end{equation} 
Owing to this estimate we deduce in  the same way as  in \cite[Lemma3.2]{WWS}  to conclude that there  is a constant $N_\infty$   such that
\begin{equation} \label{Ninf}
\Vert N(t)\Vert_\infty\leq N_\infty\quad\mbox{for}\;\;t\in [0,T_{max}) \,.
\end{equation} 
Next by Lemma \ref{lemma1} we obtain that $N\nabla W$ is bounded in  $L(\Omega)^q$ for any $q>n$ . This enables to prove using again Lemma\ref{lemma1}  that for some $\tau_1\in (0, T_{max})$ and any $p>1$ 
\begin{equation}\label{N1pp}
\sup_{t\in(\tau_1\,,T_{max})} \Vert N(\cdot, t) \Vert_{1,p}\leq K_N(\tau_1, p)\,.
\end{equation}
Similarly to  the proof of Theorem \ref{exist} we next find an $L^2(\Omega)$-bound on $P$ using the Gagliardo-Nirenberg inequality in a suitable form . The key point is to find estimate in  (\ref{npn}) for $\Vert P\Vert_4^2$.
To this end one applies the Gagliardo-Nirenberg  inequality from Proposition\ref{G-N} which leads to the restriction for the space dimension as for $p=4, q=2$ and $m=1$ we obtain from (\ref{pq}) that $n<4$. For the case of $n=3$ we then find $\alpha=\frac{9}{10}$ (for  $n=2$ there is $\alpha=\frac{3}{4}$) and in consequence we get  
\[\Vert P\Vert_4^2\leq \Vert P\Vert_{1,2}^{\frac{9}{5}}\Vert P\Vert_1^{\frac{1}{5}}
\]
and then by  the Young inequality we arrive at  (\ref{K'}) and the remaining part of the proof is the same as that of Theorem \ref{exist}.  \hfill $\Box$
 %
%


\section{Model B- linear stability and Hopf bifurcation.}\label{stabB}
From now on for simplicity  we assume in model B that 
\begin{equation}\label{FGwB}
F=F(N,P)\,,\quad g=\gamma P
\end{equation}
and as a starting point we consider  the following classical Gause-type prey-predator model 
\begin{align}
\label{G111} \left\{
\begin{aligned}
\frac{d N}{d t}&=   f(N) - F(N,P)P,\\
\frac{dP}{d t}&= -\delta P+ bF(N,P)P,
\end{aligned}
\right.
\end{align}
in which  the functional response $F$  satisfies the following  natural conditions 
\begin{equation}  \label{parF}
\frac{\partial F}{\partial N}(N,P)>0\,, \quad \frac{\partial F}{\partial P}(N,P)<0\,. 
\end{equation} 
Notice that they  are satisfied  by  the  Holling functional responses (\ref{FH}), Bedington-DeAngelis response (\ref{FB})as well as Crowley- Martin response (\ref{CM}). 
We shall consider the case when there  exists a  coexistence steady state $\bar{E}_0=(\bar{N},\bar{P})$ , $F(\bar{N}\,,\bar{P})=\frac{\delta}{b}$ which is linearly stable. In this case $f'(\bar{N})<0$ and it is easy to check that the coefficients of the  linearization  matrix  $J(\bar{E}_0)=[a_{i,j}]_{i,j=1...2}$ satisfy
\begin{equation} \label{aij}
a_{11}<0\,,\quad a_{12}<0\,,\quad a_{21}>0\,,\quad a_{22}\leq 0\,.
\end{equation}
We note that for the Holling functional responses  $a_{22}=0$  while it  is negative for both Beddignton-DeAngelis and Crowley-Martin responses. 
The linear stability of $\bar{E}_0$  then follows  from  
\[ tr J(\bar{E})= a_{11}+a_{22}<0\,,\quad det J(\bar{E})=a_{11}a_{22} -a_{12}a_{21}>0\,.\] 
Now we are in a position to consider model B (\ref{modelB}) for which  the coexistence steady is of form 
\begin{equation}\label{Esteady}
\bar{E}=(\bar{N},\bar{P},\bar{W})\quad\mbox{where}\quad \bar{W}= \frac{\mu}{\gamma}\bar{P}\,.
\end{equation}
 The linearization  of model B (\ref{modelB})leads to  the following system
\begin{align}\label{aa}
\begin{pmatrix}
\phi_t\\
\varphi_t\\
\eta _t\\
\end{pmatrix}=\begin{pmatrix}
D_1\Delta+ a_{11} & a_{12} &\chi \bar{N}\Delta\\
a_{21} & D_2\Delta+a_{22} & 0\\
0 & a_{32} & D_3\Delta+a_{33}
\end{pmatrix}\begin{pmatrix}
\phi\\
\varphi\\
\eta\\
\end{pmatrix}
\end{align}
where $(\phi, \varphi, \eta)\approx(N-\bar{N}, P-\bar{P}, W-\bar{W})$ and $ a_{ij},\ i,j=1,2,3$ are corresponding partial derivatives of the reaction terms  with respect to $N$, $P$ and $W$. Note that in addition to (\ref{aij})  we have  
\begin{equation} \label{aij1}
 a_{32}>0,\ a_{33}<0.
\end{equation}
 The stability matrix associated with (\ref{aa}) is following
\begin{align}\label{smat1}
M_j=\begin{pmatrix}
-D_1 h_j+a_{11} & a_{12} & -\chi \bar{N} h_j\\
a_{21} & -D_2 h_j+a_{22} & 0\\
0 & a_{32} & -D_3 h_j+a_{33}
\end{pmatrix}.
\end{align}
Here $\{h_j\}_{j=0}^{\infty}$ denotes the  eigenvalues of the Laplace operator $-\Delta$ with homogeneous Neumann boundary condition in smooth domain $\Omega$. The dispersal equation of stability matrix (\ref{smat1}) is following
\begin{align}\label{char1}
\lambda^3+\rho_j^{(1)} \lambda^2+ \rho_j^{(2)}\lambda +\rho_j^{(3)}(\chi) =0
\end{align}
where
\begin{align}
&\rho_j^{(1)}=-\mbox{tr}M_j=- (a_{11}+a_{22}+ a_{33}) +(D_1+ D_2 + D_3)h_j\,,\label{ro1}\\
&\quad:=\alpha_0+\alpha_1 h_j,\nonumber\\ 
&\rho_j^{(2)}=a_{11}a_{22} - a_{12}a_{21} + a_{11}a_{33} + a_{22}a_{33}\\ 
& \quad\quad\; + h_j(- a_{22}D_1 - a_{33}D_1 - a_{11}d_2 - a_{22}D_3- a_{11}D_3 - a_{33}D_2)\nonumber\\
& \quad\quad\;+h_j^2(D_1D_2+ D_1D_3 + D_2D_3) \nonumber\\
&\quad\quad\; := \beta_0+\beta_1 h_j+ \beta_2h_j^2,\label{ro2}\\ \nonumber 
&\rho_j^{(3)}(\chi)=-\mbox{det}M_j= - a_{11}a_{22}a_{33} + a_{12}a_{21}a_{33} \\
&\quad\quad\; +h_j( a_{22}a_{33}D_1+ a_{11}a_{22}D_3 - a_{12}a_{21}D_3+ a_{11}a_{33}d_2)\nonumber\\
&\quad\quad\;+h_j^2(- a_{22}D_1D_3 - a_{33}D_1D_2 - a_{11}D_2D_3)+ D_1D_2D_3h_j^3+ \chi a_{21}a_{32} \bar{N}h_j ,\nonumber\\
&\quad\quad\quad= (\gamma_0+ \gamma_1h_j +\gamma_2 h_j^2 + \gamma_3h_j^3)+ \chi(\gamma_4  h_j):=\rho_j^{(3,1)}+\chi \rho_j^{(3,2)}\label{ro3}
\end{align}
where we denoted  $\rho_j^{(3)}(\chi)=\rho_j^{(3,1)}- \chi\rho_j^{(3,2)}\,.$
It can be checked using (\ref{aij}) and (\ref{aij1}) that all coefficients $\alpha_j\,,\beta_j\,,\gamma_j $ are positive. 
The linear operator in (\ref{aa}) is sectorial as it may be  viewed as a perturbation of a sectorial operator given by the main part of the system in divergence form by a bounded operator given by the matrix $J(\bar{E})=[a_{i,j}]_{i,j=1...3}$. This observation along with the fact that $\rho_j^{(3)}(\chi)=-\det M_j>\gamma_0>0 $ leads to the conclusion that the spectrum of the linearization is contained in some cone separated from the origin of the coordinate system in the complex plane. Therefore the steady state $\bar{E}$ is linearly stable if and only if for each $j\geq 0$ matrices $M_j$ have eigenvalues with negative real parts which according to the Routh-Hurtwitz stability criterion is equivalent to the conditions    
\begin{align}
&\rho^{(1)}_j>0,\ \rho_j^{(3)}>0,\label{local1B} \\ \label{local2B}
&\mbox{and}\quad Q_j:=\rho_j^{(1)}\rho_j^{(2)}-\rho_j^{(3)}(\chi)=\rho_j^{(1)} \rho_j^{(2)}-\rho_j^{(3,1)}- \chi\rho_j^{(3,2)}>0~~~~~~~\text{for all}\ j\geq 0\,.
\end{align}
For the ODE case which corresponds to $h_0=0$ the steady state $\bar{E}$  is stable since
\[\rho^{(1)}_0=\alpha_0 \,, \quad \rho^{(2)}_0=\beta_0>0\,,\quad Q_0=\alpha_0\beta_0-\gamma_0 >0\,.
\]
While for the reaction diffusion system with $\chi=0$ the stability condition (\ref{local1B}) is obviously satisfied and (\ref{local2B}) may be rewritten in the following form 
\begin{align}
&\rho_j^{(1)} \rho_j^{(2)}-\rho_j^{(3,1)}=(\alpha_0+ \alpha_1h_j)(\beta_0+\beta_1h_j+\beta_2h_j^2)-(\gamma_0+\gamma_1 h_j+\gamma_2 h_j^2 +\gamma_3h_j^3)\label{ro31}\\
&=(\alpha_0\beta_0-\gamma_0) + (\alpha_1\beta_0+\alpha_0\beta_1-\gamma_1) h_j +
(\alpha_0\beta_2+\alpha_1\beta_1-\gamma_2 )h_j^2 +(\alpha_1\beta_2 -\gamma_3)h_j^3 \nonumber\\ \nonumber
&:=\Psi(h_j)\,.
\end{align}
Again it is straightforward to  check that all coefficients of the third order polynomial $\Psi$ are positive, so the diffusivity does not impact the local stability of the steady state (an observation made already in earlier works, see e.g. \cite{Lee,WaWuShi,WWS})  and only taxis may lead to the instability. Indeed, 
now we are in a position to  find a $\chi-$dependent  stability condition for the steady state $\bar{E}$ in model B. To this end  consider 
\begin{equation}\label{psi}
\tilde{\Psi}(h_j)=\frac{\rho_j^{(1)} \rho_j^{(2)}-\rho_j^{(3,1)}}{\rho_j^{(3,2)}}=\frac{\Psi(h_j)}{\gamma_4 h_j}\,.
\end{equation} 
Notice that   $\gamma_4=a_{21}a_{32} \bar{N}>0$ for all $j\in \mathbb{N}_{+}$.
Since the coefficients of the polynomial $\Psi$ are positive we infer that 
\[\lim_{x\rightarrow 0^+}\tilde{\Psi}(x)=\lim_{x\rightarrow +\infty}\tilde{\Psi}(x)=+\infty\,.
\]
and computing the second derivative of $\tilde{\Psi}$ we deduce that it is  a strictly convex function.
Hence,  there exists  $\chi^H>0$ such that
\begin{align}\label{chih}
 \chi^H=\min_{j\in \mathbb{N}_{+}} \Big\{\frac{\rho_j^{(1)} \rho_j^{(2)}-\rho_j^{(3,1)}}{\rho_j^{(3,2)}}\Big\}
 \end{align}
and the steady state $\bar{E}$ is stable if  $\chi<\chi^H$ . 
If 
\begin{equation} \label{jnek}
\tilde{\Psi}(h_j)\neq \tilde{\Psi}(h_k)\quad \mbox{for}\quad j\neq k
\end{equation}
then of course the minimum is attained for a singe $j=j_0$. We thus proved the following theorem 
\begin{theorem} \label{ThmstabB}
Under assumptions (\ref{FGwB})  and (\ref{parF}) the constant steady state $\bar{E}$ in model B is locally asymptotically stable if 
$\chi<\chi^H$ defined in (\ref{chih}).
\end{theorem}
\begin{remark}
The repulsive chemotaxis described in model B may be viewed as indirect predator taxis as described in \cite{WWS}. It is worth underlining that 
contrary to our case a direct predator taxis studied in \cite{WWS} does not affect the stability of the constant steady state for the  Rosenzweig-MacArthur type model.
\end{remark} 
The steady state is unstable when condition (\ref{local2B}) fails since (\ref{local1B}) is always satisfied.  It is worth underlining that for any $\chi >0$ and $j\geq 0$,   $\rho_j^{(3)}(\chi)=-\det M_j>0$ , so,  all eigenvalues of $M_j$ are non-zero when the steady state $\bar{E}$ loses stability at $\chi=\chi^H$. Hence,  any static bifurcation of the steady state  is precluded in this case and  only Hopf's bifurcation may occur which is a subject of the following theorem. To this end the chemotactic sensitivity $\chi$ is considered as the bifurcation parameter. 
Next  we discuss emergence of Hopf-bifurcation for model B (\ref{ModelB1}) at coexistence steady state $\bar{E}$ which is stated in the following theorem. 
\begin{theorem}\label{hopfthm}
For model B with (\ref{parF}) suppose that (\ref{jnek}) holds. Then for $\chi=\chi^H$ defined in (\ref{chih}) the  Hopf-bifurcation appears and in the vicinity of $\bar{E}$ there exist a one parameter family of periodic solutions.
\end{theorem} 
 \begin{proof}
 To show the  occurence of Hopf bifurcation for the quasiliner parabolic system we use \cite{Am91}  and follow approach in Theorem 5.2 from \cite{wangq}. From the  stability analysis and assumption (\ref{jnek}) we have  that
\begin{itemize}
 \item[(i)] $\rho_j^{(1)}=-\mbox{tr}M_j>0,\rho_j^{(2)}>0,\ \rho_j^{(3)}(\chi)=-\mbox{det}M_j>0$ for all  $j\geq 0$ and $\chi>0$,
 \item[(ii)] $\rho_j^{(1)}\rho_j^{(2)}=\rho_j^{(3)}(\chi^H)$ for some $j=j_0$\,.
 \end{itemize}
It follows that the characteristic polynomial corresponding to $M_{j_0}$ has  real negative root $\lambda_{1}^H=-\rho_{j_0}^{(1)}$ and a pair of purely imaginary roots $\lambda_2^H\,,\lambda_3^H=\pm i\sqrt{\rho_{j_0}^{(2)}}:=\pm i\tau_0>0$. Now let us suppose that $\lambda_1(\chi)$ and $\lambda_2(\chi)\,,\lambda_3(\chi)=\sigma(\chi)\pm i\tau (\chi)$ are the  unique  eigenvalues in the neighbourhood of the bifurcation threshold $\chi^H$, where $\lambda,\ \sigma,\tau$ are smooth functions of $\chi$ satisfying $\lambda_1(\chi^H)=\lambda_{1}^H$ as well as  $\sigma(\chi^H)=0$ with  $\tau(\chi^H):=\tau_0>0$. Substituting eigenvalues $\lambda_1(\chi)$ and $\lambda_2(\chi)\,,\lambda_3(\chi)$ into the characteristic equation of stability matrix $M_{j_0}$ and equating the corresponding coefficients we find
\begin{align}\label{hopf2}
\left\{
\begin{aligned}
-\rho_{j_0}^{(1)}&=2\sigma(\chi)+\lambda_1 (\chi),\\
\rho_{j_0}^{(2)}&=\sigma^2(\chi)+\tau^2(\chi)+2\sigma(\chi)\lambda_1(\chi),\\
-\rho_{j_0}^{(3)}(\chi)&=\sigma^2(\chi)+\tau^2(\chi)\lambda_1(\chi).
\end{aligned}
\right.
\end{align}
For the sake of simplicity  we denote $\displaystyle \sigma'(\chi)=\frac{d\sigma(\chi)}{d\chi}$ and differentiating (\ref{hopf2}) with respect to bifurcation parameter $\chi$, we obtain using (\ref{local2B}) 
\begin{align}
&2\sigma'(\chi)+\lambda_1'(\chi)=0, \label{h1}\\
&2\sigma(\chi)\sigma'(\chi)+2\tau(\chi)\tau'(\chi)+2\sigma'(\chi)\lambda_1(\chi)+2\sigma(\chi)\lambda_1'(\chi)=0,\label{h2}\\ \label{h3}
&2\sigma(\chi)\sigma'(\chi)+2\tau(\chi)\tau'(\chi)\lambda_1(\chi)+\sigma^2(\chi)+\tau^2(\chi)\lambda_1'(\chi)=\rho_{j_0}^{(3,2)}\,.
\end{align}
Evaluating the above functions  at $\chi=\chi^H$  we obtain from (\ref{h1}) 
\begin{equation}\label{h33}
\sigma'(\chi^H)=-\frac{1}{2}\lambda_1'(\chi^H)
\end{equation}
and reminding    that $\sigma(\chi^H)=0$  it follows  from (\ref{h2})) and (\ref{h33}) that 
\begin{align*}
0&=2\tau_0\tau'(\chi^H) + \sigma'(\chi^H)\lambda_1^H\,,\\
\rho_{j_0}^{(3,2)}&=2\tau_0\tau'(\chi^H)\lambda_1^H+\tau_0^2\lambda_1'(\chi^H)\,.
\end{align*}
Now by  solving this system and  making use of (\ref{h33}) and equality  $\lambda_1(\chi^H)=-\rho_{j_0}^{(1)}=\mbox{tr}M_{j_0}$ we finally get 
\begin{align}\nonumber
\lambda_1'(\chi^H)&=-\frac{\rho_{j_0}^{(3,2)}}{\rho_{j_0}^1+\tau_0^2}<0,\\ \label{transv1}
\sigma'(\chi^H)&>0\,.
\end{align}
This verifies the transversality condition required for the occurrence of Hopf-bifurcation at $\chi=\chi^H$ which completes the proof.
  \end{proof}
	\begin{remark}
More detailed description of the periodic solution emerging according to the  Hopf bifurcation may be find in \cite[Theorem 3.1]{wangq} or  	\cite{KeWang}.
\end{remark}
\subsection{Model B - extended Rosenzweig-MacArthur model. }\label{sec4}
In this section, we consider in details model B in the case when it  may be viewed as  an extension of the  Rosenzweig-MacArthur model (\ref{RM1}). We note that a thorough analysis of the diffusive Rosenzweig-MacArthur model  including stability analysis and bifurcations was investigated in many papers and we refer in particular  to \cite{Yi} and references given there. The extended Rosenzweig-MacArthur model  will be investigated numerically in Section \ref{NumSim} where we shall exhibit spatio-temporal patterns which emerge due to chemorepulsion for $\chi$ big enough in the  regime of parameters such that   the constant steady state $\bar{E}$ is globally stable when pointwise ODE   or reaction-diffusion models are taken into account. Making the following substitutions: 
\begin{align*}
\tilde{P}&=\frac{N}{K},\ \tilde{P}=\frac{P}{K},\ \tilde{W}=\frac{W}{K},\ \tilde{t}=\frac{D_1T}{L^2},\ \tilde{x}=\frac{x}{L}, \ \tilde{\chi}=\frac{\chi K}{D_1},\   \tilde{r}=\frac{rL^2}{D_1},  \tilde{a}=\frac{aKL^2}{D_1},\\
\tilde{\delta}&=\frac{\delta L^2}{D_1}, \ \tilde{\beta}=aKT_h,\ \ c=\frac{abKL^2}{D_1},\ d_p=\frac{D_2}{D_1}, \ d_w=\frac{D_3}{D_1}, \ \tilde{\mu}=\frac{\mu L^2}{D_1},\ \tilde{\gamma}=\frac{\gamma L^2}{D_1},
\end{align*}
and finally removing tilde we get the following  non-dimensional  version of the extended Rosenzweig-MacArthur model in the frame of model B 
\begin{align}\label{ModelB1}
\left\{
\begin{aligned}
N_t&= \Delta N +r N\big(1-N \big)+\nabla \cdot (\chi N \nabla W)- \frac{a NP}{(1+ \beta N)}\,,\\
P_t&=d_p \Delta P-\delta P+\frac{c NP}{(1+ \beta N)}\,, \\
W_t&=d_w \Delta W+\gamma P-\mu W\,,
\end{aligned}
\right.
\end{align}
with initial and boundary conditions (\ref{IC})-(\ref{BC}). It is easy to check that $\bar{E}=(\bar{N}, \bar{P}, \bar{W})$ is a positive constant steady state of the system (\ref{ModelB1}) where 
\begin{align}
\bar{N}=\frac{\delta}{c-\delta \beta},\ \bar{P}=\frac{cr(c-\delta\beta-\delta)}{a(c-\delta \beta)^2},\ \bar{W}=\frac{cr\gamma(c-\delta\beta-\delta)}{\mu a(c-\delta \beta)^2}\big),
\end{align}
where $c>\delta \beta+\delta$. From Section \ref{stabB}, we infer that the constant steady state $\bar{E}=(\bar{N}, \bar{P}, \bar{W})$ of model (\ref{ModelB1}) is locally stable for small $\chi$ up to $\chi=\chi^H$ when  it loses stability and the  Hopf-bifurcation emerges.

\subsection{Model B - global stability in the case of intraspecific competition  of predators.} \label{SecGlobal}
In this subsection, we  investigate the global stability of the constant steady state to model B  for $\chi$ is subcritical. It turns out that  well known Lyapunov functions which are suitable for the ODE part of the model are not useful neither for model B nor for model A because of lack of sufficiently strong dissipation. The latter effect may be incorporated to the model by assuming intraspecific competition among predators which may  be modeled by adding the logistic term $-\eta P^2$, $\eta >0$,  to  $P$-equation.  We next consider model B with the  Beddington-DeAngelis functional response (\ref{FB}) as the case of Holling II functional response requires  only obvious modifications resulting from setting $\alpha=0$ in the Beddington-DeAngelis functional response:
\begin{align}\label{ModelB2}
\left\{
\begin{aligned}
N_t&= \Delta N +r N\big(1-N \big)+\nabla \cdot (\chi N \nabla W)- \frac{a NP}{(1+ \beta N+\alpha P)}\,,\\
P_t&=d_p \Delta P-\delta P-\eta P^2+\frac{c NP}{(1+ \beta N+\alpha P)}\,, \\
W_t&=d_w \Delta W-\mu W+\gamma P\,,
\end{aligned}
\right.
\end{align}
with the initial and boundary conditions defined in (\ref{IC})-(\ref{BC}). Existence of global solutions  to model (\ref{ModelB2}) along with estimates (\ref{r1}) may be proved in the same way as in the case of $\eta=0$ (cf. Theorem \ref{exist}). 

From now on we assume that $E^\star=(N^\star,P^\star, W^\star)$, $N^\star,P^\star, W^\star >0$, is the  unique constant steady state to model (\ref{ModelB2}) such that 
\begin{equation}\label{BDeq}
r(1-N^\star)= \frac{aP^\star}{1+\beta N^\star +\alpha P^\star}\,,\quad \delta=-\eta P^{\star} +\frac{cN^\star}{1+\beta N^\star +\alpha P^\star}\,,\quad W^\star=\frac{\gamma}{\mu}P^\star\,.
\end{equation}
Indeed, to justify this assumption for the Beddington-DeAngelis model ($\alpha >0$)  with the help of symbolic MATLAB computation from the first two equation one  obtains a third order polynomial with respect to $P^*$ while for the Rosenzweig-McArthur model ($\alpha=0$) by a straightforward  computation one obtains a third order polynomial  with respect to $N^*$. Then by  the Decartes rule of signs applied to the polynomials it follows that there is only one constant steady state satisfying (\ref{BDeq}) provided:
\begin{itemize}
\item $r\alpha> 2a\,,\; \beta\in (0,1)\,, \;\delta\in(\frac{c}{2}\,,\frac{c}{\beta+1})$ for $\alpha >0$, 
\item $\beta<\min\{\frac{1}{2}\,, \frac{c}{\delta}\}$  for $\alpha=0$. 
\end{itemize} 
The following theorem assures the stability of the constant steady state $\bar{E}$ . 
\begin{theorem} \label{GlobStab}
If  $\beta(1-N^\star)<1$ and $\chi<\chi_0$ where $\left(\chi_0\right)^2=\frac{d_w\eta \mu a(1+\beta N^*)}{N^*c\gamma (1+\alpha P^*)}$ then  the unique coexistence steady state $E^\star$ to system  (\ref{ModelB2}) is globally asymptotically stable i.e. any solution (N,P,W) to  (\ref{ModelB2}) with $N_0\,,P_0\,,W_0 >0$ in $\bar{\Omega}$ satisfies 
\[
\lim_{t\rightarrow\infty} \max\left\{\Vert N(t)-N^\star \Vert_\infty\,, \Vert P(t)-P^\star \Vert_\infty\,, 
\Vert W(t)-W^\star \Vert_\infty \right\}=0\,.
\]
\end{theorem}

\begin{proof}
Let us consider following Lyapunov function
\begin{align}
\mathcal{L}(N\,,P\,,W)=\int_{\Omega} \big(N-N^\star-N^\star\log\frac{N}{N^\star}\big) &+C_1\int_{\Omega} \big(P-P^\star-P^\star\log\frac{P}{P^\star}\big) \nonumber\\
&+ \frac{C_2 }{2}\int_{\Omega} (W-\bar{W})^2 \label{Lyap1}
\end{align}
with positive constants $C_1$ and $C_2$ which will be specified later on. We note that a similar function was used in \cite{BaiWinkler} to analyze the stability of equilibrium for a competition system with chemotaxis. Notice that by Taylor's expansion for $ z, z^*> 0 $ there exists$\zeta\in (z, z^*)$ such that
\begin{equation}\label{z}
 z^*-z- z^*(\ln z-\ln z^*)=\frac{1}{\zeta^2}(z-z^*)^2\,.
\end{equation} 
Hence, we deduce that   $\mathcal{L}(N\,,P\,,W)\geq 0$. Differentiating (\ref{Lyap1}), we get
\begin{align*}
&\frac{d}{dt}\mathcal{L}(N,P,W)=\int_{\Omega} \Big(1-\frac{N^\star}{N}\Big)N_t+C_1\int_{\Omega} \Big(1-\frac{P^\star}{P}\Big)P_t+C_2\int_{\Omega}(W-W^\star)W_t\\
&= \int_{\Omega} \Big(1-\frac{N^\star}{N}\Big) \Big(\Delta N +\nabla \cdot (\chi N \nabla W)+r N\big(1-N \big)- \frac{a NP}{(1+ \beta N+\alpha P)}\Big)\\
&+C_1\int_{\Omega} \Big(1-\frac{P^\star}{P}\Big)\Big(d_p \Delta P-\delta P-\eta P^2+\frac{c NP}{(1+ \beta N+\alpha P)}\Big)\\
&+ C_2 \int_{\Omega}(W-W^\star)(d_w \Delta W-\mu W+\gamma P)\\
&= \overbrace{\int_{\Omega} (N-N^\star)\Big(r(1-N)- \frac{aP}{(1+ \beta N+\alpha P)}\Big)}^{I_1^N}\\
& \overbrace{\int_{\Omega}(P-P^\star)\Big(-\delta -\eta P+\frac{c N}{(1+ \beta N+\alpha P)}\Big)}^{I_1^P} +\overbrace{\int_{\Omega} (W-W^\star)(\gamma P-\mu W)}^{I_1^W}\\
& \overbrace{-N^\star\int_{\Omega} \Big\vert\frac{\nabla N}{N}\Big\vert^2+ \chi N^\star \int_{\Omega}\frac{\nabla N \cdot \nabla W}{N} -d_p P^\star C_1\int_{\Omega} \Big\vert\frac{\nabla P}{P}\Big\vert^2 -d_w C_2\int_{\Omega}\Big\vert \nabla W\Big\vert^2}^{I_2}\\
& I_1^N+I_1^P+I_1^W +I_2\,.
\end{align*}
Now we find bounds on $I_1^N\,,I_1^P\,,I_1^W $;
\begin{align*}
\begin{aligned}
I_1^N&=\int_{\Omega} (N-N^\star)\Big(r(1-N)- \frac{a P}{(1+ \beta N+\alpha P)}\Big)\\
&=\int_{\Omega} (N-N^\star) \Big(-r(N-N^\star)-\frac{aP}{(1+ \beta N+\alpha P)}+\frac{aP^\star}{(1+ \beta N^\star+\alpha P^\star)}\Big),\\
&=-r\int_{\Omega} (N-N^\star)^2+\int_{\Omega}\frac{(N-N^\star)[a\beta P^\star(N-N^\star)-a(1+\beta N^\star)(P-P^\star)]}{(1+ \beta N+\alpha P)(1+ \beta N^\star+\alpha P^\star)},\\
&=\int_{\Omega}\Big(-r+\frac{a\beta P^\star}{(1+ \beta N+\alpha P)(1+ \beta N^\star+\alpha P^\star)}\Big)(N-N^\star)^2\\
&\quad -\int_{\Omega}\frac{a(1+\beta N^\star)(N-N^\star)(P-P^\star)}{(1+ \beta N+\alpha P)(1+ \beta N^\star+\alpha P^\star)}\,.
\end{aligned}
\end{align*}
If $\beta(1-N^\star)<1$ then the coefficient in front of $\int_{\Omega} (N-N^\star)^2$ can be rewritten as
\begin{align*}
 & -r+\frac{a\beta P^\star}{(1+ \beta N+\alpha P)(1+ \beta N^\star+\alpha P^\star)}\\
 &\leq -r+\frac{a\beta P^\star}{(1+ \beta N^\star+\alpha P^\star)},\\
 &=r[-1+\beta(1-N^\star)]:=-\theta<0.
\end{align*}
Thus reaction terms associated with $N$ can be estimated as follows
\begin{align}
I_1^N= &\int_{\Omega} (N-N^\star)\Big(r(1-N)- \frac{a NP}{(1+ \beta N+\alpha P)}\Big) 
\leq -\theta \int_{\Omega} (N-N^\star)^2 \nonumber \\
&-\int_{\Omega}\frac{a(1+\beta N^\star)(N-N^\star)(P-P^\star)}{(1+ \beta N+\alpha P)(1+ \beta N^\star+\alpha P^\star)}\,.
\label{n1}
\end{align}
Making  use of (\ref{BDeq}) we can handle the reaction terms associated with $P$ 

\begin{eqnarray}
&I_1^P & = C_1\int_{\Omega}(P-P^\star)\Big(-\delta - \eta P +\frac{c N}{(1+ \beta N+\alpha P)}\Big) \nonumber\\
& &=-C_1\int_{\Omega}\eta(P-P^\star) +C_1\int_{\Omega} (P-P^\star)\Big(\frac{c N}{(1+ \beta N+\alpha P)}-\frac{c N^\star}{(1+ \beta N^\star+\alpha P^\star)}\Big) \nonumber\\
& &=-C_1\int_{\Omega}\eta(P-P^\star)+ C_1 \int_{\Omega}  \frac{c[(1+\alpha P^\star)(N-N^\star)(P-P^\star)-\alpha  N^\star(P-P^\star)^2]}{(1+ \beta N+\alpha P)(1+ \beta N^\star+\alpha P^\star)} \nonumber\\
& &\leq -C_1\int_{\Omega}\eta(P-P^\star)+ C_1 \int_{\Omega}  \frac{c[(1+\alpha P^\star)(N-N^\star)(P-P^\star)}{(1+ \beta N+\alpha P)(1+ \beta N^\star+\alpha P^\star)}\,.
\label{p01}
\end{eqnarray}
Using again (\ref{BDeq})  for the terms associated  with $W$ we obtain
\begin{align*}
\begin{aligned}
I_1^W=C_2 \int_{\Omega} (W-W^\star)(\gamma P-\mu W)=-\mu C_2\int_{\Omega} (W-W^\star)^2+\gamma C_2\int_{\Omega} (P-P^\star) (W-W^\star)\,.
\end{aligned}
\end{align*}
By  Young's inequality we obtain 
\begin{align}\label{w1}
I_1^W=C_2\int_{\Omega} (W-W^\star)(\gamma P-\mu W)\leq \frac{2\gamma C_2}{\mu}\int_{\Omega} (P-P^\star)^2- \frac{\mu C_2}{2}\int_{\Omega} (W-W^\star)^2\,.
\end{align}
Selecting  first $C_1=\frac{a(1+\beta N^\star)}{c(1+\alpha P^\star)}$ to cancel terms in (\ref{n1}) and (\ref{p01})  then setting $C_2=\frac{\eta C_1\mu}{4\gamma} $ and using (\ref{w1}) we arrive at  
\begin{align}\label{A1}
I_1^N+I_1^P+I_1^W\leq-\theta \int_{\Omega} (N-N^\star)^2-\frac{\eta}{2} \int_{\Omega} (P-P^\star)^2-\frac{\mu C_2}{2}\int_{\Omega}(W-W^\star)^2\leq 0 \,.
\end{align}
Next we  turn to 
\begin{align}\label{deriv1}
I_2&=-N^\star\int_{\Omega} \Big\vert\frac{\nabla N}{N}\Big\vert^2+ \chi N^\star \int_{\Omega}\frac{\nabla N \cdot \nabla W}{N} -d_p P^\star C_1\int_{\Omega} \Big\vert\frac{\nabla P}{P}\Big\vert^2-d_wC_2 \int_{\Omega}\Big\vert \nabla W\Big\vert^2
\end{align}
where  using  Young's inequality to the second term we obtain that 
\begin{align*}
\chi N^\star \int_{\Omega}\frac{\nabla N \cdot \nabla W}{N}\leq \frac{\chi^2 {N^\star}^2}{4d_wC_2}\int_{\Omega} \Big\vert\frac{\nabla N}{N}\Big\vert^2+d_w C_2 \int_{\Omega} \Big\vert\nabla W \Big\vert^2\,.
\end{align*}
The following bound is  obtained from (\ref{deriv1})  after cancellation of the last term in (\ref{deriv1}) 
\begin{align}\label{A2}
I_2\leq-N^\star\big(1- \frac{\chi^2 {N^\star}}{4d_wC_2} \big)\int_{\Omega} \Big\vert\frac{\nabla N}{N}\Big\vert^2 - C_1 d_p P^\star\int_{\Omega} \Big\vert\frac{\nabla P}{P}\Big\vert^2 \leq 0
\end{align}
which holds for  $\chi\leq \chi_0$  where  
$\displaystyle \left(\chi_0\right)^2=\frac{4d_w C_2}{N^\star}=\frac{d_w\eta \mu a(1+\beta N^*)}{N^*c\gamma (1+\alpha P^*)}\,.$
Now we combine inequalities (\ref{A1}),   (\ref{A2})  as well as (\ref{z}) to obtain  
\[\frac{d}{dt}\mathcal{L}(t) +\varrho (t)\leq 0
\]
where 
\[\varrho (t)= \theta \int_{\Omega} (N(x,t)-N^\star)^2dx+\frac{\eta}{2} \int_{\Omega} (P(x,t)-P^\star)^2dx+\frac{\mu C_2}{2}\int_{\Omega}(W(x,t)-W^\star)^2dx\,.
\]
It follows that for any $T>1$ 
\[\mathcal{L}(T) +\int_1^T\varrho(t) dt\leq \mathcal{L}(1)\]
Using the nonnegativity of $\mathcal{L}$, the uniform bound for solution from Theorem \ref{exist} and letting $T\rightarrow\infty$ we infer that
\begin{equation} \label{T}
\int_1^\infty\varrho(t) dt <\infty\,.
\end{equation}
It follows from the  parabolic regularity of the classical solution $(N,P,W)$ to (\ref{ModelB2}) and uniform $L^\infty$- bound that solution components are H\"{o}lder continuous functions on $\Omega\times [1,T]$ with H\"{o}lder constant independent on $T$ (cf.\cite[ChapterV]{Lady}. This fact entails uniform continuity of $\varrho(t)$ , $t\in (1,\infty)$ and we conclude using \cite[Lemma 3.1.]{BaiWinkler}) that
\[\lim_{t\rightarrow\infty} \varrho(t)= 0 \,.
\] 
Next using (\ref{r1}) and the Gagliardo-Nirenberg inequality (see Proposition \ref{G-N} with $p=q=\infty$ and $m=2$)   we obtain that 
\[\Vert N(\cdot,t)-N^\star)\Vert_\infty \leq C_{G-N} \Vert N(\cdot,t)-N^\star)\Vert_{1,\infty}\Vert N(\cdot,t)-N^\star) \Vert_2
\]
and similarly for remaining components of the solution.
It completes the proof.  
\end{proof}
\begin{remark}Using the same arguments as in \cite{BaiWinkler} one can prove that in fact the convergence in Theorem \ref{GlobStab} has  an exponential rate.
\end{remark}

\section{Model A -linear stability.}\label{stabA}
In this  section we consider linear stability of the coexistence steady state $\bar{E}=(\bar{N},\bar{P},\bar{W})$ in (\ref{Esteady}) for model A (\ref{modelA}) assuming the same structural assumptions for the reaction part as in the previous section (\ref{parF}), (\ref{aij}) and (\ref{aij1}) which encompass  the  Holling functional responses (\ref{FH}), Bedington-DeAngelis response (\ref{FB}) as well as Crowley- Martin response (\ref{CM}). The linearisation of the model (\ref{modelA}) at the  coexistence steady state $\bar{E}$ leads to the following stability matrix 
\begin{align}\label{jacm2}
S_j=\begin{pmatrix}
a_{11}-D_1h_j &\displaystyle a_{12}& -\chi \bar{N} h_j\\
a_{21}+\xi \bar{P} h_j & a_{22}-D_2h_j & 0\\
0 & a_{32} & a_{33}-D_3h_j
\end{pmatrix}\,.
\end{align} 
 It's characteristic polynomial follows
\begin{align*}
\sigma^3+\phi^{(1)}_j \sigma^2+\phi^{(2)}_j(\xi) \sigma+\phi^{(3)}_j(\xi) =0,
\end{align*}
where
\begin{align}
 \phi^{(1)}_j&= \rho_j^{(1)}, \label{fi1}\\
 \phi^{(2)}_j(\xi)&=\rho_j^{(2)}-\xi a_{12}\bar{P} h_j :=\rho_j^{(2)}+\xi \phi^{(2,1)}_j, \label{fi2}\\ \label{fi3}
 \phi^{(3)}_j(\xi)&=\rho^{(3)}_j+\xi (a_{12}a_{33}-D_3 a_{12}h_j+\chi \bar{N}  a_{32} h_j) \bar{P} h_j:=\rho^3_j+\xi \phi^{(3,1)}_j, 
 \end{align}
with  $\rho^{(1)}_j,\ \rho^{(2)}_j$ and $\rho^{(3)}_j$ defined in (\ref{ro1})-(\ref{ro3}). 
By the Routh-Hurtwitz stability criterion, $\bar{E}$ is locally stable in model A if and only if for all $j\in \mathbb{N}_{+}$
\begin{equation}\label{RH2}
\phi^{(1)}_j>0,\  \phi^{(3)}_j(\xi)>0 \ \text{and}\  \phi^{(1)}_j\phi^{(2)}_j(\xi)-\phi^{(3)}_j(\xi)>0.
\end{equation}
It follows from (\ref{fi1})-(\ref{fi3}) that 
\begin{equation}
\phi^{(1)}_j\phi^{(2)}_j-\phi^{(3)}_j(\xi)=\rho^{(1)}_j\rho^{(2)}_j-\rho^{(3)}_j-\xi(\rho^{(1)}_j\phi_j^{(2,1)}+\phi_j^{(3,1)}):=\rho^{(1)}_j\rho^{(2)}_j-\rho^{(3)}_j-\xi \phi_j^{(4)}\,. \label{fi4}
\end{equation}
Hence, by   (\ref{fi4}) using (\ref{ro3}) and (\ref{ro31}) we obtain that  
\begin{align}
\rho^{(1)}_j\rho^{(2)}_j-\rho^{(3)}_j-\xi \phi_j^{(4)}&=\rho^{(1)}_j\rho^{(2)}_j-\rho_j^{(3,1)} -\chi a_{21}a_{32} \bar{N}h_j-\xi \phi_j^{(4)} \nonumber \\
&:=\tilde{Q}(h_j)=\Psi(h_j)- \chi a_{21}a_{32}\bar{N}h_j-\xi \phi_j^{(4)}\,.
\end{align}
By  straightforward calculation using (\ref{ro1}) we have
 \begin{align}
\label{fi44}
\phi_j^{(4)}&= \bar{P}h_j\left(-(a_{12}a_{11}+a_{12}a_{22})+((D_1+D_2)a_{1,2}+\chi a_{32}\bar{N})h_j\right)\,, \\ \label{fi45}
&:=\bar{P}h_j\left(\zeta_1 +\zeta_2(\chi) h_j\right)\,.
\end{align} 
Notice that from the fact that $a_{12}\,, a_{11}<0$,  $a_{22}\leq 0$ and  $a_{32}>0$ it follows that $\zeta_1<0$  and $\zeta_2(\chi)<0$ provided
\begin{equation} \label{RHA}
\quad \chi <\chi^S:= \frac{-a_{12}(D_1+D_2)}{a_{32}\bar{N}}\,.
\end{equation} 
Now we are in position to formulate the following  stability result for the coexistence steady state $\bar{E}$ in model A.
\begin{theorem}\label{ThmstabA}
Suppose that $\chi<\max\{\chi^H\,,\chi^S\}$ and $\xi>0$. Under assumptions (\ref{FGwB})  and (\ref{parF})  the following conditions determine the local stability of the constant steady state $\bar{E}$ in model A. 
\noindent
\begin{enumerate}
\item Suppose that  $\chi^S\leq \chi^H$ .
\begin{enumerate}
\item[a)]
If  $\chi\in(0\,,\chi^S)$ then $\bar{E}$ is locally asymptotically stable for all $\xi\geq 0$.
\item[b)]
If  $\chi\in(\chi^S\,,\chi^H)$ then there exists $\xi^S>0$ such that $\bar{E}$ is locally asymptotically stable for all $\xi<\xi^S$ and it is unstable if $\xi>\xi^S$.
\end{enumerate}
\item Suppose that  $\chi^S>\chi^H$. 
\begin{enumerate}
\item[a)]If $\chi\in(0\,,\chi^H)$ then $\bar{E}$ is locally asymptotically stable for all $\xi\geq 0$.
\item[b)] If $\chi\in(\chi^H\,,\chi^S)$ then there exists $\xi^\star>0$ such that  $\bar{E}$ is locally asymptotically stable for all $\xi>\xi^\star$ and it is unstable if $\xi<\xi^\star$
\end{enumerate}
\end{enumerate}
\end{theorem} 
\begin{proof}

\noindent
In the case $1a)$ there holds (\ref{RHA}), so, $\zeta_2(\chi)<0$ and $\xi \phi_j^{(4)}<0$. At the same time there holds (\ref{local1B}) and for $\chi<\chi^H$ by (\ref{local2B}) we have 
\begin{equation}\label{phj}
R(h_j)(\chi):= \Psi(h_j)- \chi a_{21}a_{32}\bar{N}h_j>0
\end{equation}
and we infer that for any $j\geq 0$ and $\xi\geq 0$ 
\[\tilde{Q}(h_j)=R(h_j)(\chi)-\xi \phi_j^{(4)}>0\]
which completes the proof of $1a)$.

\noindent
In the case $1b)$ we have  $\zeta_2(\chi)\geq 0$ and 
\begin{equation}\label{zh}
\zeta_1+\zeta_2(\chi) h_j>0 
\end{equation}
for 
\[ h_j\geq h^\star:= \frac{-\zeta_1}{\zeta_2(\chi)}\,.\]
Then $\tilde{Q}(h_j)>0$ for any $\xi>0$ and $j\geq 0$ such that $h_j\leq h^\star$ and (\ref{phj})  still holds. By (\ref{fi44}) and (\ref{zh})  there exists $\xi^S>0$ such that 
\begin{equation} \label{ksis}
\xi^S=\min_{\{j:\{h_j>h^\star\}}\left\{\frac{R(h_j)(\chi)}{\bar{P}h_j\left(\zeta_1 +\zeta_2(\chi) h_j\right)} \right\}\,.
\end{equation}
 It completes the proof of 1b).                                                       

\noindent
The case $2a)$ is similar to $1a)$. In the case $2b)$ for some $j>0$ we have   $R(h_j)(\chi)<0$ . 
Since $\Psi(x) $ in (\ref{phj}) is a third order polynomial with positive coefficients (see (\ref{ro31}) ) there is $j^\star$ such that 
\[R(h_{j^\star})= \min_{\{j\in N_+\}} R(h_{j})\,.\]
Since $\zeta_2(\chi)<0$ we define 
\[\min_{\{j\in N_+\}} \{-\bar{P}h_j(\zeta_1+\zeta_2(\chi)h_j)\}=K_0>0
\]
end finally there is a minimal $\xi=\xi^\star>0$ such that
\[\tilde{Q}(h_j)\geq R(h_{j^\star}) +\xi K_0> 0\quad\mbox{for}\quad \xi>\xi^\star
\]
and proof of $2b)$ is completed. 
\end{proof}
\section{Numerical Simulations}\label{NumSim}
In this section, we present numerical results for  model B  (\ref{modelB}) \& model A (\ref{modelA}) which exhibit  the spatio-temporal dynamics of the proposed models. We fix a set of positive parameters and investigate the spatio-temporal dynamics for models A  (\ref{modelA})  \& B  (\ref{modelB}) with special emphasis on Holling II functional response which corresponds to the extension of the Rosenzweig-MacArthur model (\ref{ModelB1}). Patterns obtained  for model A and model B  with Beddington-DeAngelis functional response turned out not to  exhibit essentially new effects with respect to Rozenzweig-MacArthur model and were not included to this section.  Solutions  in 1D domain are obtained with the help of MATLAB PDEPE tool ($\Delta x=0.01\,, \Delta t=0.1$) and for 2D simulations  Freefem++ with $\Delta x=\Delta y=0.01,\ \Delta t=0.1$ was used. For the  following  values of model parameters which are chosen in simulations
\begin{align}\label{para1}
\begin{aligned}
 & r=0.25,\ \alpha=0.5,\ \beta=2,\ b=0.85,\  a=0.95,\ \delta=0.17, \mu=0.5,\ \gamma=10, \ d_p=0.01,\\\
&  \ d_w=0.01.
 \end{aligned}
\end{align}
model (\ref{ModelB1}) has  the  unique positive coexistence steady state 
\begin{equation}\label{Enum}
\bar{E}=(0.3333,\ 0.2924,\ 5.8490)\,.
\end{equation}  First we start with some numerical results related to model B (\ref{ModelB1}) and discuss the impact of chemo-repulsion on the stability of predator-prey system. 

The figures contained in this section exhibit the following features of solutions :
\begin{itemize}
\item stabilization to the constant steady state; Figs. 1 \& 7,
\item periodic space-time patterns corresponding to periodic initial data; Figs. 3, 8 \& 9,
\item emergence of periodic or almost periodic patterns  corresponding to a localized in space  initial perturbation of the coexistence steady state; Figs. 2, 4 \& 14,
\item transient patterns with abrupt change of characteristic scale of oscillations Figs. 5, 10 \& 14,
\item oscillatory rings and periodic change of pattern geometry illustrating pursuit and evasion dynamics in 2D-simulations in square; Figs. 6, 12\& 13, 
\item formation of singular spiky solutions in 2D simulation; Fig. 11. 
\end{itemize}
\subsection{Numerical simulation for Model B in version (\ref{ModelB1})}
 The objective of this section is to investigate the transient dynamics for the model B (\ref{ModelB1}) depending upon the parameter $\chi$ and initial data with  fixed set of parameters (\ref{para1}). Accordingly numerical values for the stability matrix (\ref{smat1}) and the critical value $\chi=\chi^H$  are following
\begin{align*}
M_j\approx\begin{pmatrix}
-0.0167-h_j & -0.19 & -0.3333\chi h_j\\
0.0895& -0.01h_j & 0\\
0 & 10 & -0.5-0.01 h_j
\end{pmatrix}.
\end{align*}
The coefficients of polynomial (\ref{char1}) are calculated as
\begin{align*}
&\rho_j^1\approx0.55+1.02h_j>0,\ \rho_j^2\approx0.0201h_j^2+0.506h_j+0.1994>0,\\
& \rho_j^3\approx0.0001h_j^3+0.005h_j^2+0.002h_j+0.0872+1.02\chi h_j>0,\\
& \rho_j^1\rho_j^2-\rho_j^3\approx0.0201h_j^3+0.506h_j^2+0.0546h_j+0.1994-1.02\chi h_j.
\end{align*}
The $\chi^H$ defined in (\ref{chih}) is calculated as
\begin{align*}
\chi^H\approx \min_{j\in \mathbb{N}_{+}}\Big\{ \frac{0.0201h_j^3+0.506h_j^2+0.0546h_j+0.1994}{1.02 h_j}\Big\}
\end{align*}
which is attained  at $j=1$ (i.e. $h_1=\frac{\pi^2}{L^2}$). We numerically obtain the minimum value of $\chi^H=6.889$ in the unit domain for the parameter values defined in (\ref{para1}). 

Fig. \ref{fig1}  presents numerical illustration of linear stability of the coexistence steady state for model B in the case when the initial data is the following  perturbation of the constant steady state 
\begin{align}\label{idata}
N(x,0)=\bar{N}+0.1 \cos\big(\frac{j\pi x}{L}\big),\ P(x,0)=\bar{P}+0.1\cos\big(\frac{j\pi x}{L}\big),\ W(x,0)=\bar{W}+\cos\big(\frac{j\pi x}{L}\big)
\end{align}
with unit domain and $j=1$. As expected  the solution  approaches  the  constant steady state $\bar{E}$  for $\chi<\chi^H$.  

Fig. \ref{fig2}  depicts simulations corresponding to initially homogeneous in space  distribution of prey and the chemical  along with initial  cluster of predators in the middle of the domain.  We observe evolution of patterns when $\chi>\chi^H$. It is worth noting that prey are able to avoid and successfully escape from predator dominant area when the time passes. It has been observed that although initially only predator density was perturbed  the amplitude of periodic patterns for the prey is much higher then both  the predator and the chemical. 

In Fig. \ref{fig3} one can see the emergence of periodic and spatially inhomogeneous patterns for the symmetric initial data (\ref{idata}) with $j=2$. The patterns are more clear for the distribution of  prey then that of predator and chemical. Prey prefer to migrate to the corners of the 1D domain and exchange the position periodically  with the predator. The predator and the chemical show a similar behavior  with a significantly smaller amplitude of fluctuations. These observations suggest that chemo-repulsion  driven instability highly affects the spatial distribution of prey and much less  the predator's distribution in case when the motility of predators is subject to the diffusive spread (Model B) of random movement of predators. 

Figs. \ref{fig4} and \ref{fig7} present simulations for increased domain size $ L = 10 $, in which all other parameters are kept the same (\ref{para1}). First we show transient  patterns starting from the initial data $N(x,0)=\bar{N}, P(x,0)=\bar{P}+0.1 \cos\big(\frac{2\pi x}{L}\big),\ W(x,0)=\bar{W}$  for a high value of  $\chi>\chi^H$. For $L=10$ we find $\chi^H\approx 2.0834$. From Figs \ref{fig41} \& \ref{fig42} we infer that prey very quickly runs away from the predator's dominant area and creates nice  spatial structures that are  non-periodic up to $ t = 19 $ but after some time the predator leads to very little dominance. It is important to note that the space-time periodic pattern appears after some time (see Fig. \ref{fig43}). On the other hand Figs  \ref{fig71} \& \ref{fig72} give an interesting example of abrupt structural   change of a regular pattern  which appears at some time ($t=40$) for the solution starting from asymmetric initial data (\ref{idata}) such that $j=5$. 
This result reveals that if chemo-sensitivity coefficient is high enough then irregular spatio-temporal pattern may appear. We observe large amplitude fluctuation in the prey population whereas amplitude of predator population fluctuation is very small and negligible (see Figs \ref{fig71} \& \ref{fig72}).

\begin{figure}[hbt!]  
\centering 
\begin{subfigure}{0.3\textwidth}
               \includegraphics[width=\textwidth]{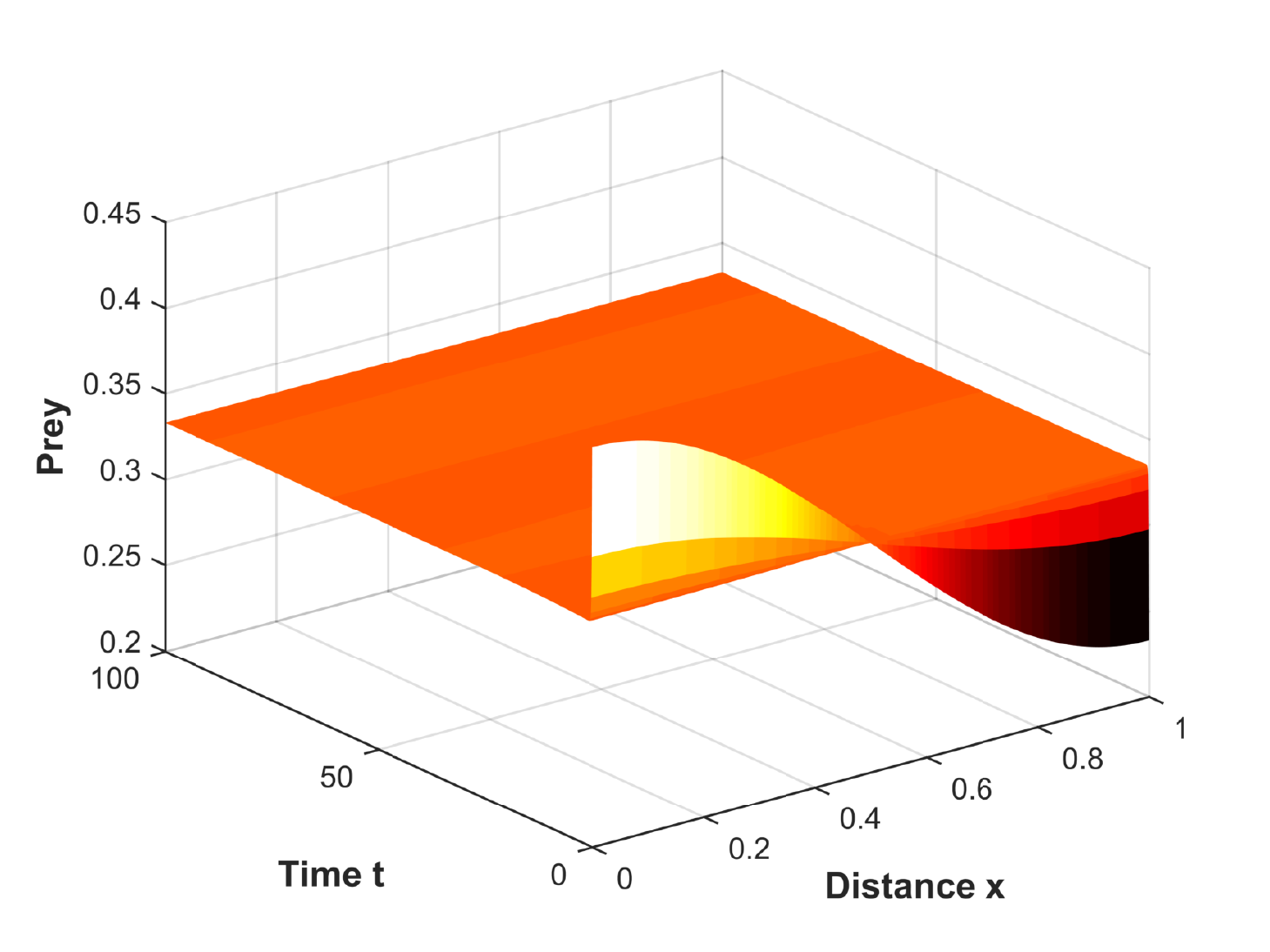} 
                \caption{}
                \label{fig11}
        \end{subfigure}
        \begin{subfigure}{0.3\textwidth} 
                  \includegraphics[width=\textwidth]{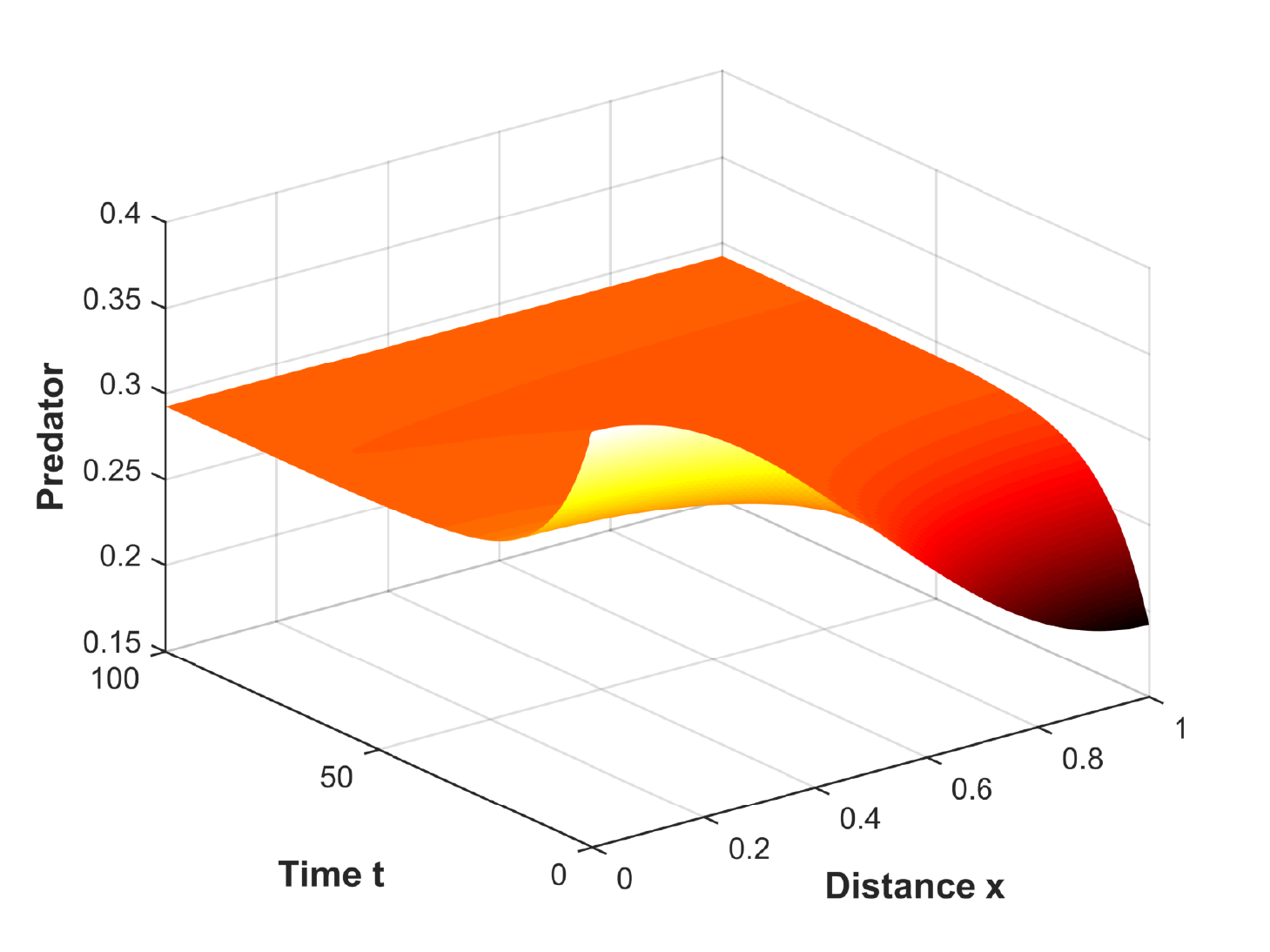}  
                \caption{}
                \label{fig12}
        \end{subfigure} 
     \begin{subfigure}{0.3\textwidth} 
    	\includegraphics[width=\textwidth]{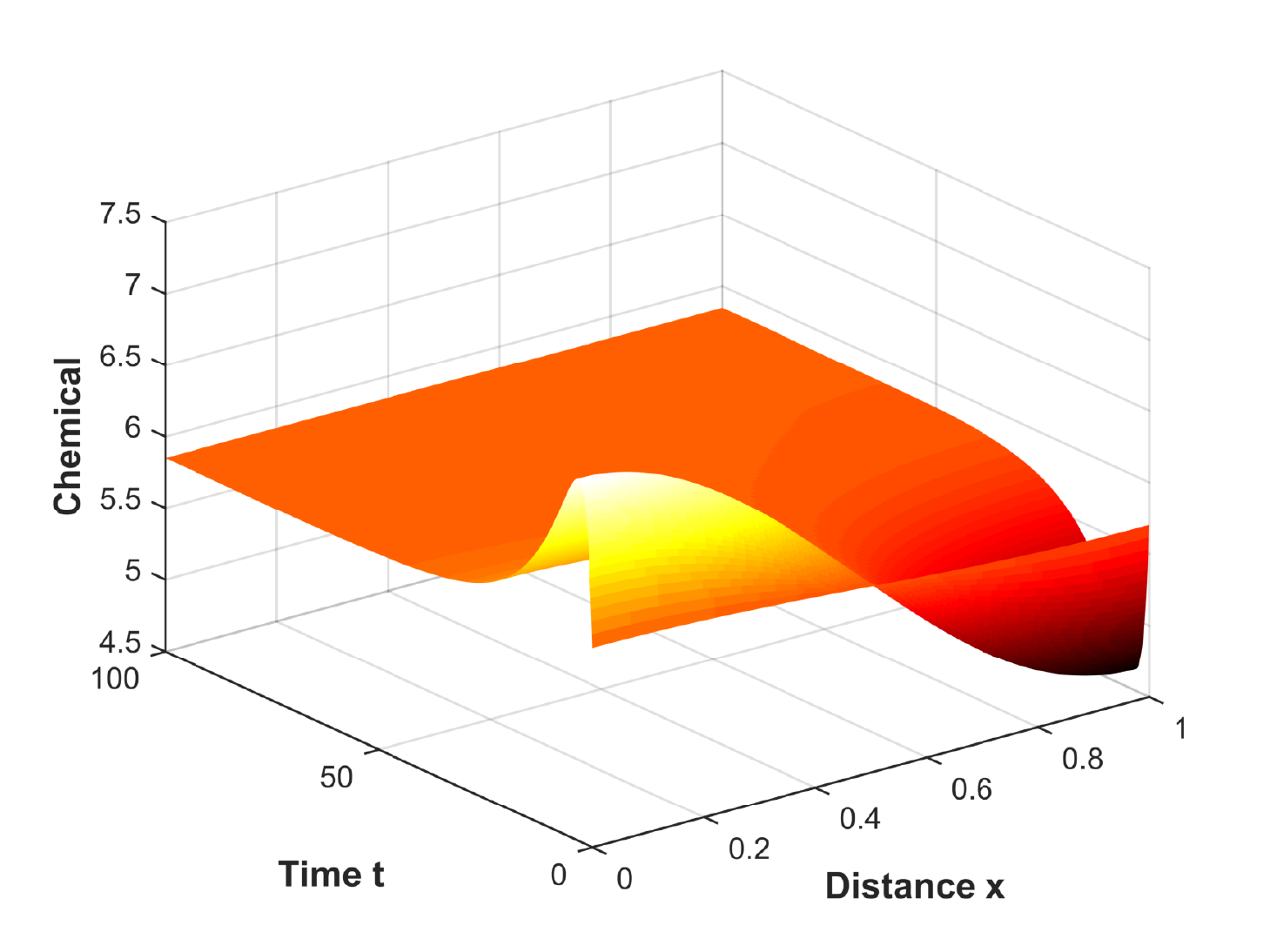} 
    	\caption{}
    	\label{fig13}
    \end{subfigure} 
\caption{Model B: spatio-temporal perturbation (\ref{idata}) in model (\ref{ModelB1}) approaches the  constant steady state $\bar{E}$ (c.f. (\ref{Enum})) for parameter set (\ref{para1}) with $\chi<\chi^H$.}
\label{fig1}
\end{figure}
\begin{figure}[hbt!] 
\center
\begin{subfigure}{0.3\textwidth}
               \includegraphics[width=\textwidth]{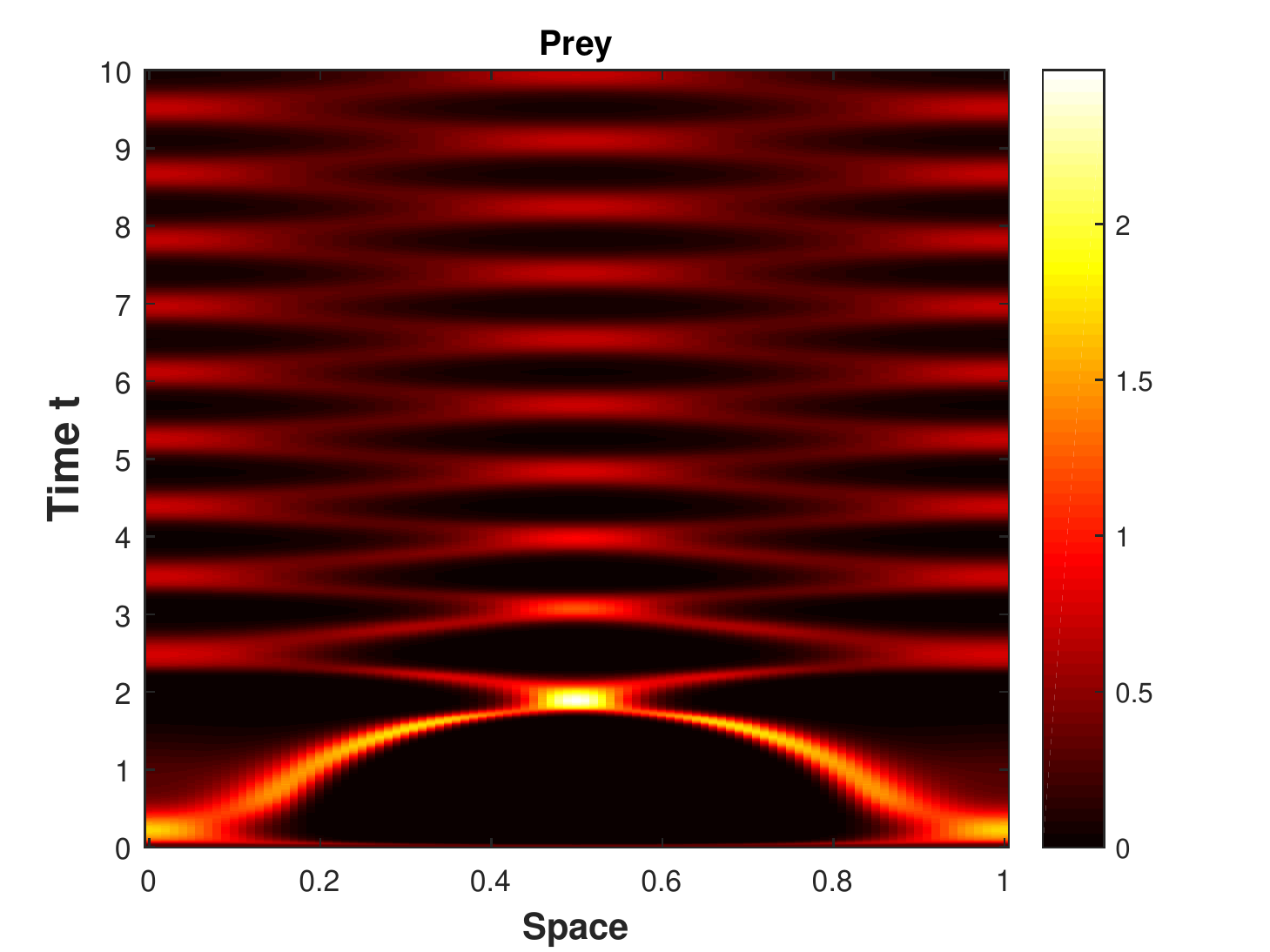} 
                \caption{}
                \label{fig21}
        \end{subfigure}
        \begin{subfigure}{0.3\textwidth} 
                  \includegraphics[width=\textwidth]{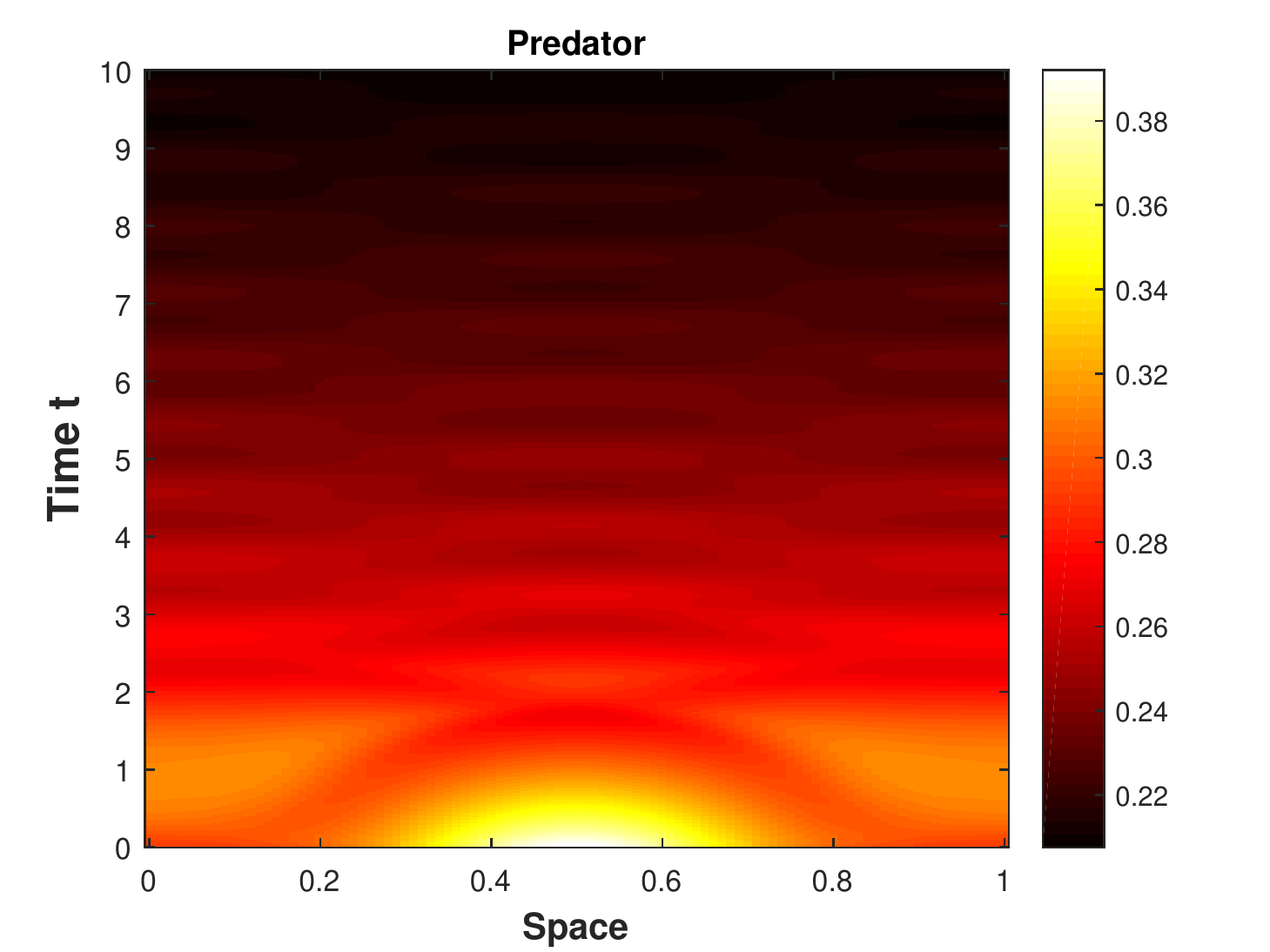} 
                \caption{}
                \label{fig22}
        \end{subfigure} 
     \begin{subfigure}{0.3\textwidth} 
    	\includegraphics[width=\textwidth]{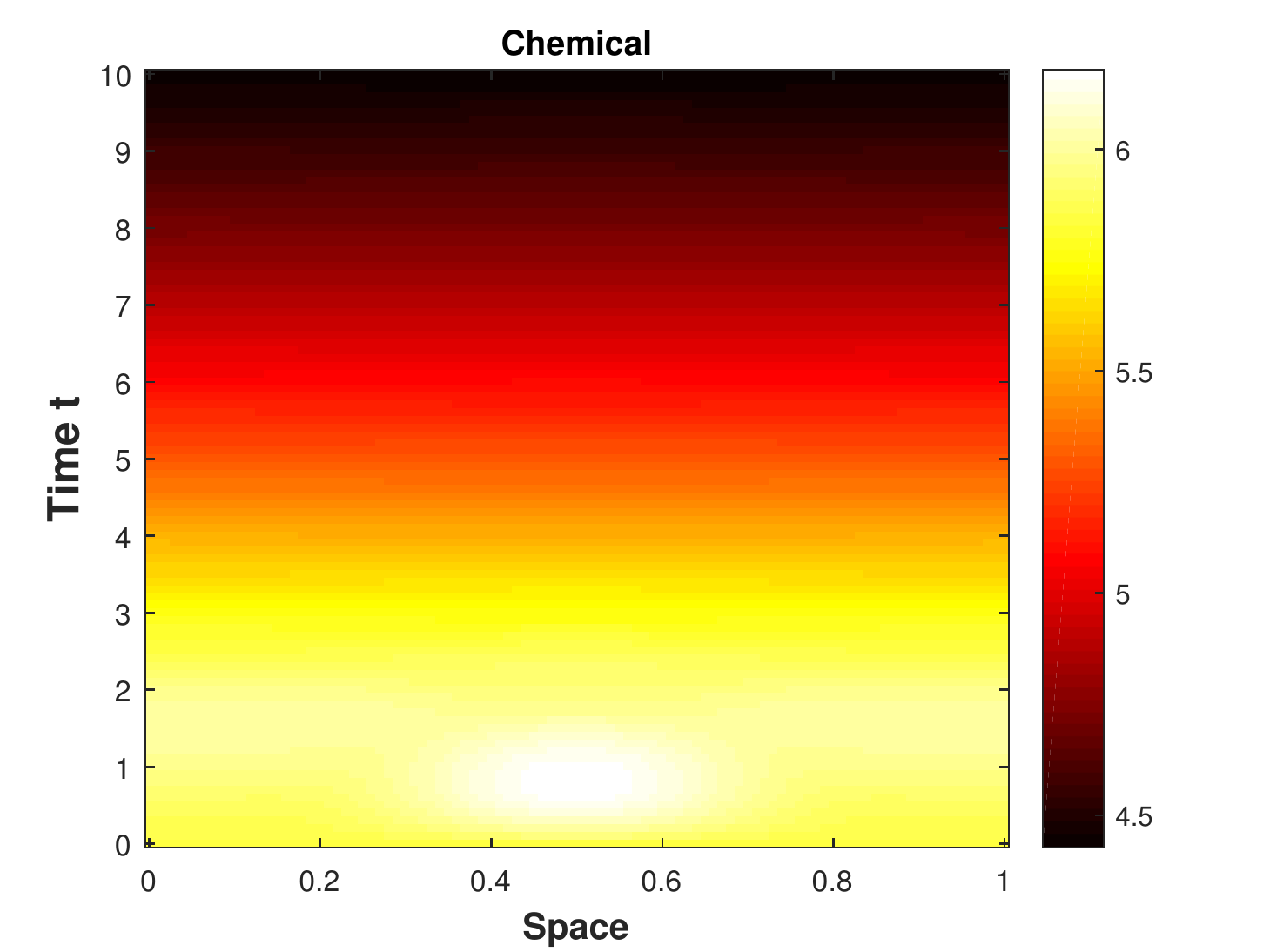} 
    	\caption{}
    	\label{fig23}
    \end{subfigure} 
\caption{Model B: emergence of the spatio-temporal patterns for $\chi>\chi^H$ parameter set (\ref{para1}) of the model(\ref{ModelB1}) for initial data $N(x,0)=\bar{N}, P(x,0)=\bar{P}+0.1e^{-( \frac{x-0.5}{0.2})^2},\ W(x,0)=\bar{W}$.}
\label{fig2}
\end{figure}

\begin{figure}[hbt!] 
	\center
		\includegraphics[height=3.5cm,width=5cm]{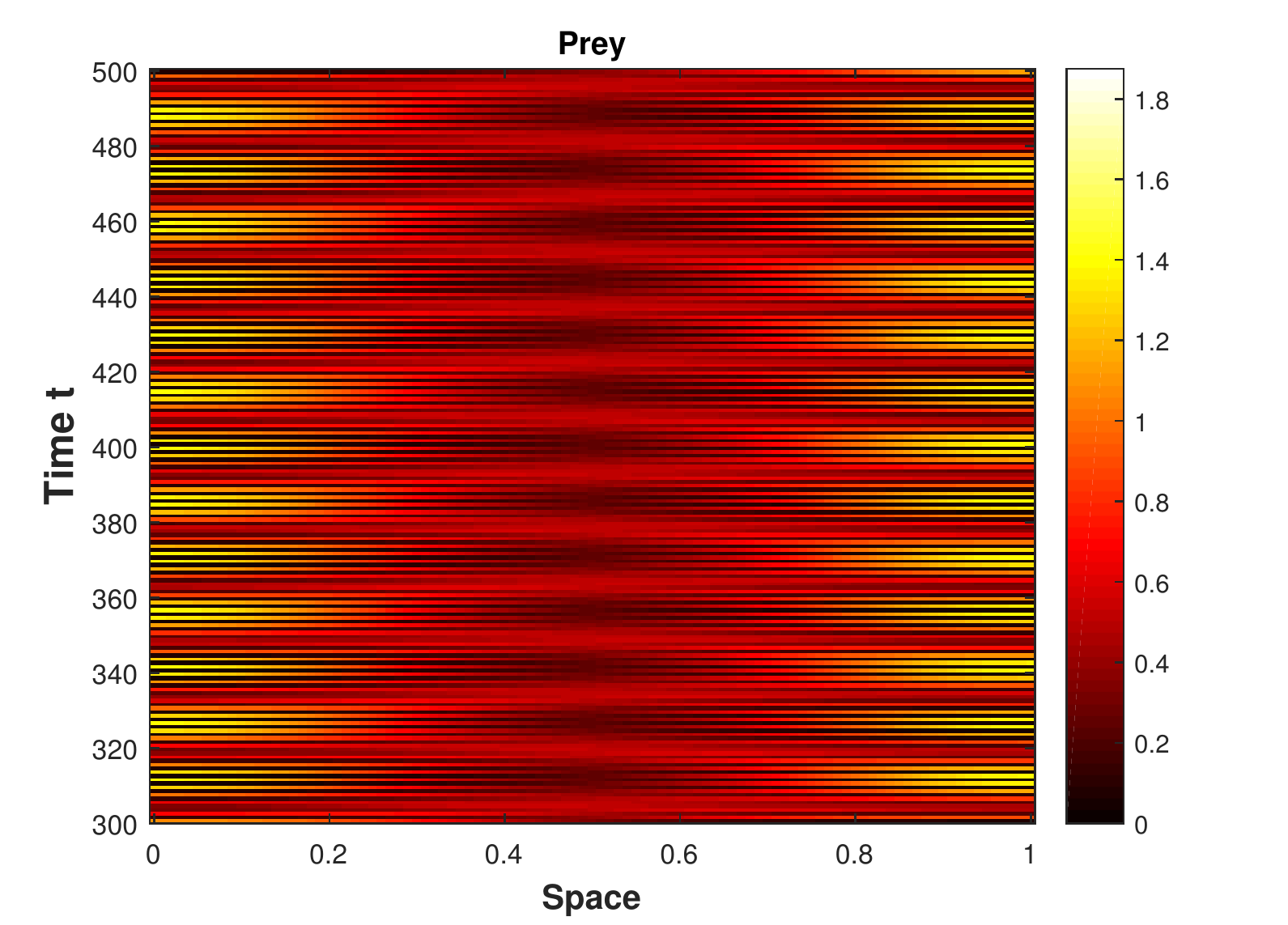} 
		\includegraphics[height=3.5cm,width=5cm]{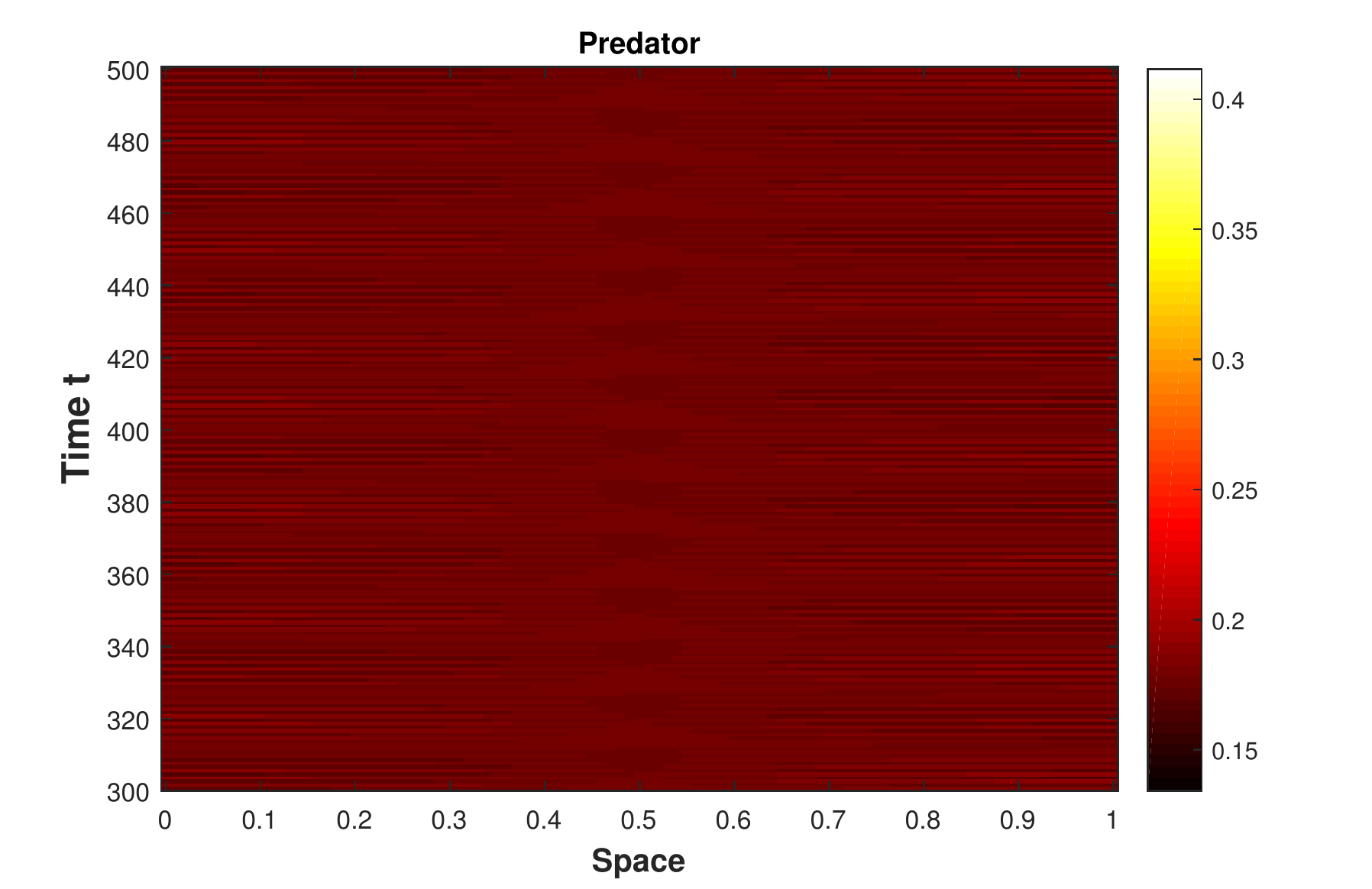}  
		\includegraphics[height=3.5cm,width=5cm]{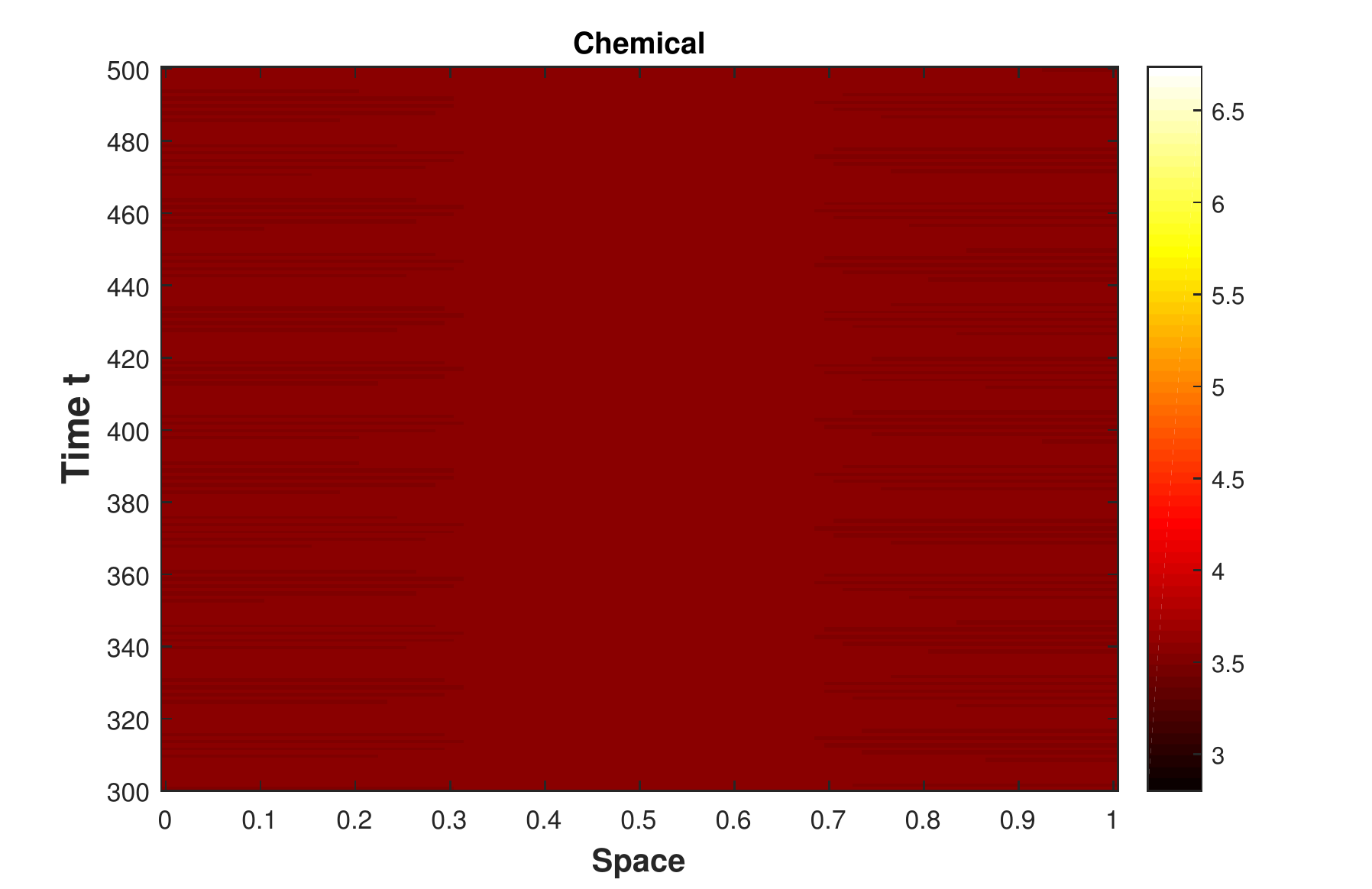}  
	\caption{Model B: spatio-temporal patterns emerged for non-symmetric initial data defined in (\ref{idata}) for $j=2$ in  the unit domain and $\chi>\chi^H$.}
	\label{fig3}
\end{figure}
\begin{figure}[hbt!] 
	\center
	\begin{subfigure}{0.3\textwidth}
		\includegraphics[width=\textwidth]{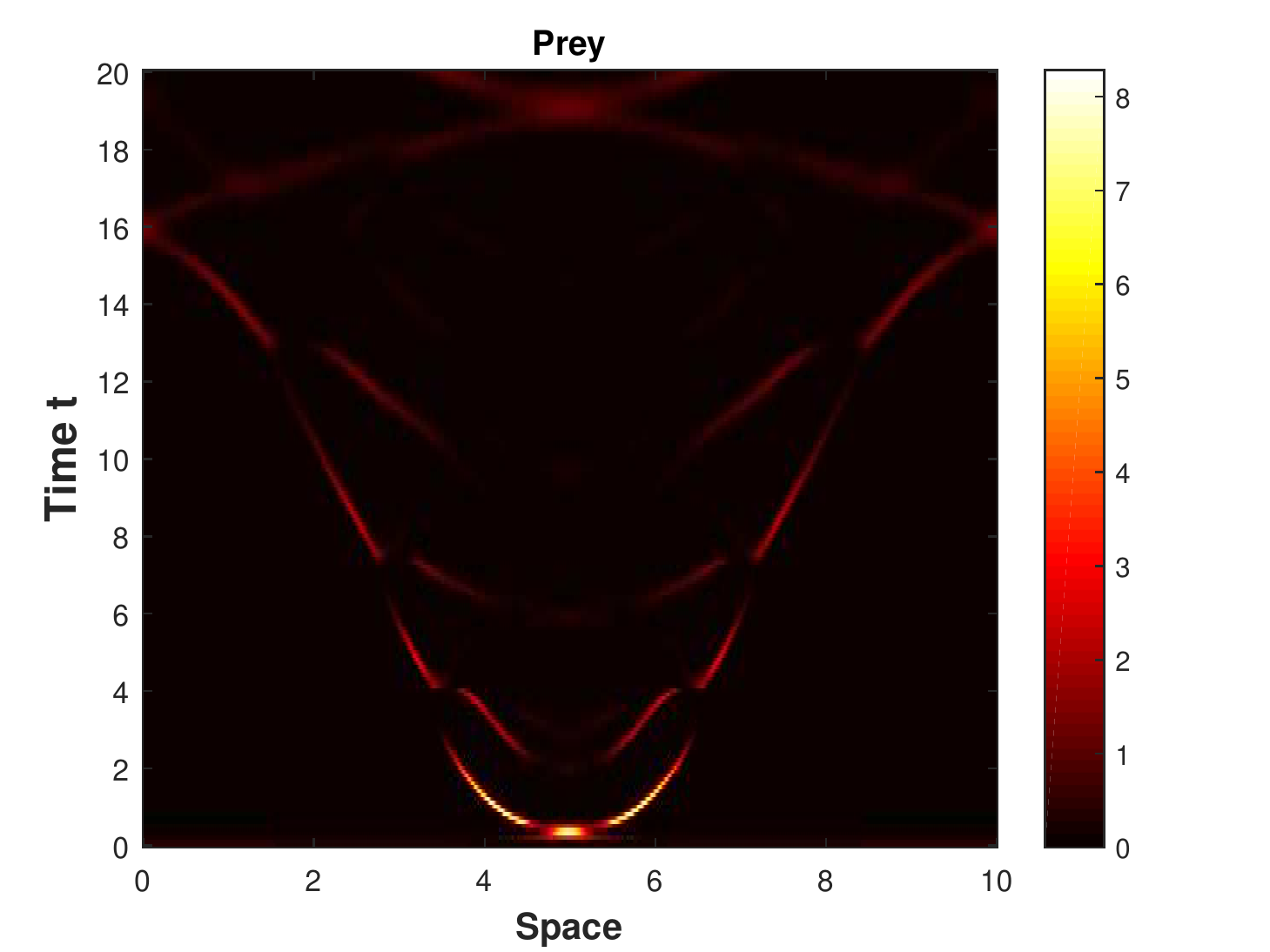} 
		\caption{}
		\label{fig41}
	\end{subfigure}
	\begin{subfigure}{0.3\textwidth} 
		\includegraphics[width=\textwidth]{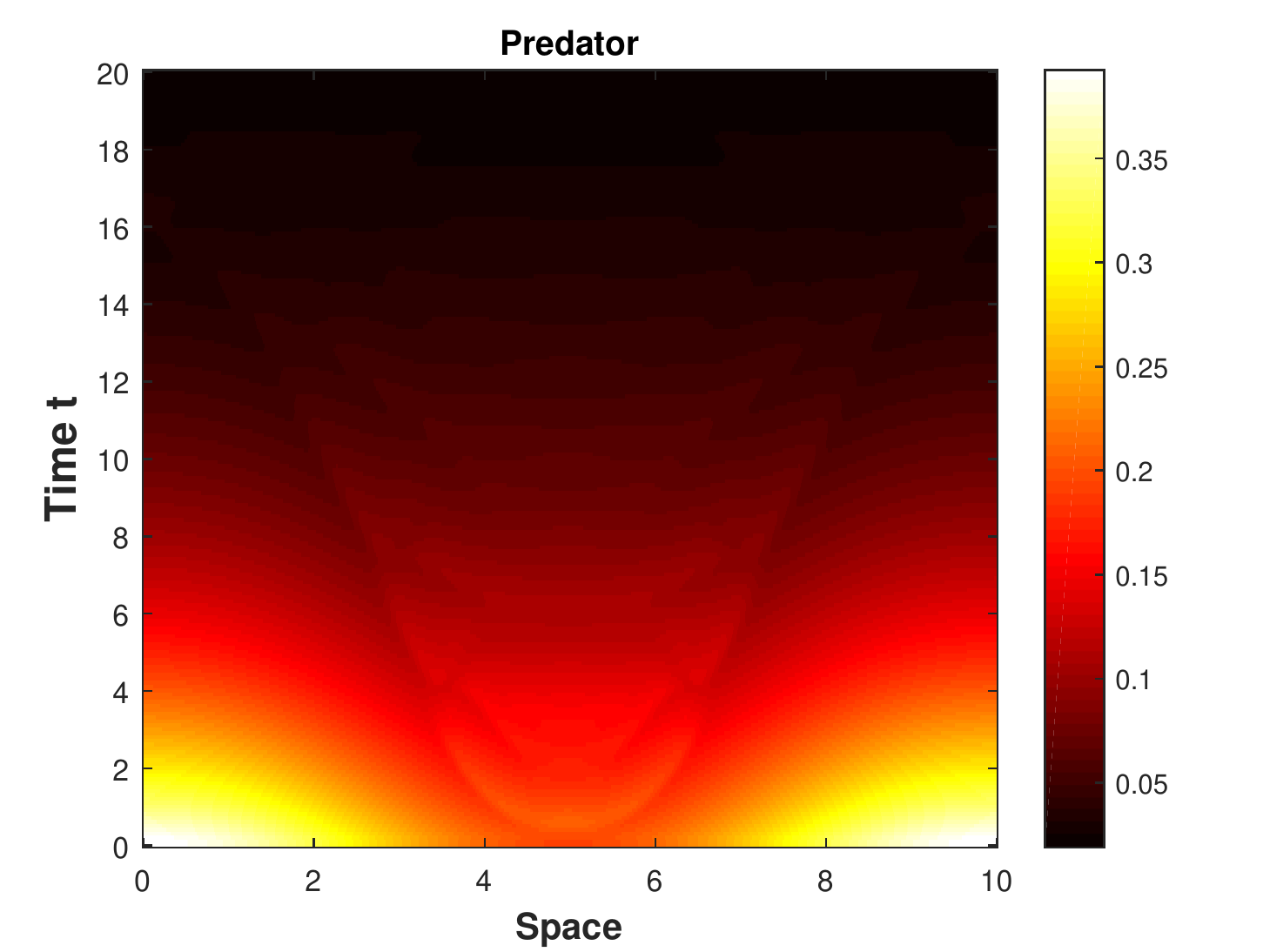}   
		\caption{}
		\label{fig42}
	\end{subfigure} 
	\begin{subfigure}{0.3\textwidth} 
	\includegraphics[width=\textwidth]{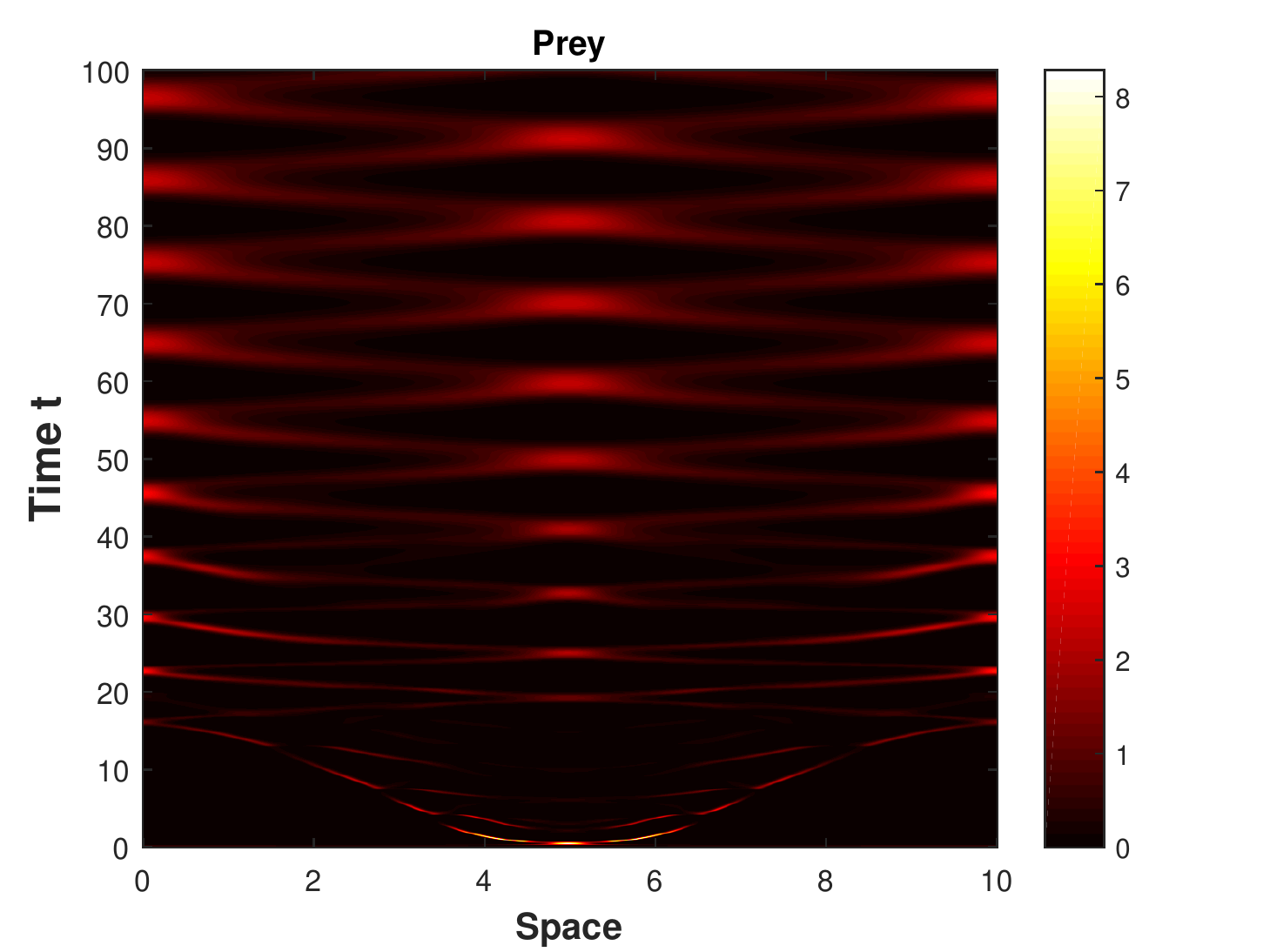} 
	\caption{}
	\label{fig43}
\end{subfigure} 
	\caption{(a)-(b)-(c) Model B: evolution of patterns starting from initial data $N(x,0)=\bar{N}, P(x,0)=\bar{P}+0.1 \cos\big(\frac{2\pi x}{L}\big),\ W(x,0)=\bar{W}$  in  the enlarged domain $L=10$ and $\chi>\chi^H$.}
	\label{fig4}
\end{figure}
\begin{figure}[hbt!] 
	\center
	\begin{subfigure}{0.3\textwidth}
		\includegraphics[width=\textwidth]{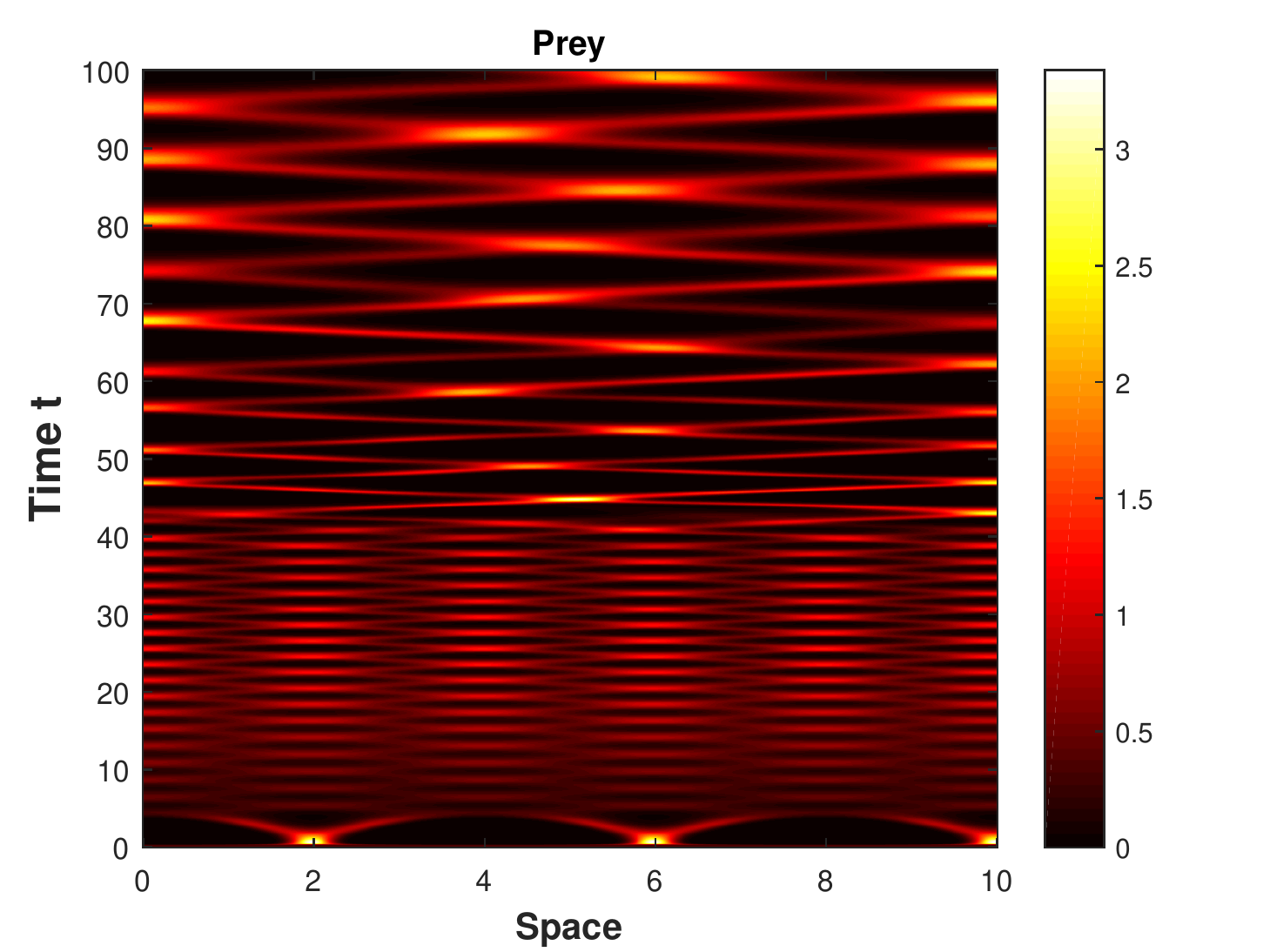} 
		\caption{}
		\label{fig71}
	\end{subfigure}
	\begin{subfigure}{0.3\textwidth} 
		\includegraphics[width=\textwidth]{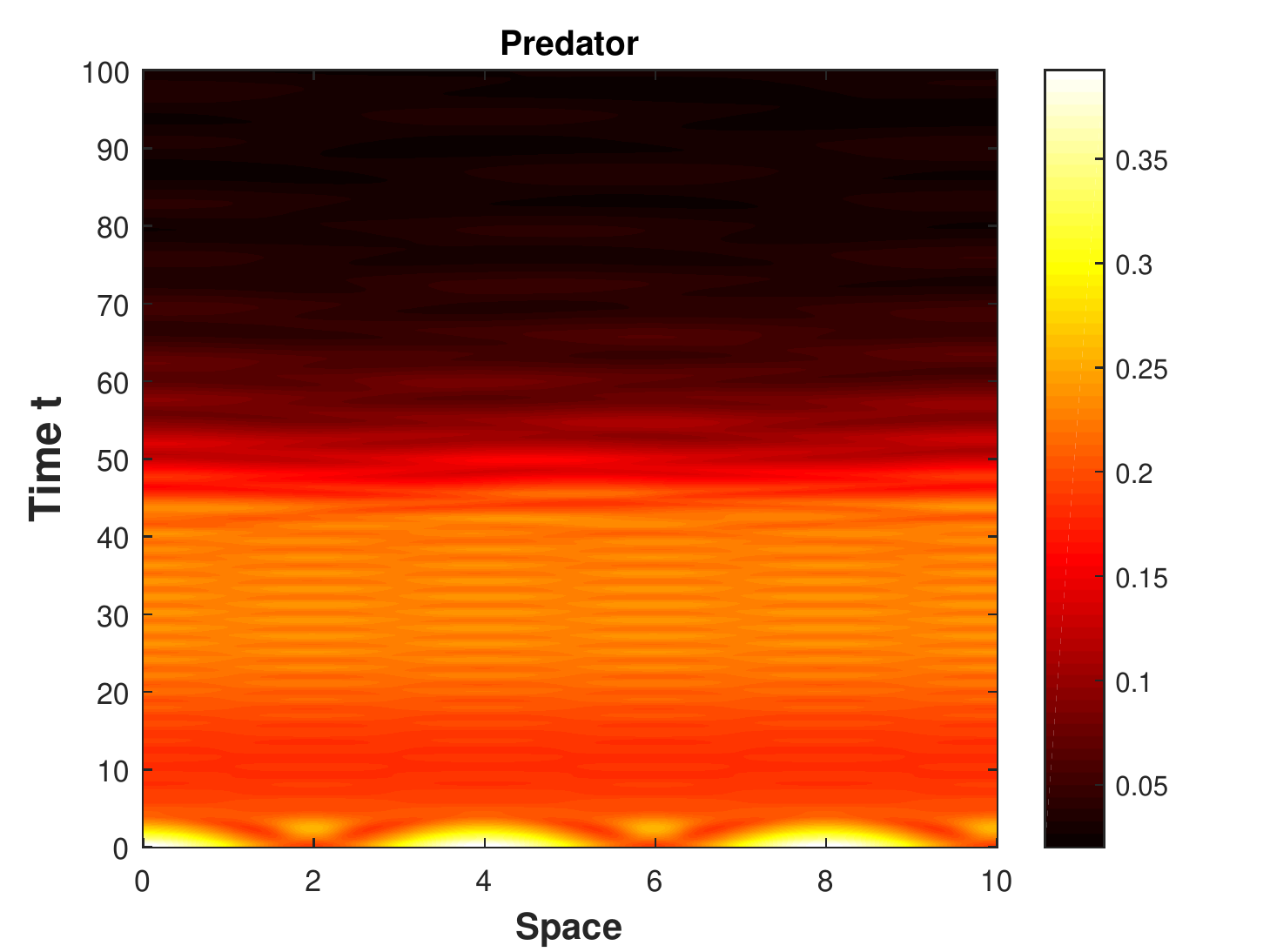}   
		\caption{}
		\label{fig72}
	\end{subfigure} 
	\caption{Model B: transient spatio-temporal pattern for $\chi$ very far away from its critical values in the enlarged domain $L=10$ with initial data $N(x,0)=\bar{N}, P(x,0)=\bar{P}+0.1 \cos\big(\frac{4\pi x}{L}\big),\ W(x,0)=\bar{W}$.}
	\label{fig7}
\end{figure}
\begin{figure}[hbt!]  
\centering 
\subfloat[\label{chm1}]{\includegraphics[height=3cm,width=4cm]{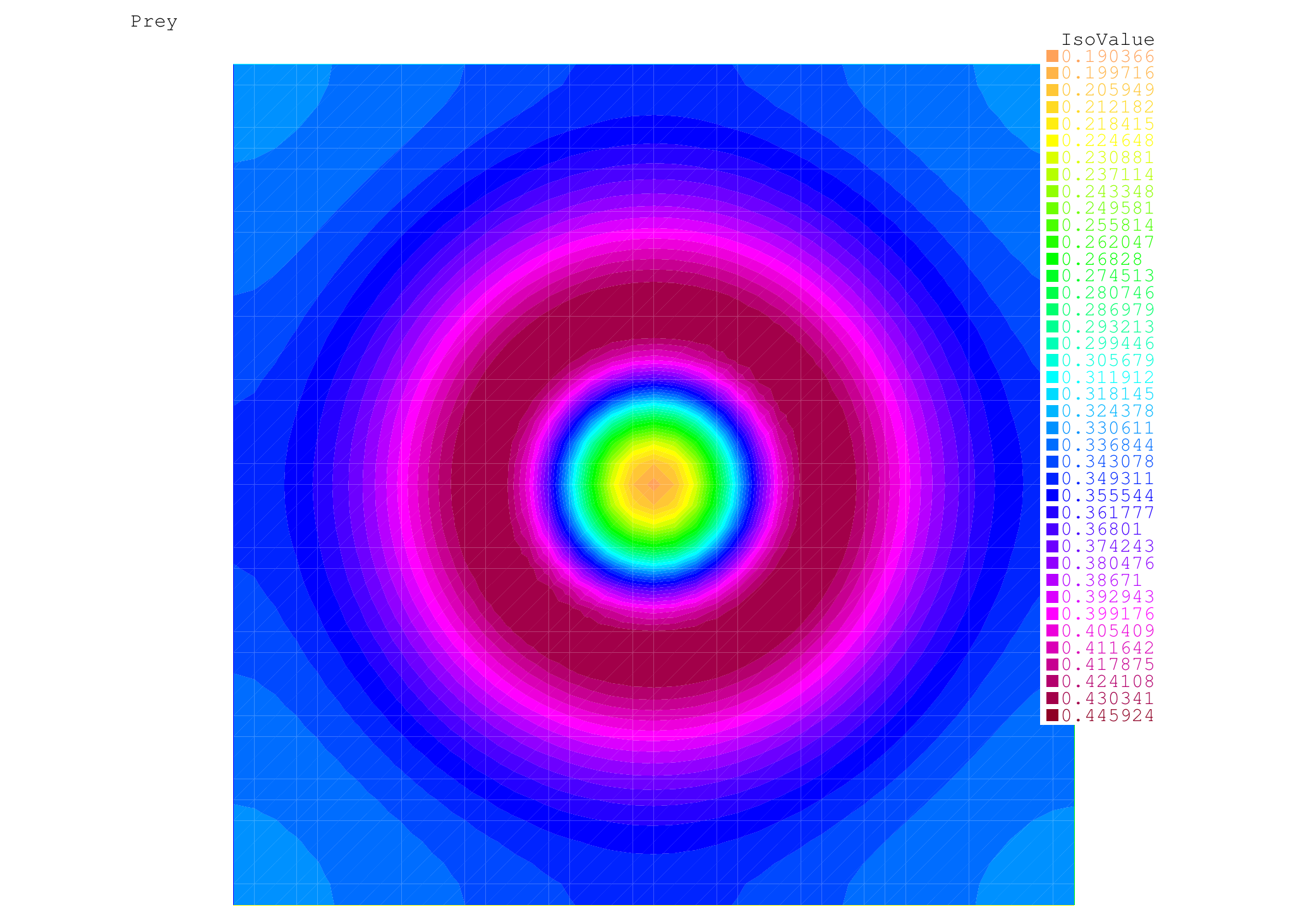}\includegraphics[height=3cm,width=4cm]{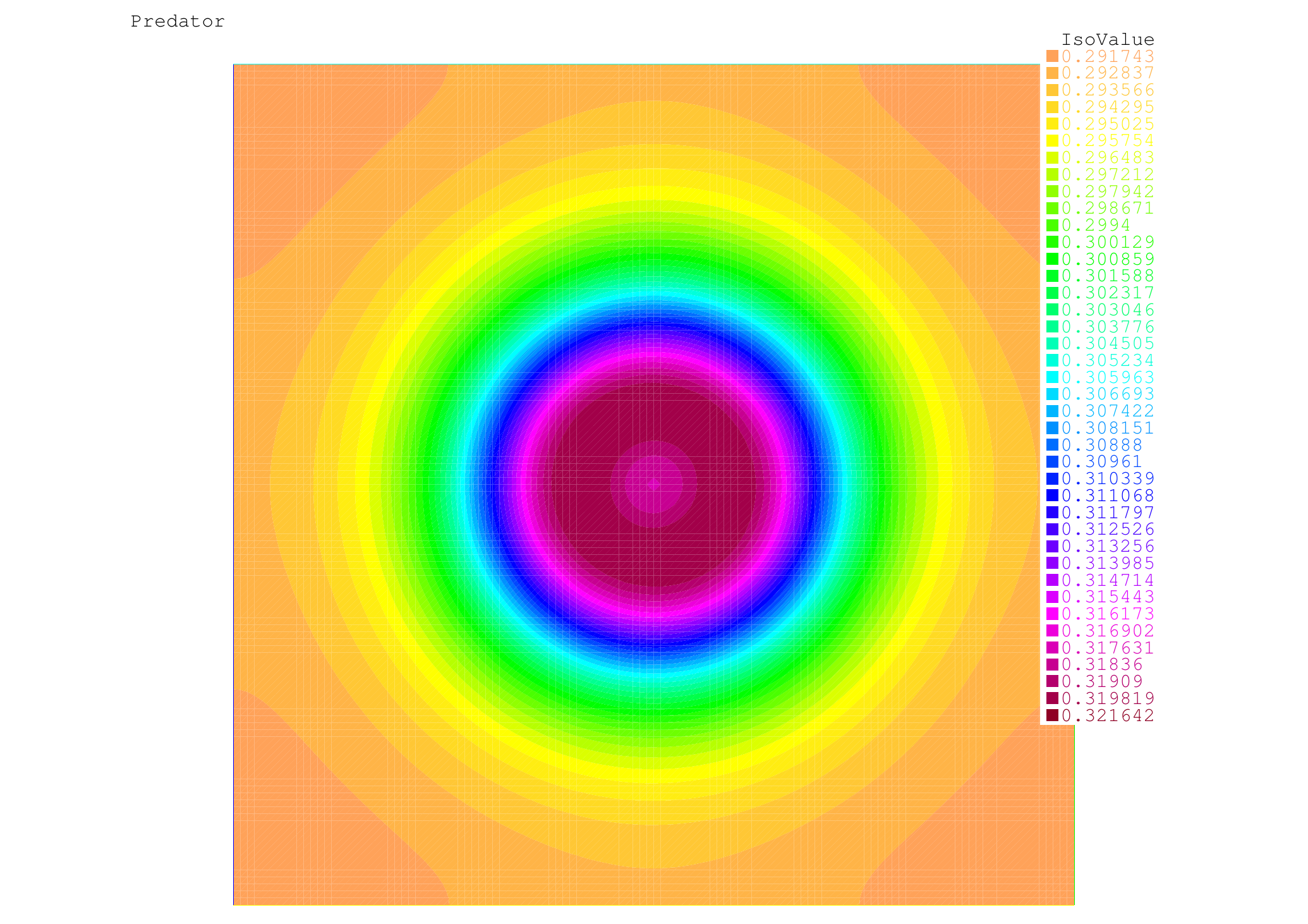} 	\includegraphics[height=3cm,width=4cm]{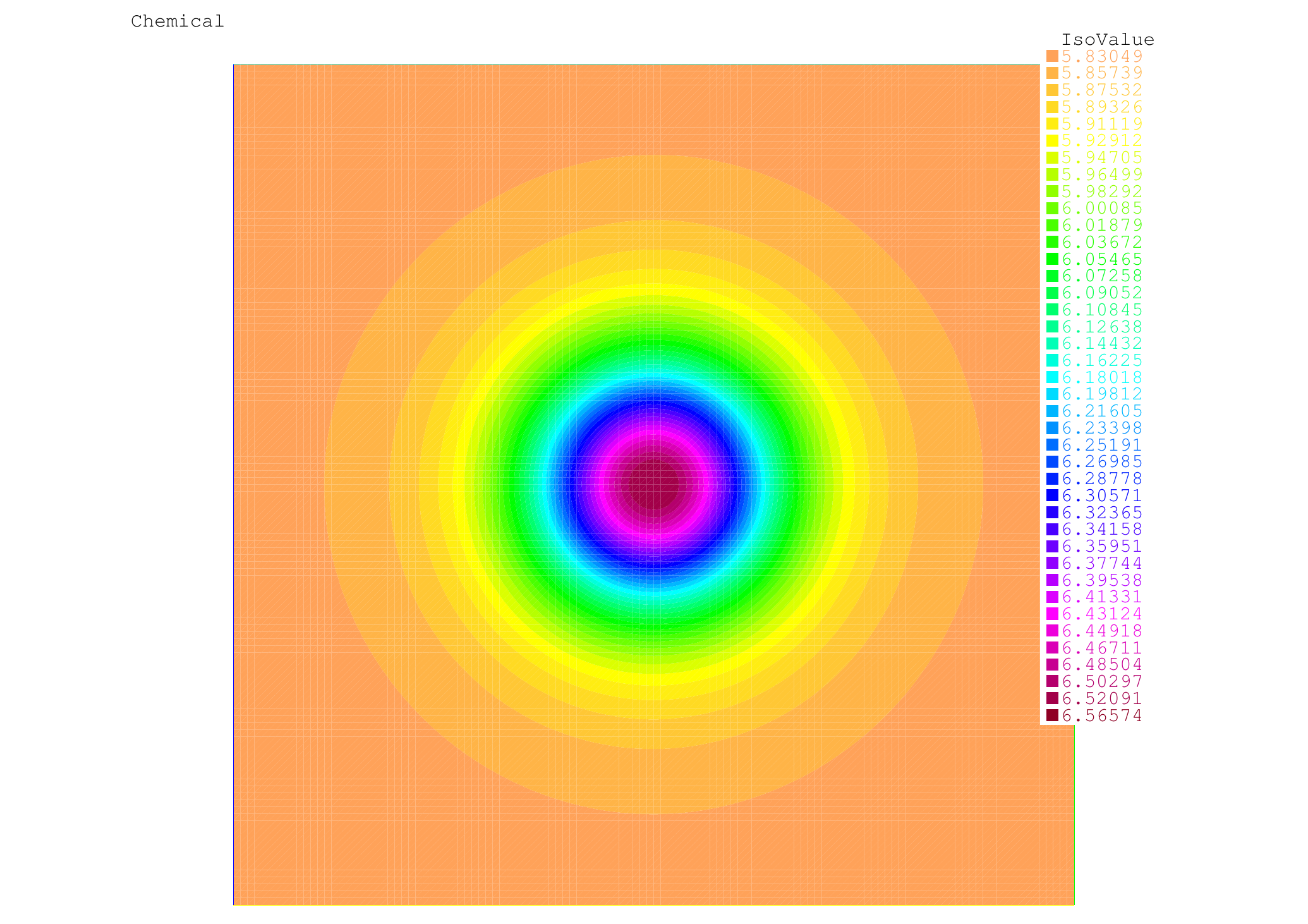}}\\
\subfloat[\label{chm2}]{\includegraphics[height=3cm,width=4cm]{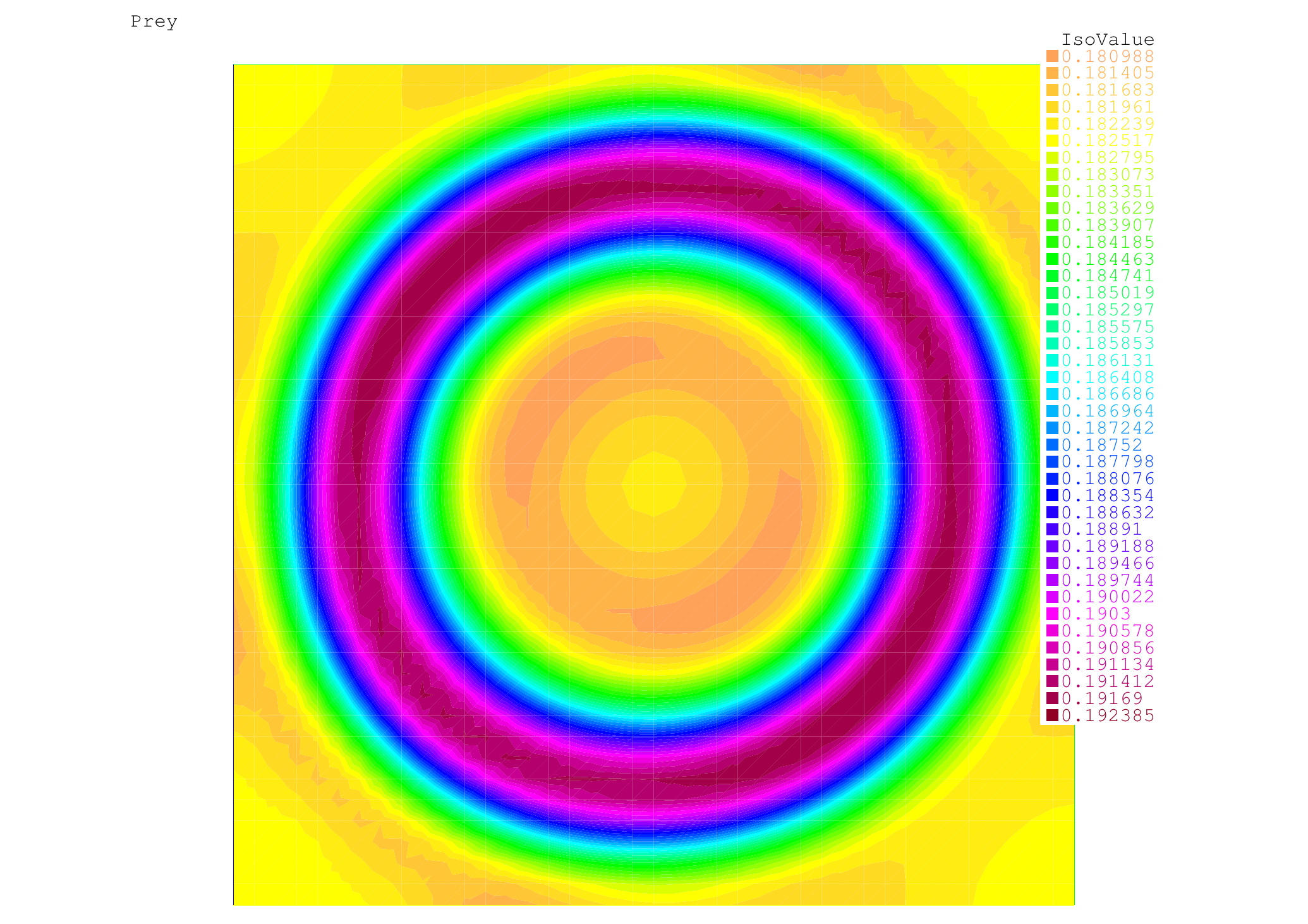}\includegraphics[height=3cm,width=4cm]{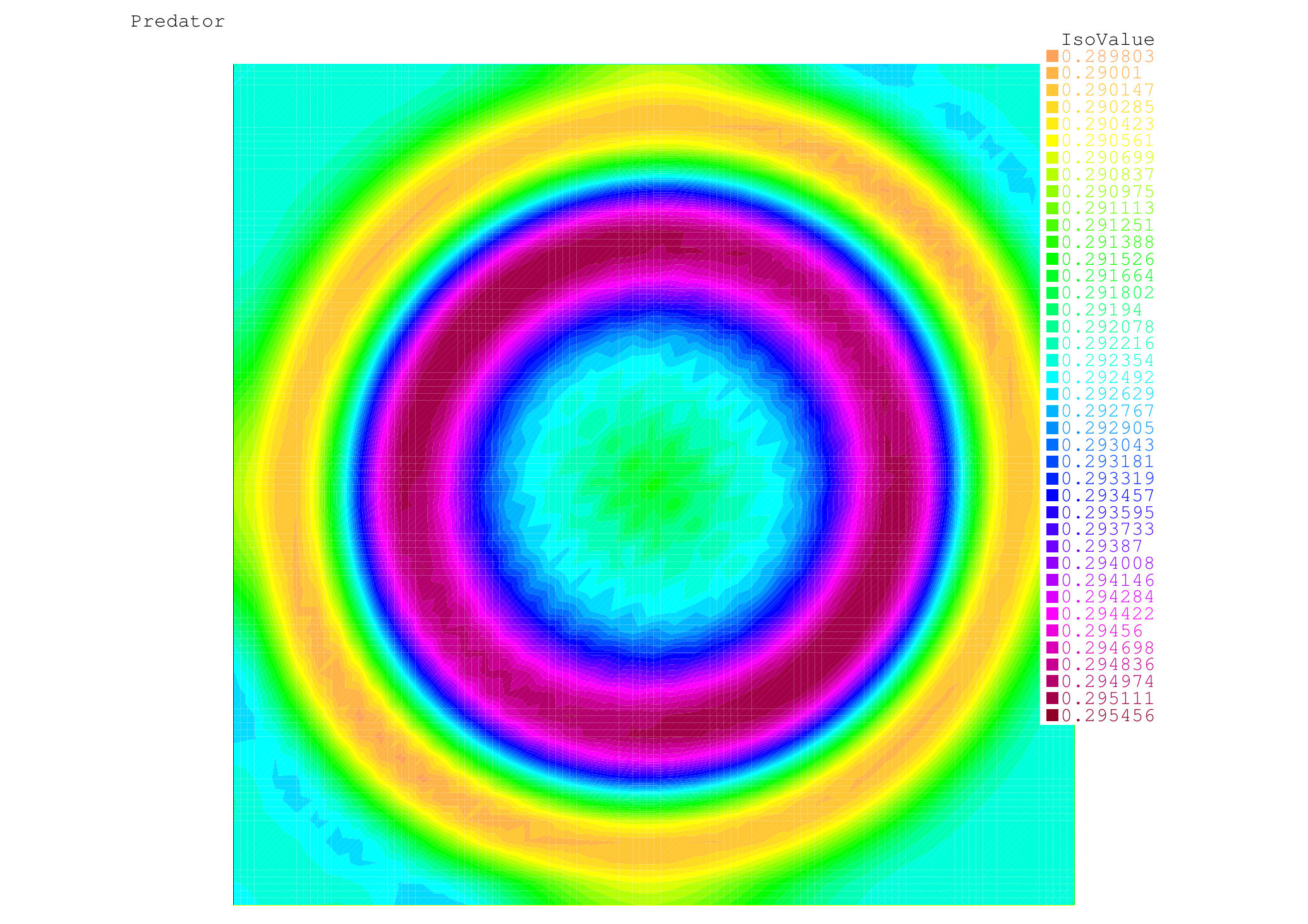} 	\includegraphics[height=3cm,width=4cm]{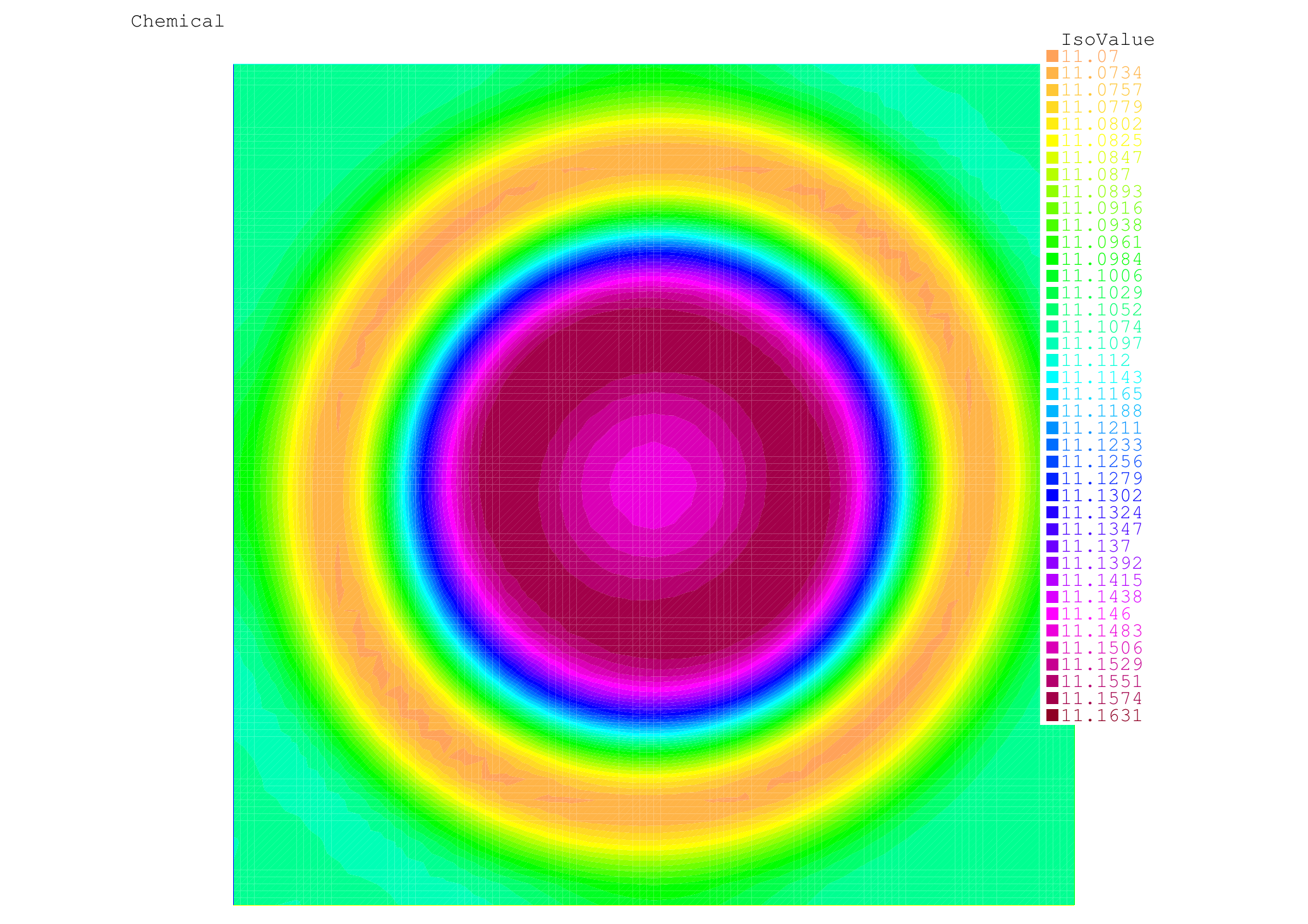} }\\
\subfloat[\label{chm3}]{\includegraphics[height=3cm,width=4cm]{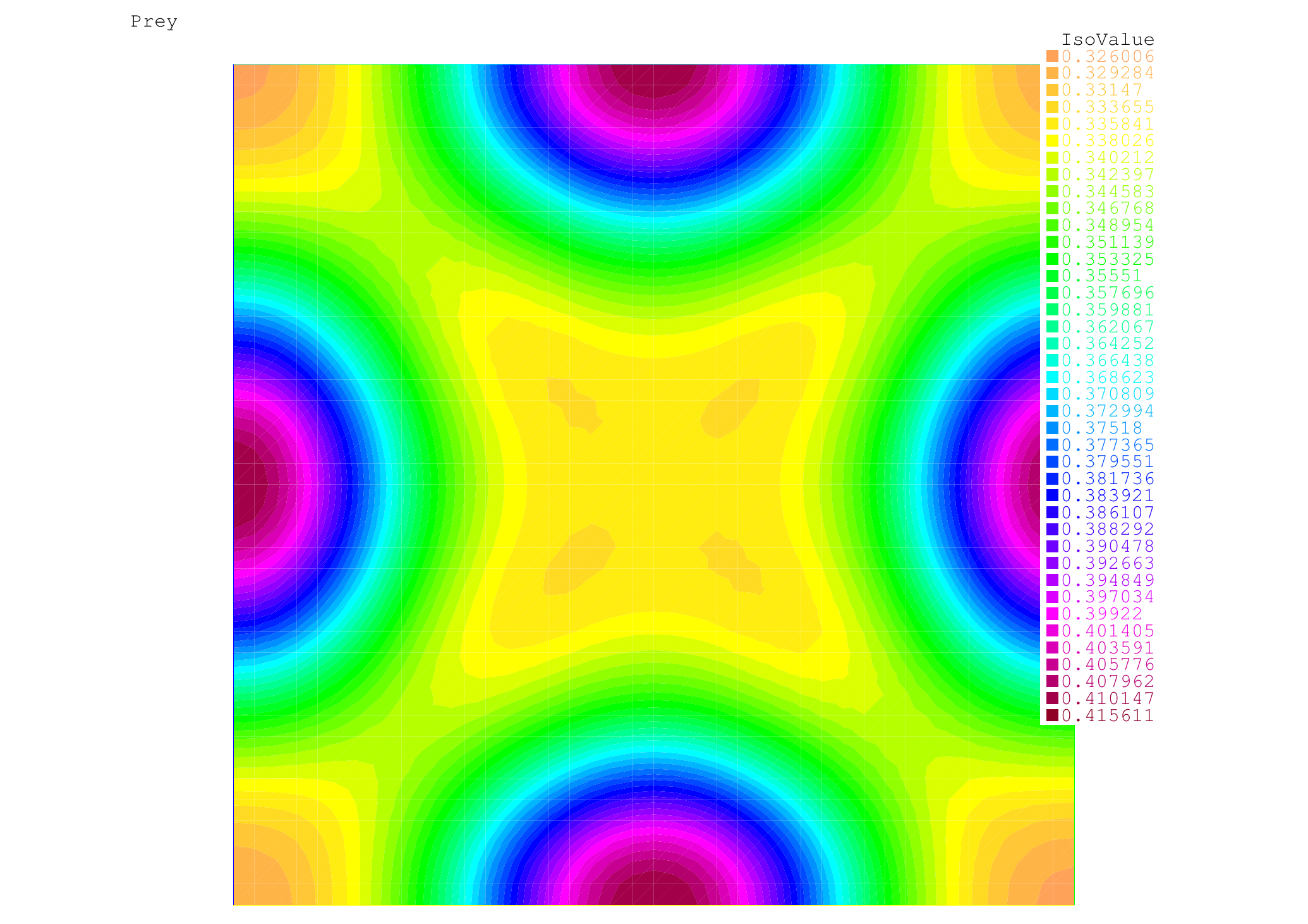}\includegraphics[height=3cm,width=4cm]{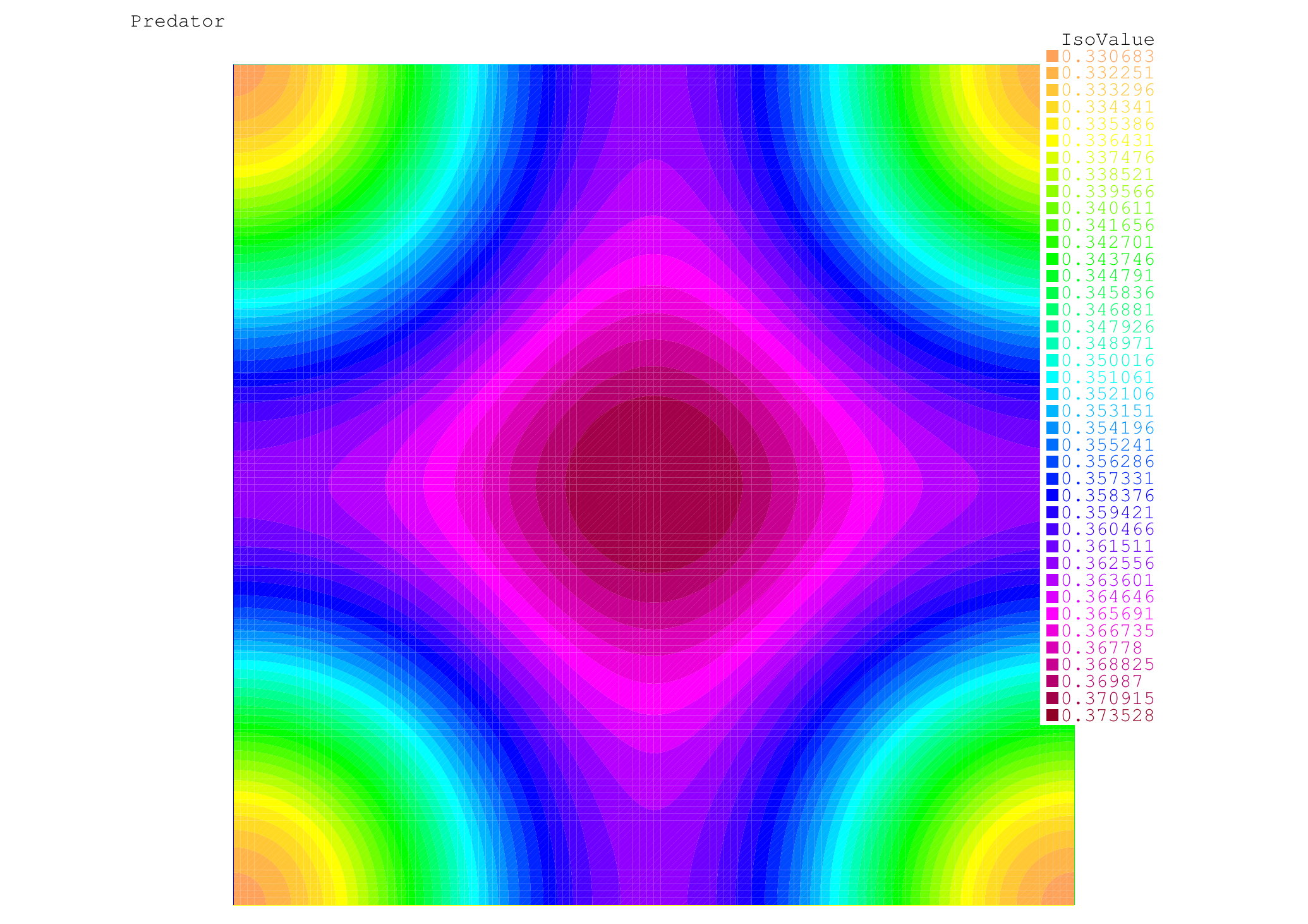} 	\includegraphics[height=3cm,width=4cm]{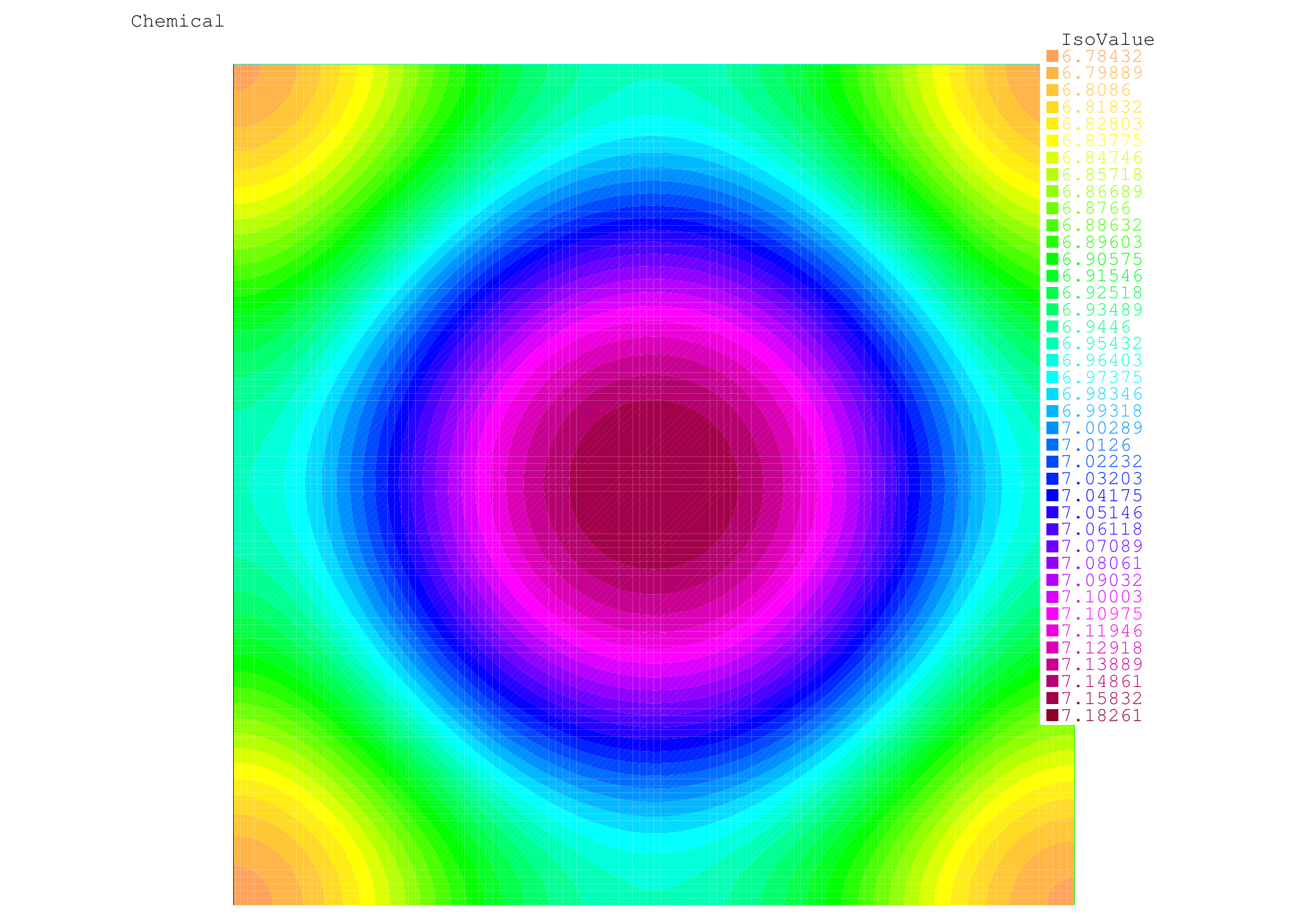}}\\
\subfloat[\label{chm4}]{\includegraphics[height=3cm,width=4cm]{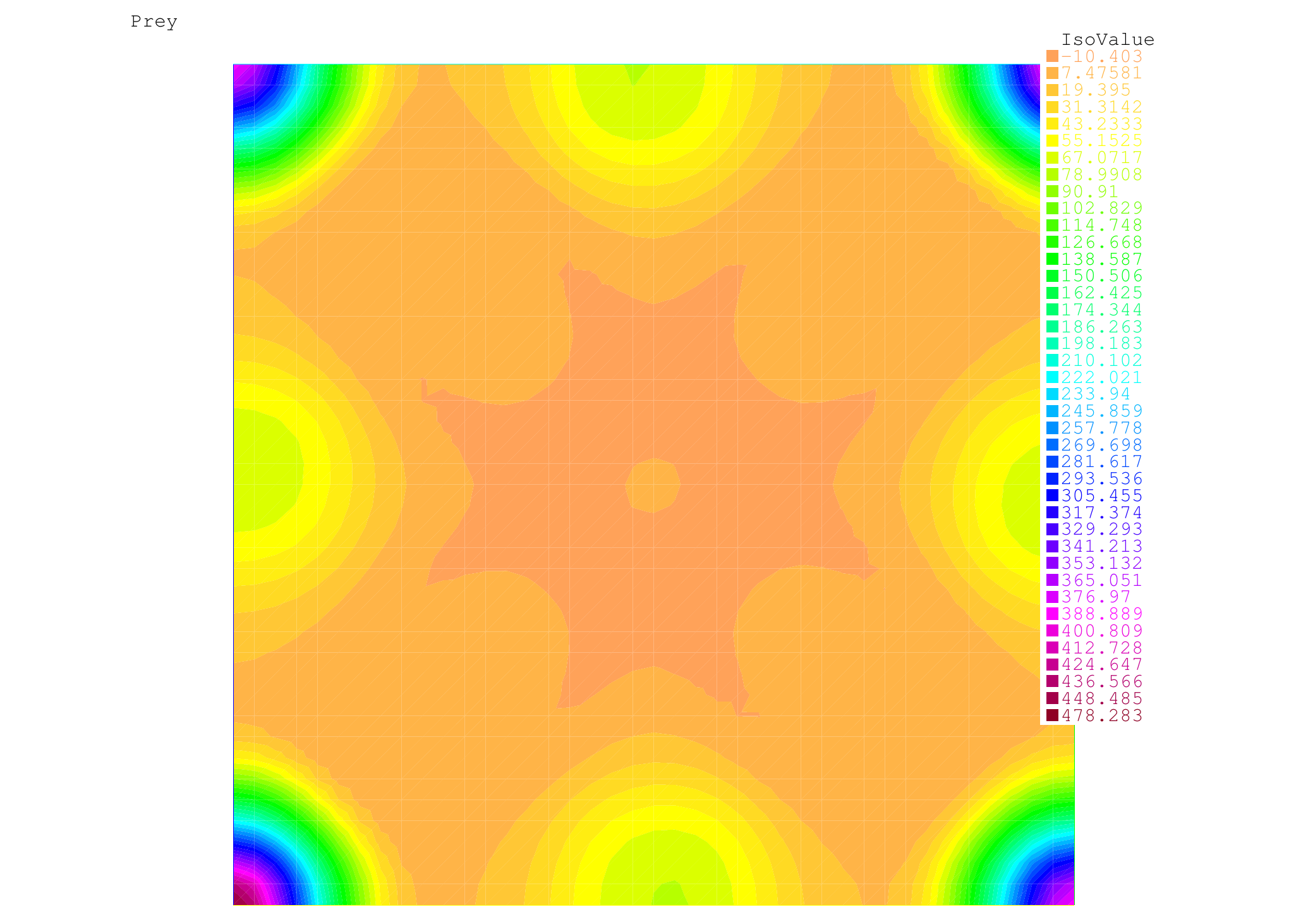}\includegraphics[height=3cm,width=4cm]{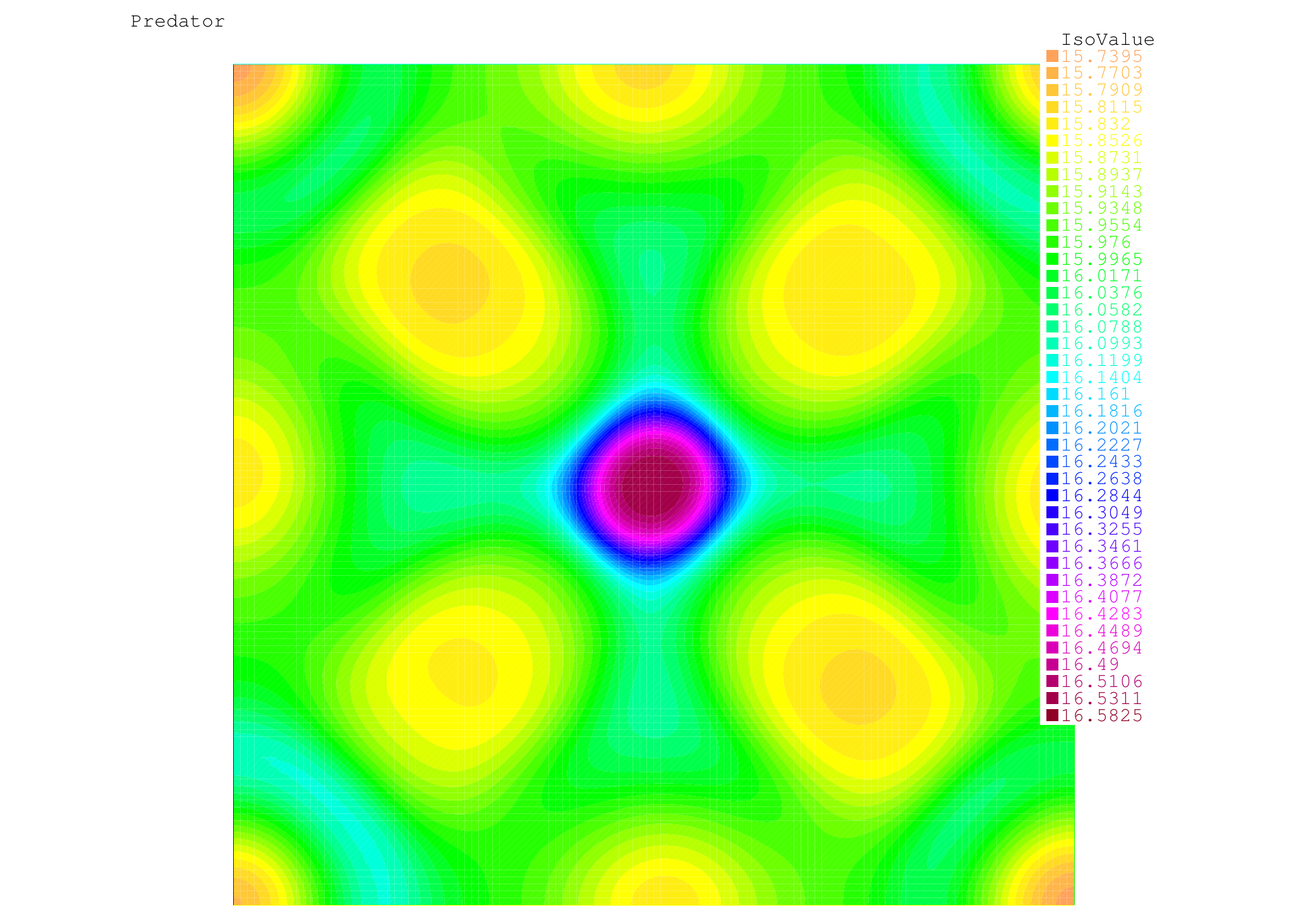} 	\includegraphics[height=3cm,width=4cm]{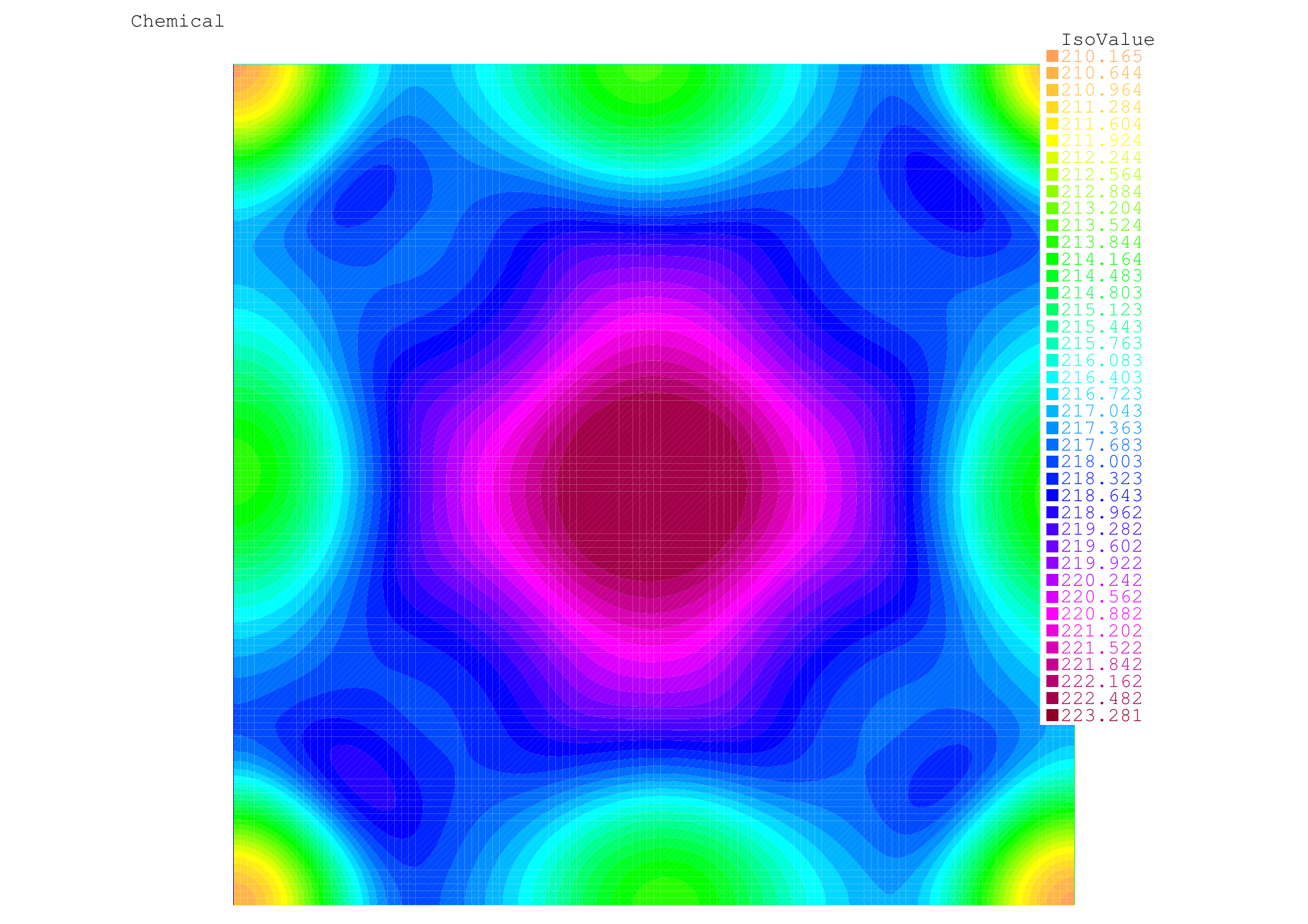}}\\
\subfloat[\label{chm5}]{\includegraphics[height=3cm,width=4cm]{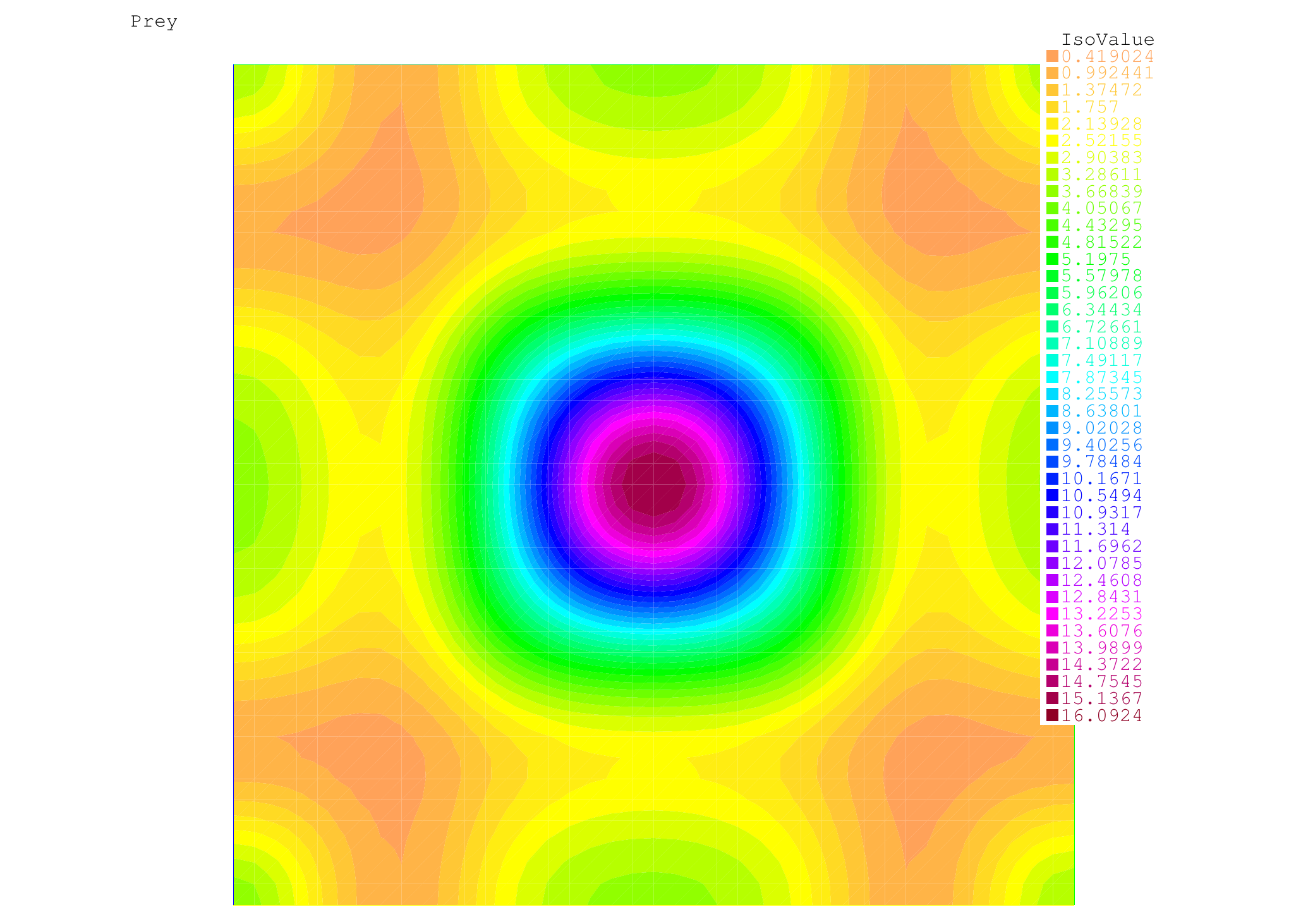}\includegraphics[height=3cm,width=4cm]{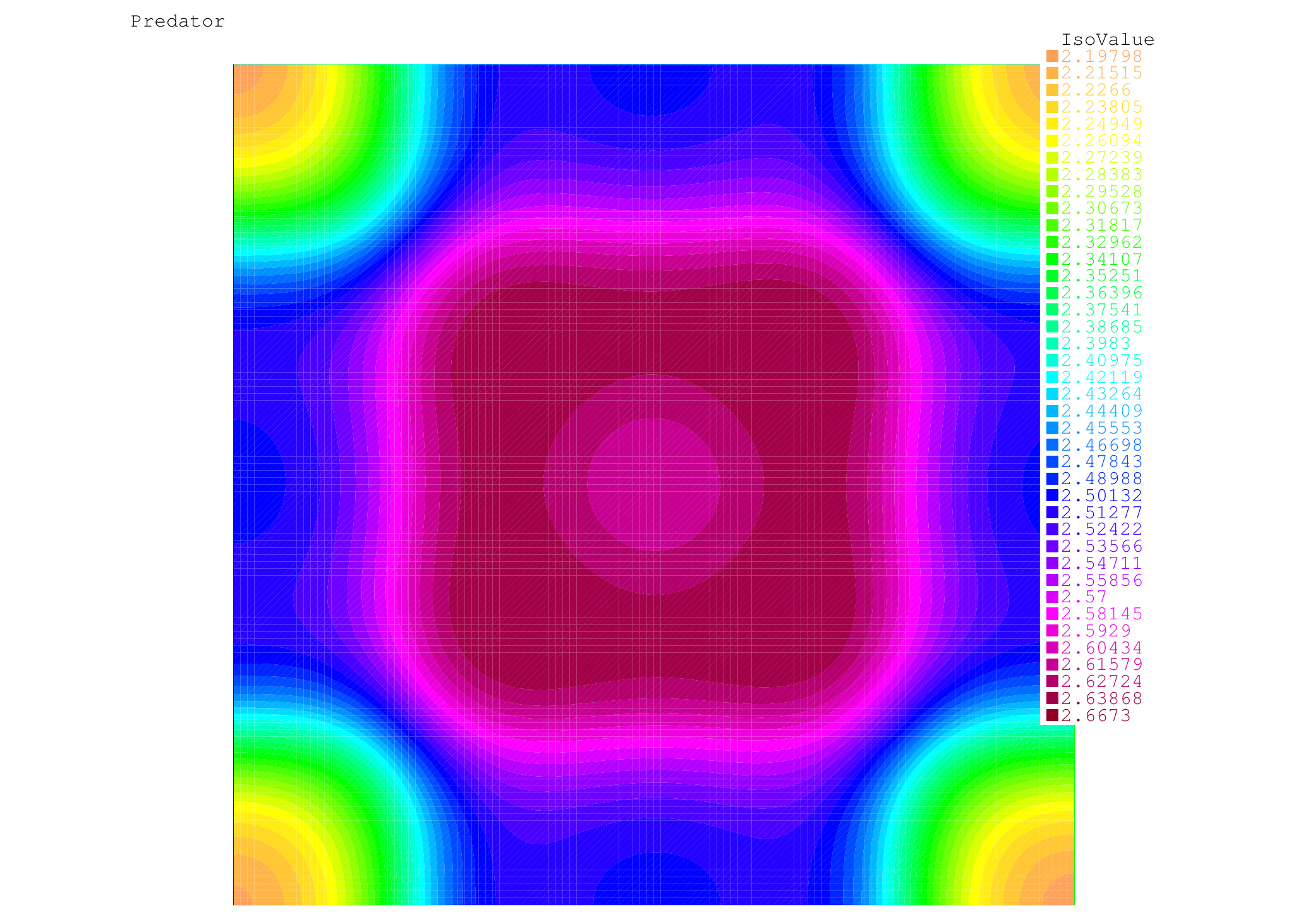} 	\includegraphics[height=3cm,width=4cm]{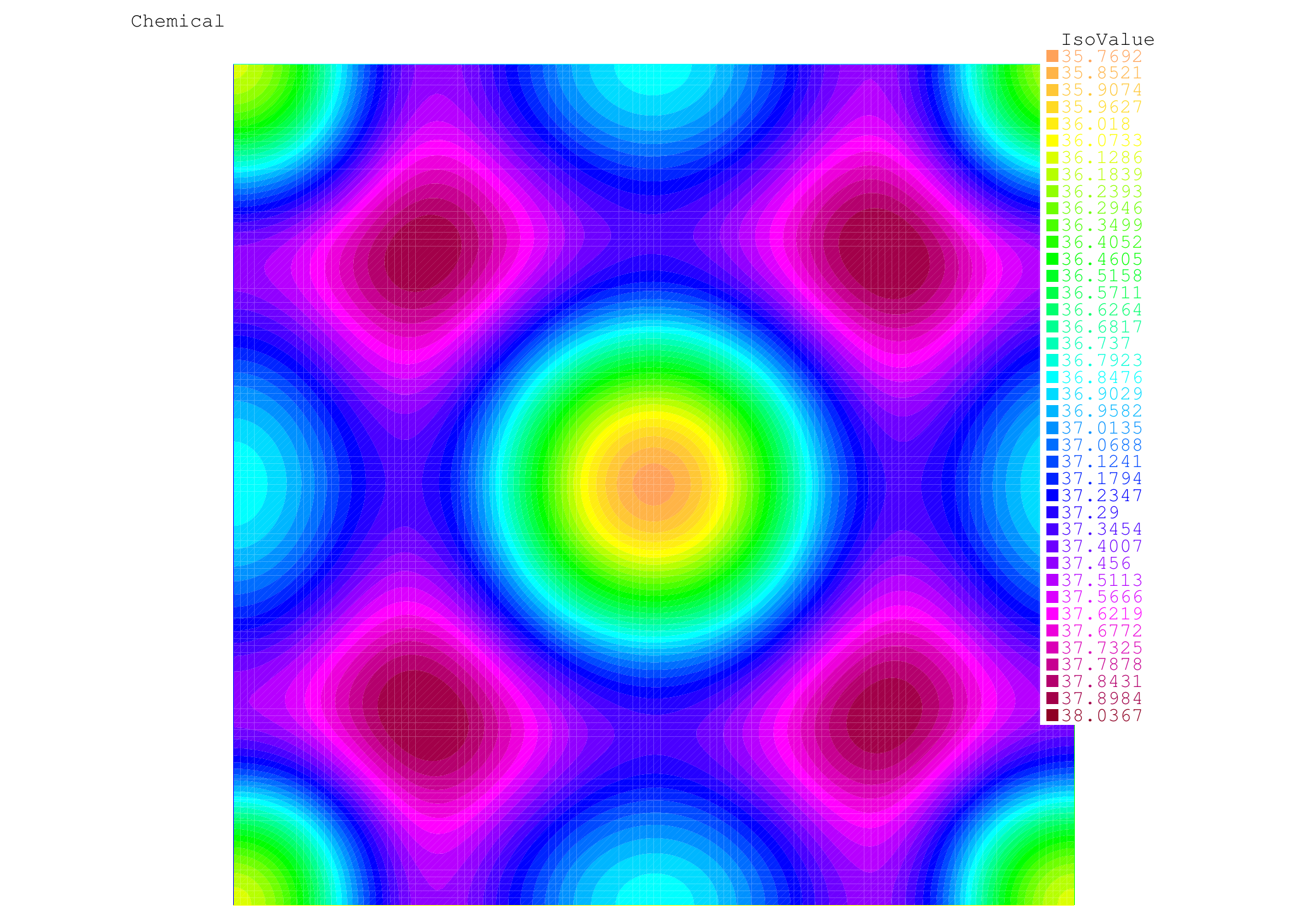}}
\caption{Model B: time dependent spatial patterns for the prey(left column), predator   and chemical when  chemo-repulsion is relatively large i.e $\chi=5$ with remaining parameters the same as in (\ref{para1})  and  Gaussian initial data for predator centered in the middle  the square with constant initial data for the prey $N=\bar{N}$ and for the chemical $W=\bar{W}$ at time steps (a)  $t=10$ (b) $t=500$ (c) $t=700$ (e) $t=1000$}
\label{chemo}
\end{figure}
\begin{figure}[hbt!] 
	\center
	\begin{subfigure}{0.3\textwidth}
		\includegraphics[width=\textwidth]{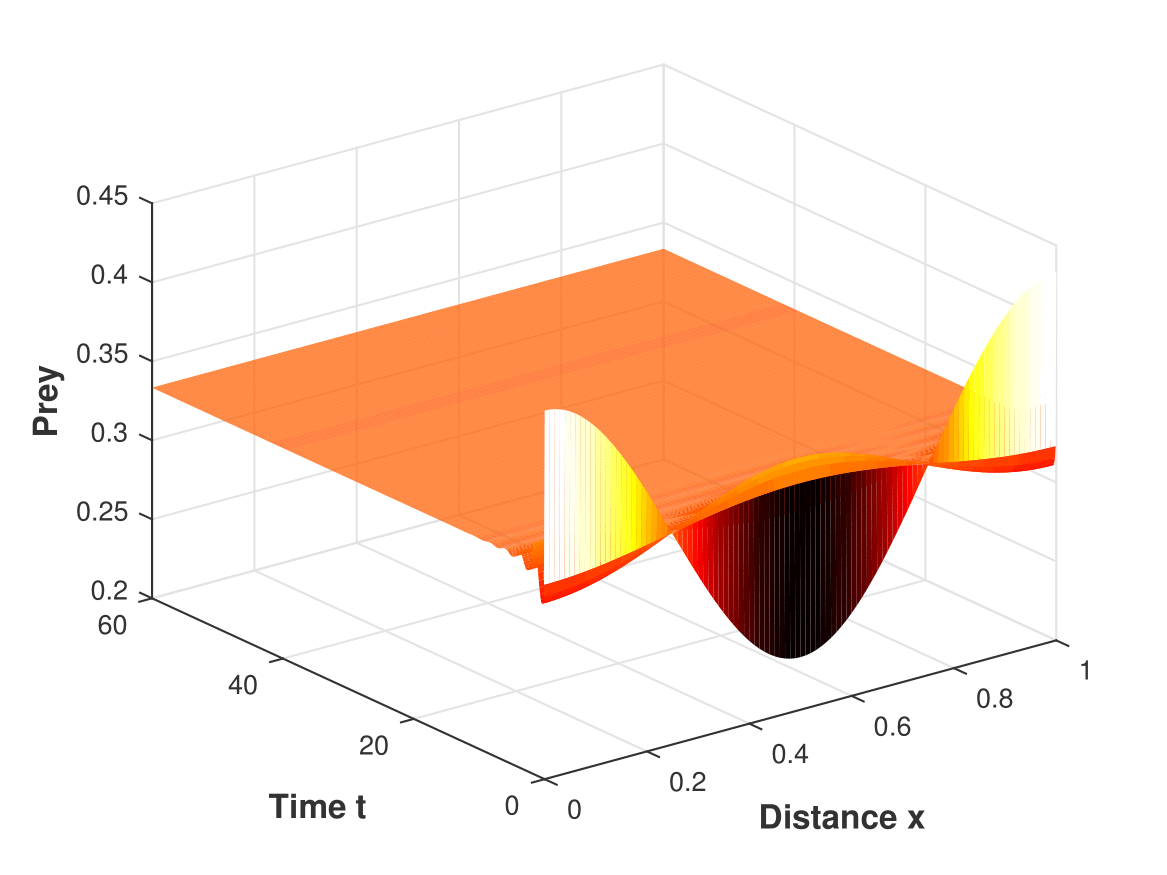} 
		\caption{}
		\label{fig91}
	\end{subfigure}
	\begin{subfigure}{0.3\textwidth} 
		\includegraphics[width=\textwidth]{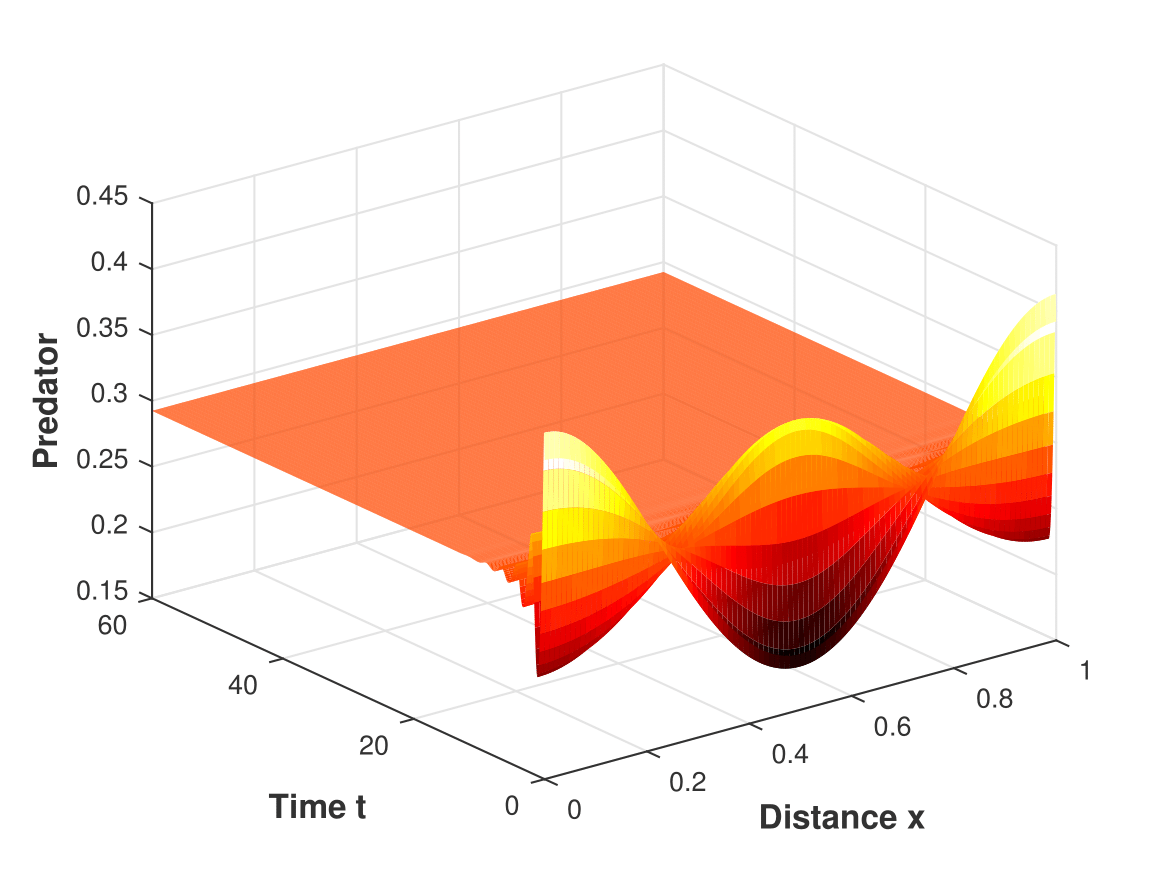}   
		\caption{}
		\label{fig92}
	\end{subfigure} 
	\begin{subfigure}{0.3\textwidth} 
		\includegraphics[width=\textwidth]{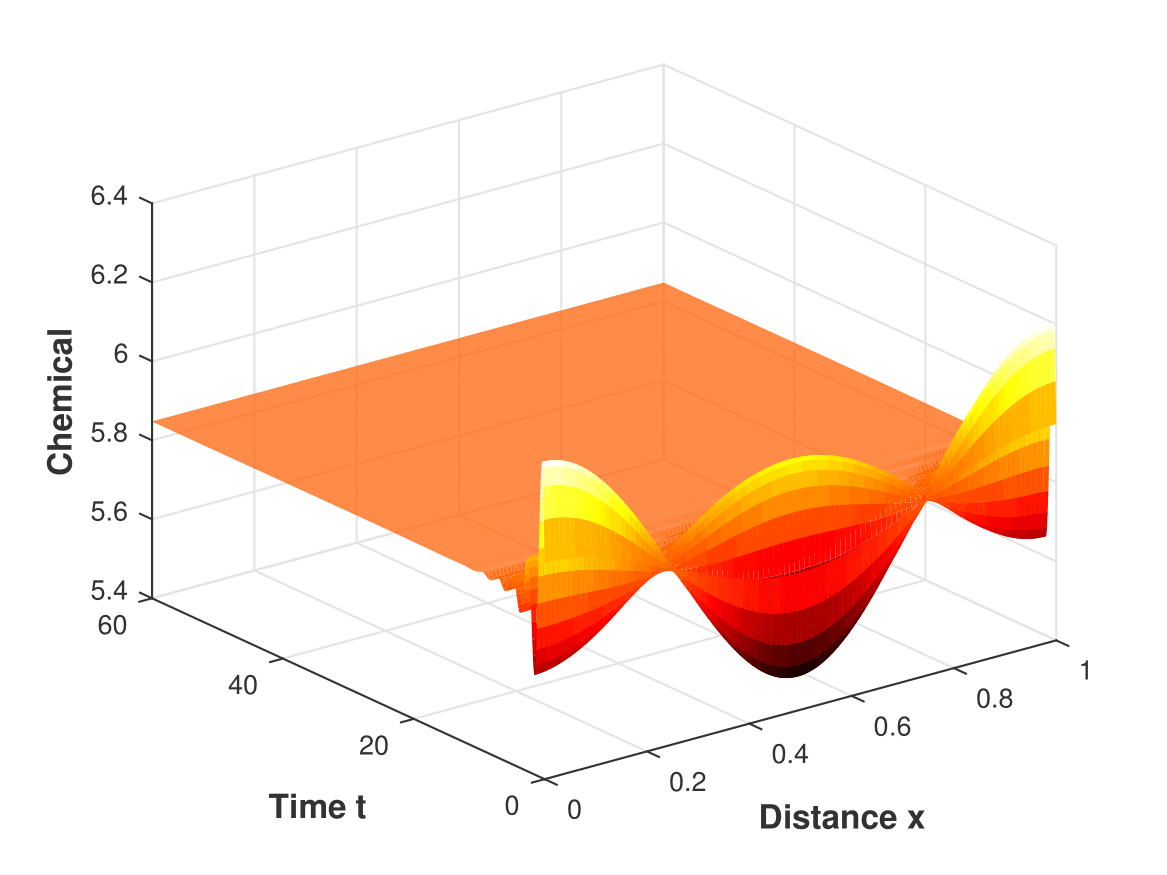}  
		\caption{}
		\label{fig93}
	\end{subfigure}  
	\caption{Model A: spatial perturbation converges to constant steady state $\bar{E}$ for $\chi=0.2$ and $\xi=1$ and other parameters are as in (\ref{para1}).}
	\label{fig9}
\end{figure}
\begin{figure}[hbt!]  
\centering 
\subfloat[\label{fig991}]{\includegraphics[height=4cm,width=4cm]{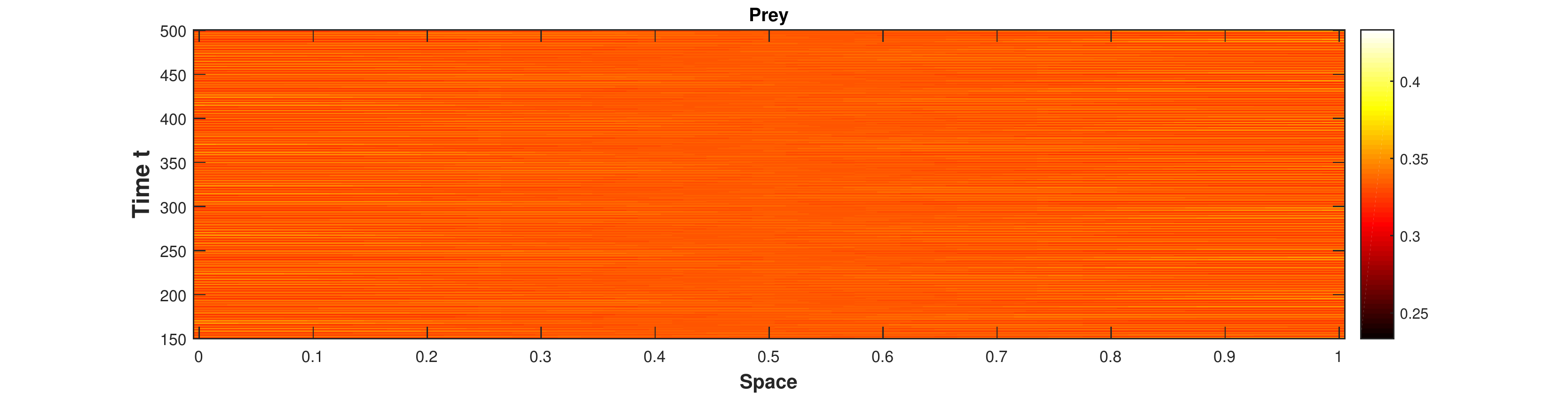}}
\subfloat[\label{fig992}]{\includegraphics[height=4cm,width=4cm]{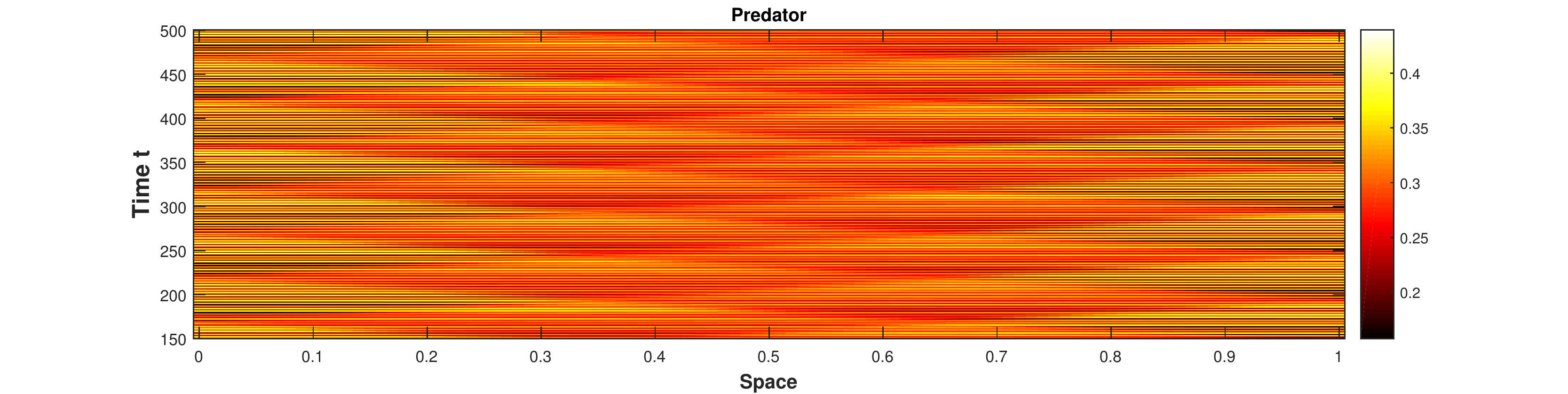}}
\subfloat[\label{fig993}]{\includegraphics[height=4cm,width=4cm]{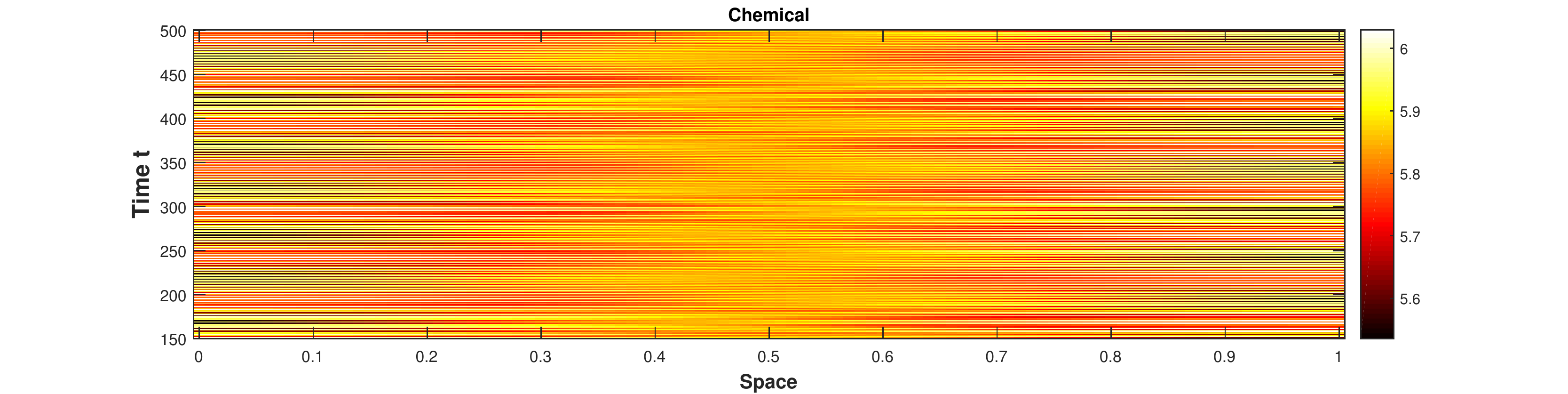}}
\caption{Model A: space-time patterns when chemo-repulsion is weaker than prey-taxis (i.e. $\chi<\xi$)  $\chi=0.2,\ \xi=10>\xi^S$  for symmetrical initial data (\ref{idata}) with $j=4$.}
\label{fig99}
\end{figure}
\begin{figure}[hbt!]  
\centering 
\subfloat[\label{fig101}]{\includegraphics[height=4cm,width=4cm]{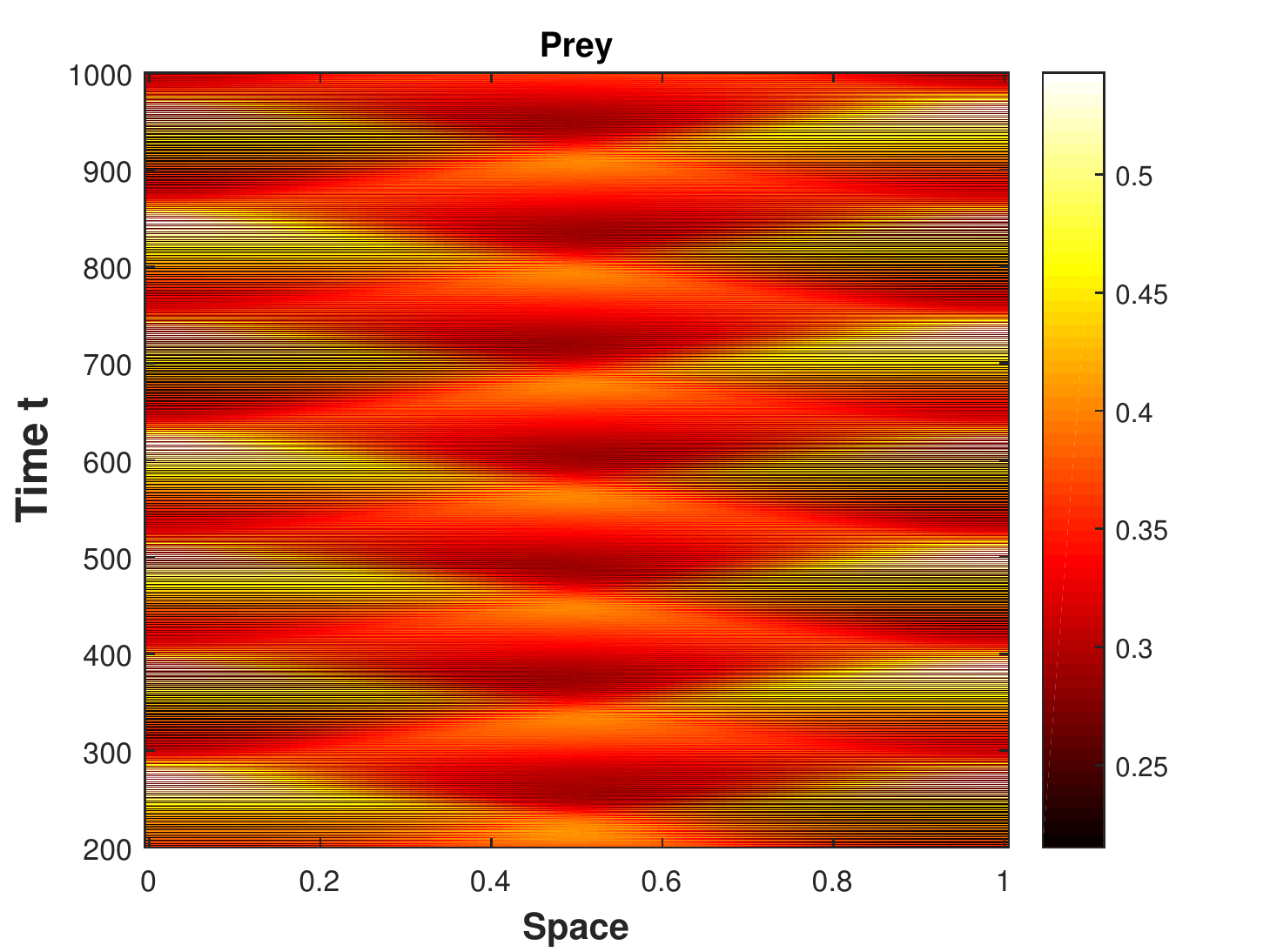}}
\subfloat[\label{fig102}]{\includegraphics[height=4cm,width=4cm]{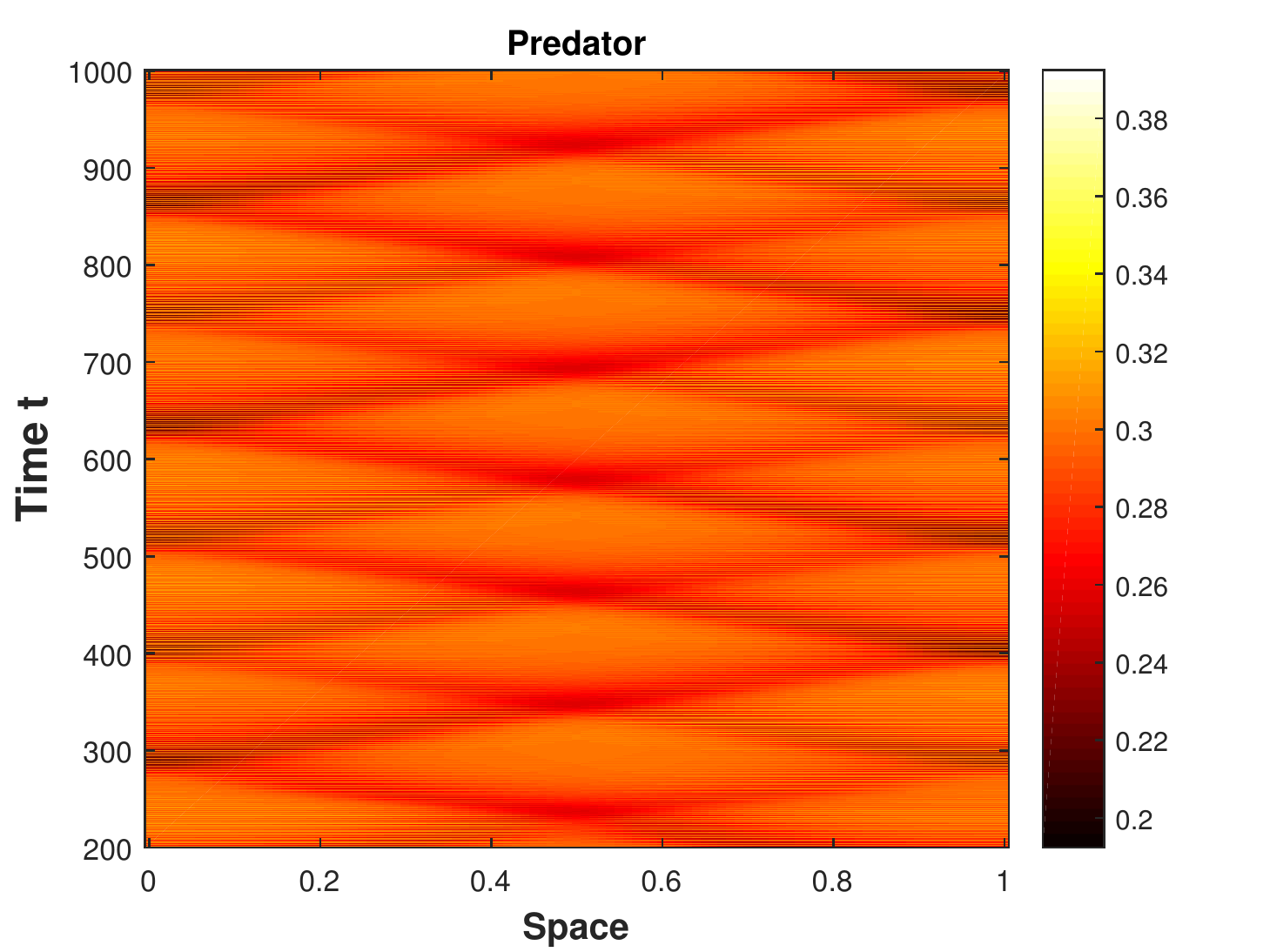}}
\subfloat[\label{fig103}]{\includegraphics[height=4cm,width=4cm]{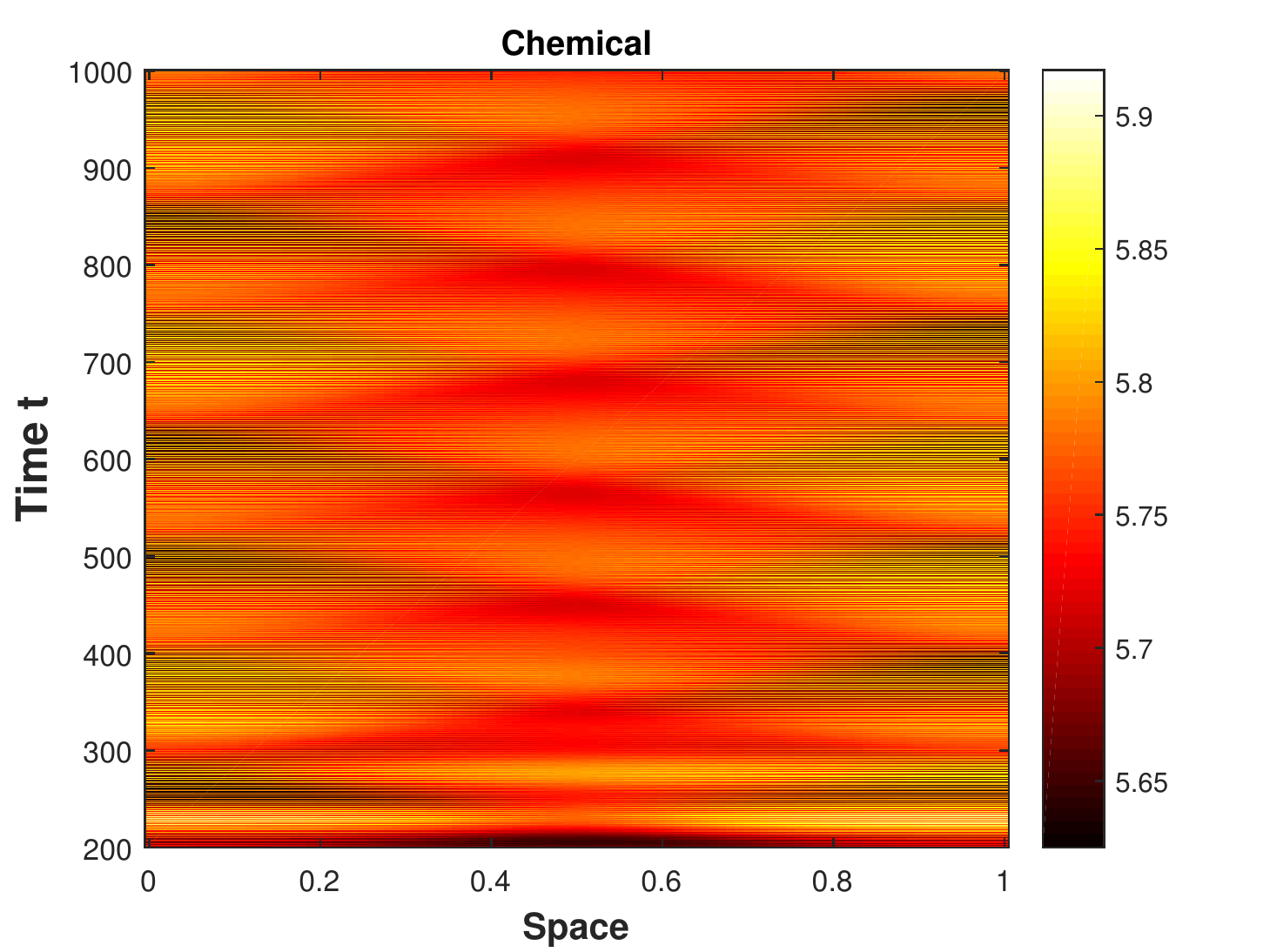}}
\caption{Model A: space-time patterns when chemo-repulsion is stronger than prey-taxis (i.e. $\chi>\xi$)  $\chi=5,\ \xi=0.2$  for symmetrical initial data (\ref{idata}) with $j=4$ and $L=1$.}
\label{fig10}
\end{figure}

\begin{figure}[hbt!]  
\centering 
\subfloat[\label{fig131}]{\includegraphics[width=0.35\textwidth]{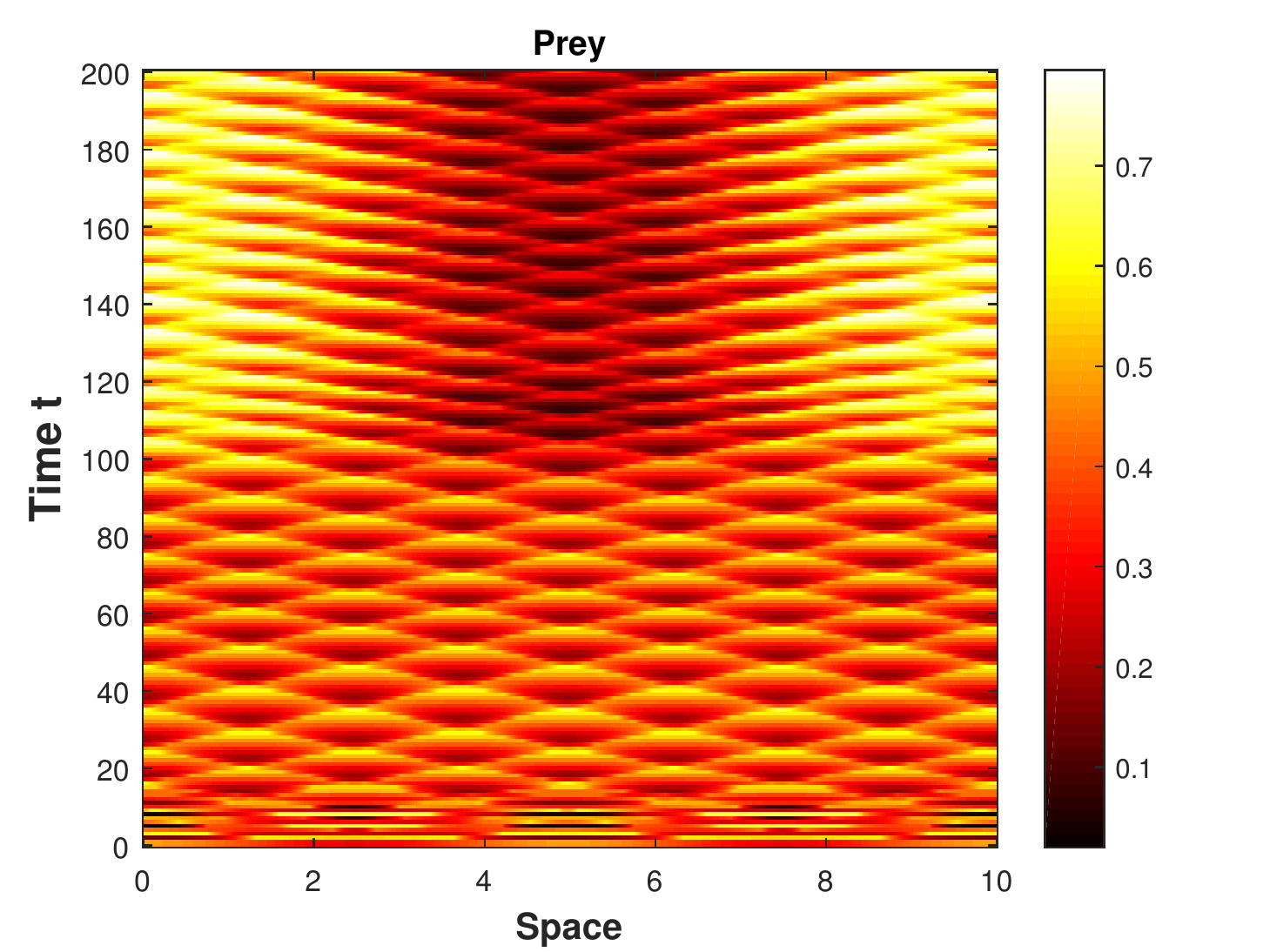}\includegraphics[width=0.35\textwidth]{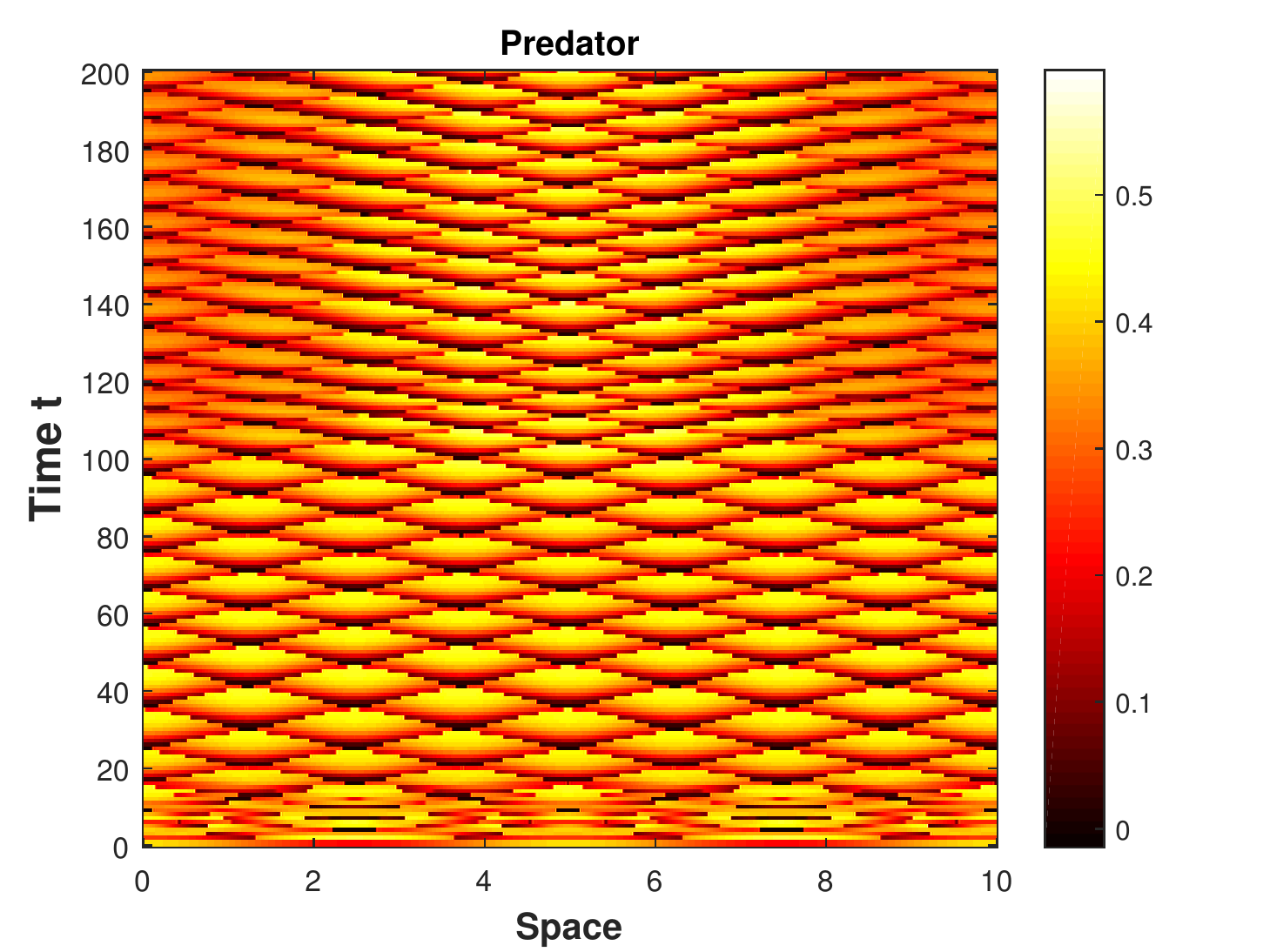}}\\
\subfloat[\label{fig132}]{\includegraphics[width=0.35\textwidth]{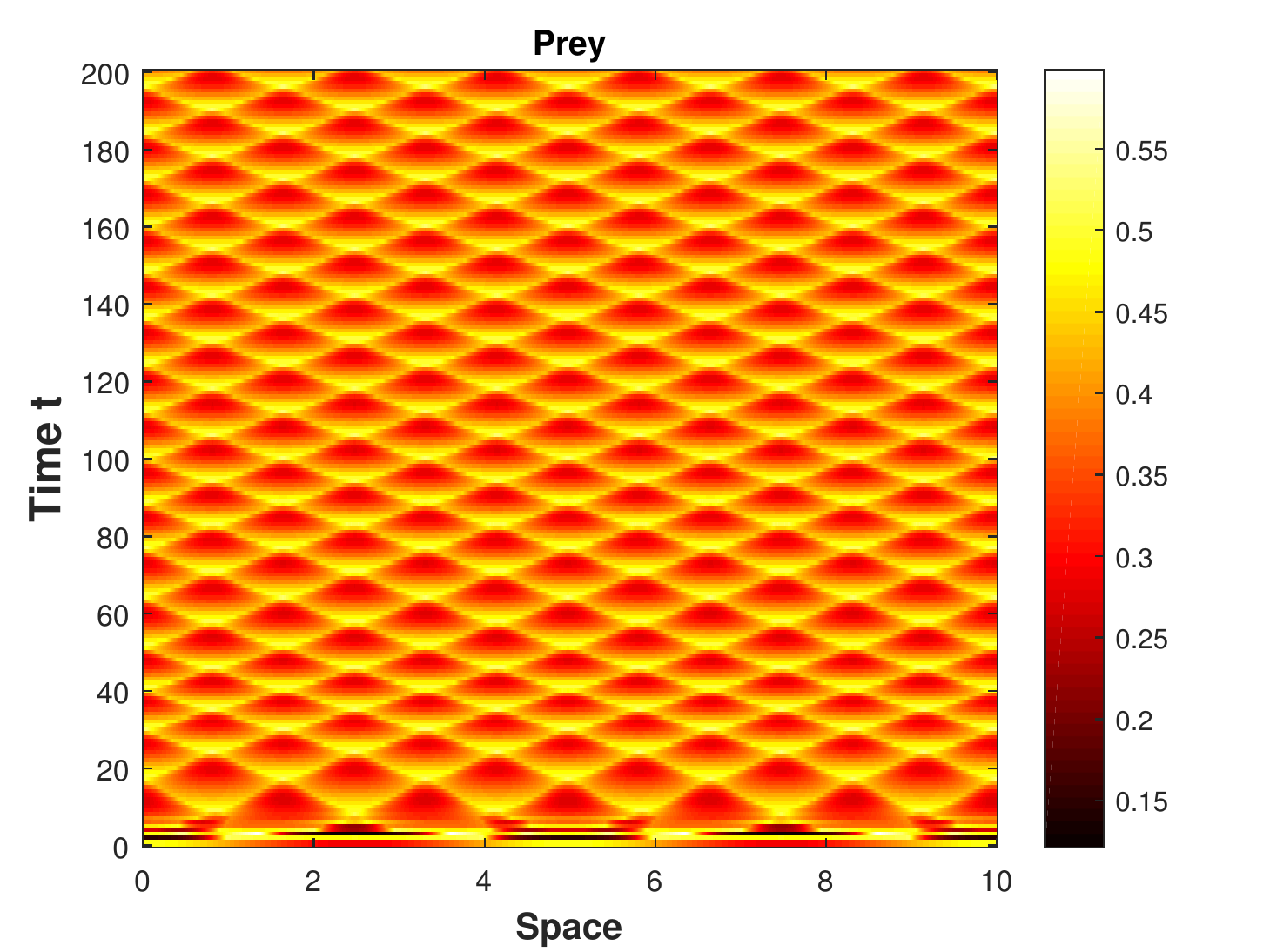}\includegraphics[width=0.35\textwidth]{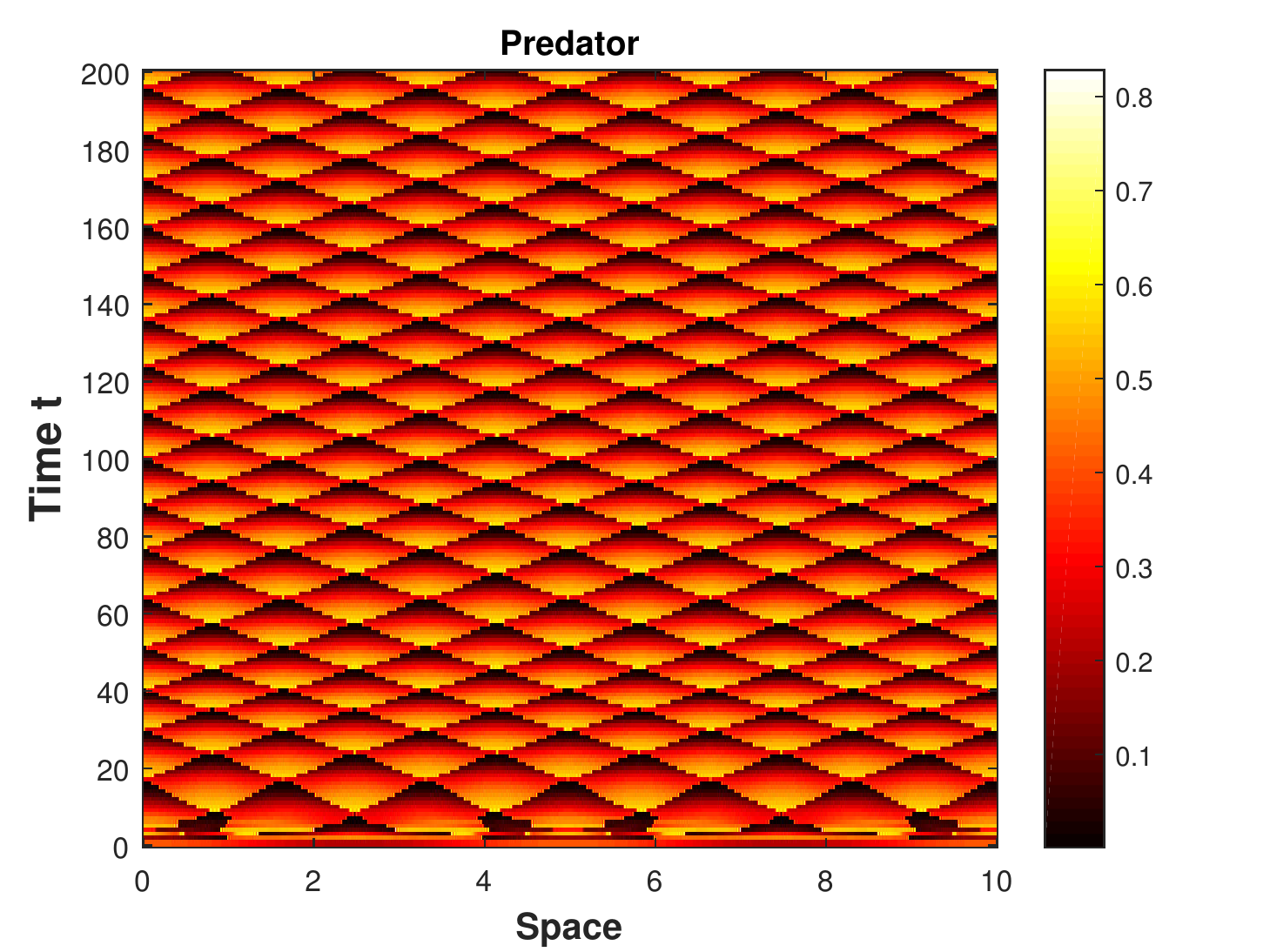}}\\
\caption{Model A: space-time patterns when prey-taxis sensitivity coefficient is stronger than that  of chemo-repulsion i.e  $\chi=2$ and (a) $\xi=5$ (b) $\xi=8$.}
\label{figa13}
\end{figure}
\begin{figure}[hbt!]  
\center
\subfloat[\label{bo1}]{ \includegraphics[width=0.35\textwidth]{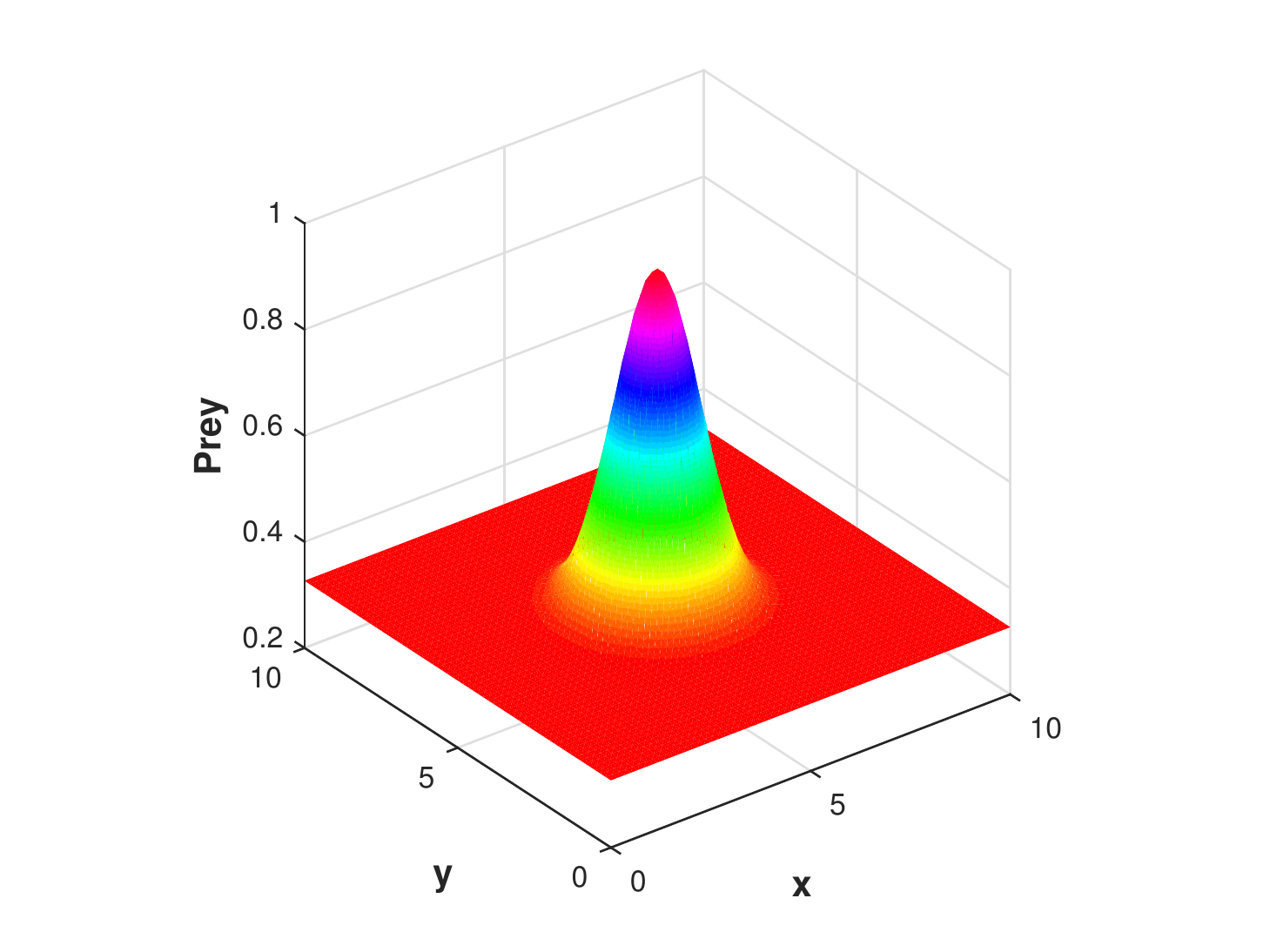} 
\includegraphics[width=0.35\textwidth]{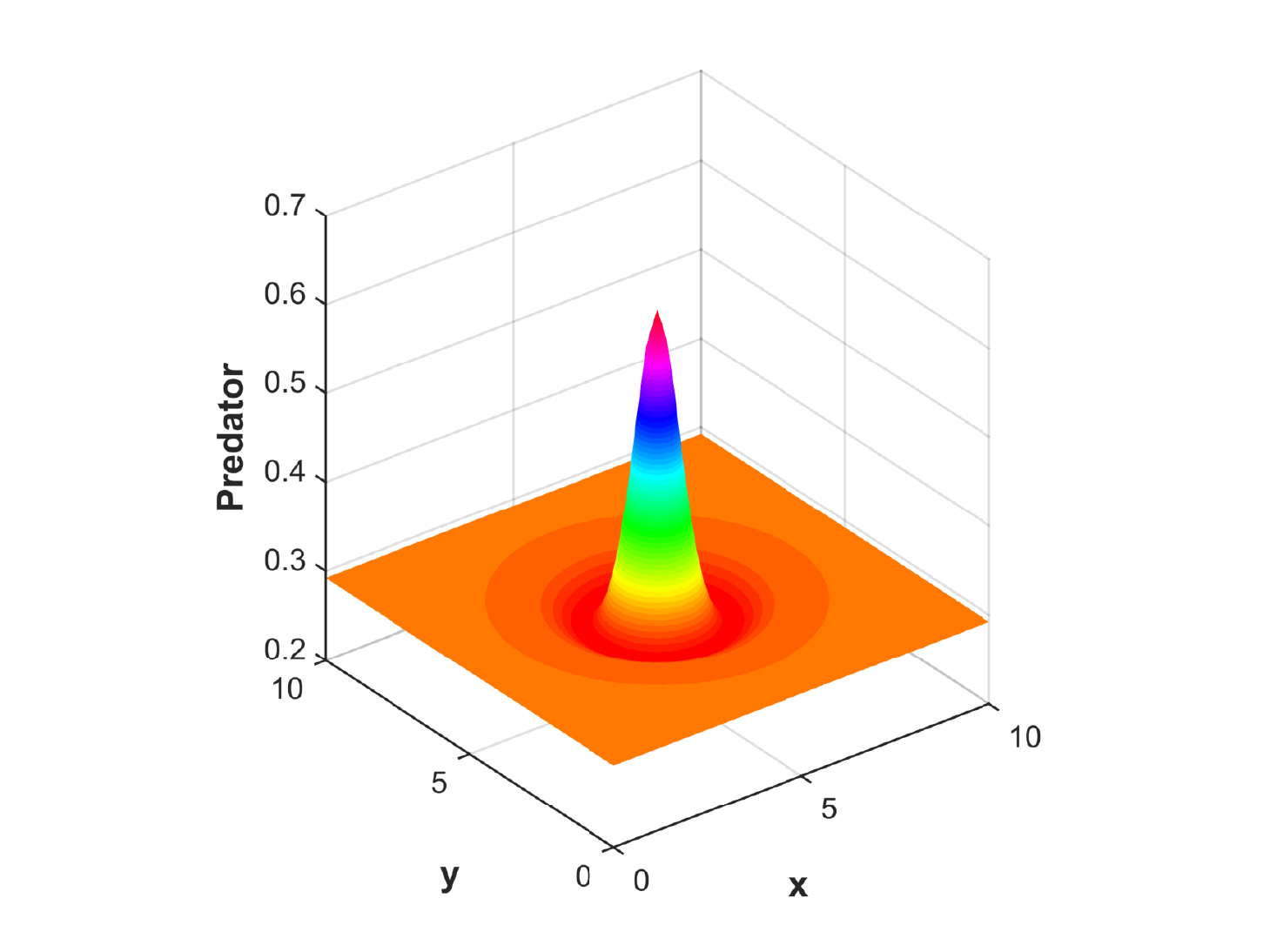} \includegraphics[width=0.35\textwidth]{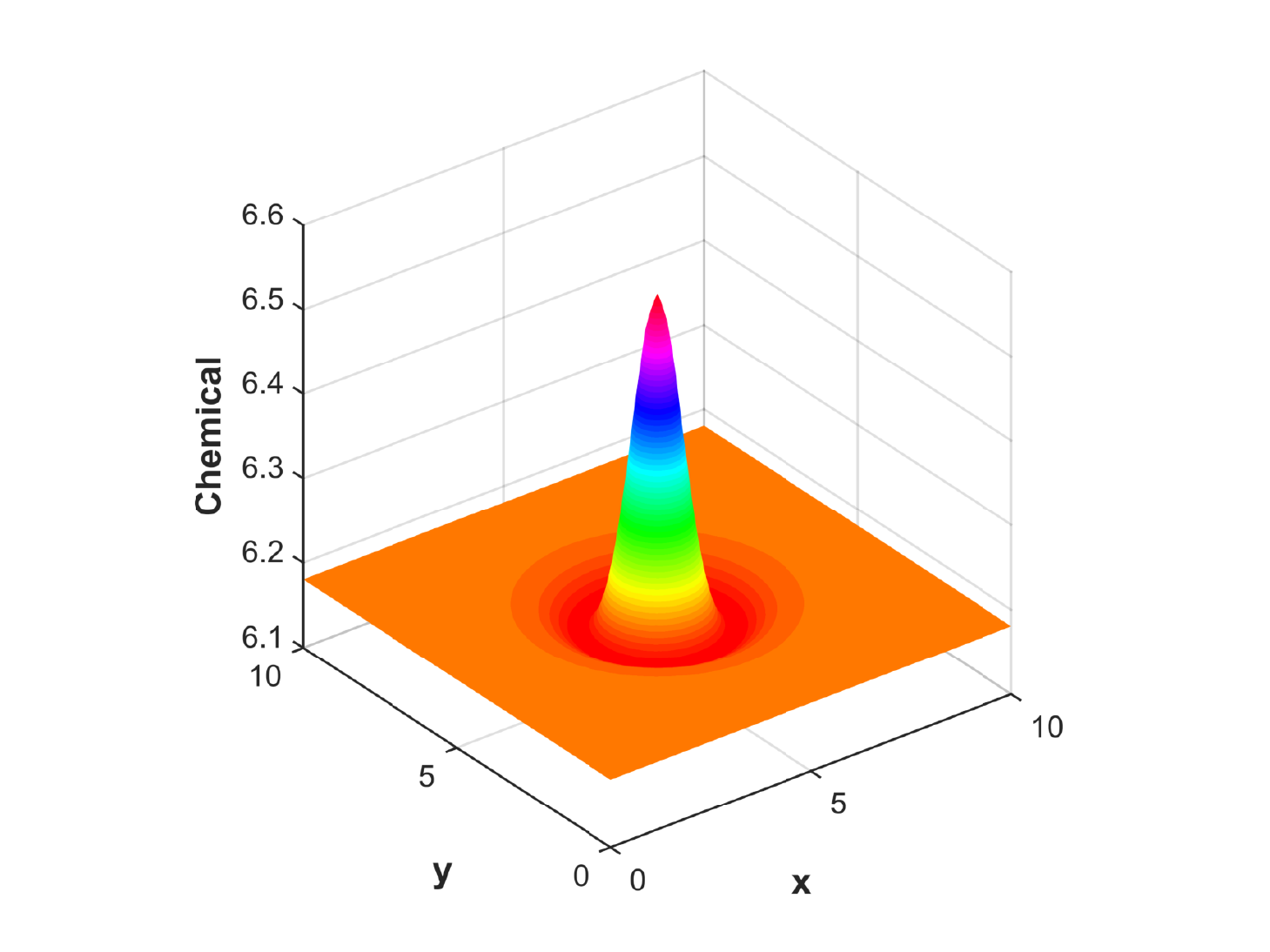}}\\
\subfloat[\label{bo3}]{	\includegraphics[width=0.35\textwidth]{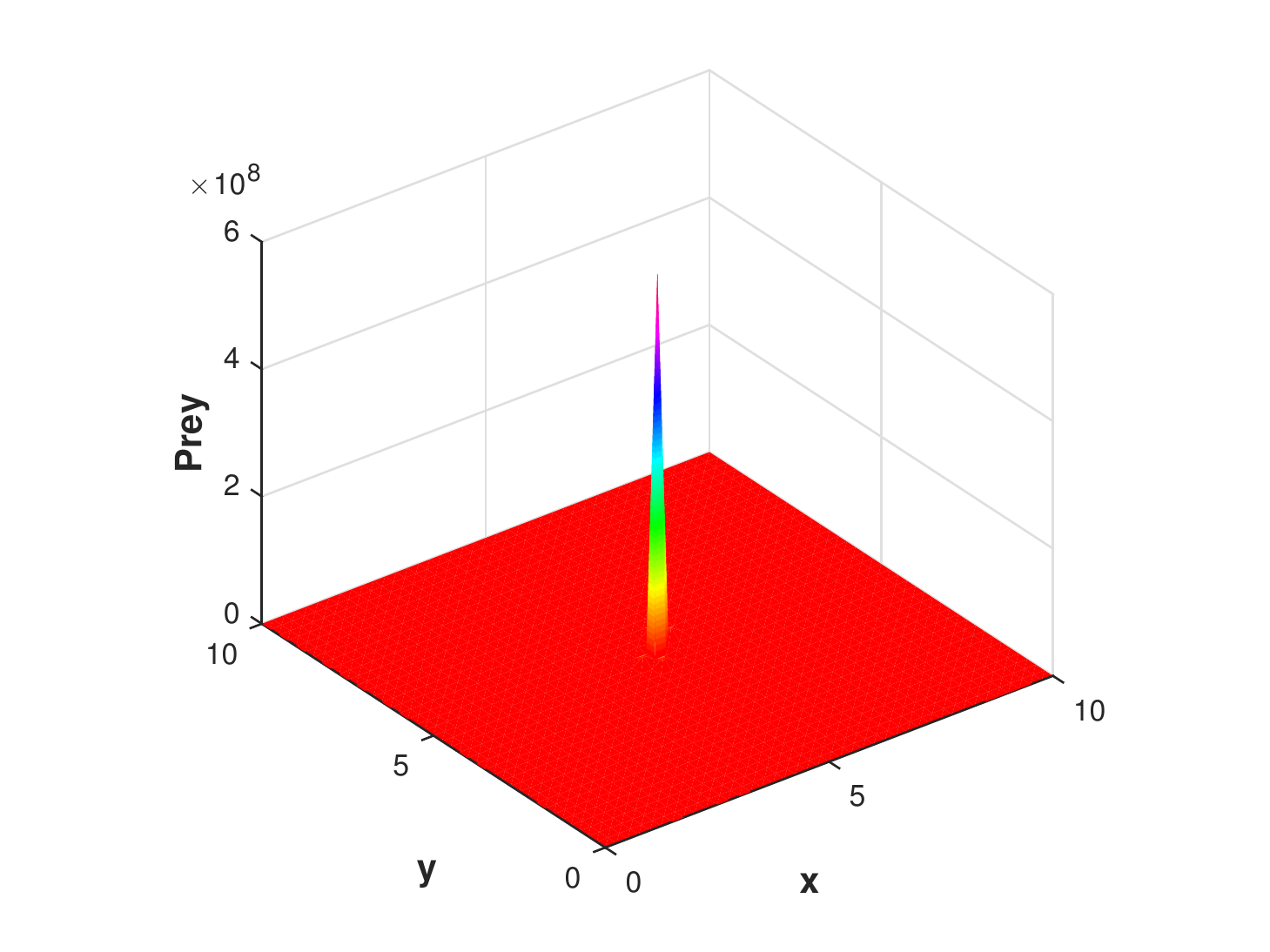} 
\includegraphics[width=0.35\textwidth]{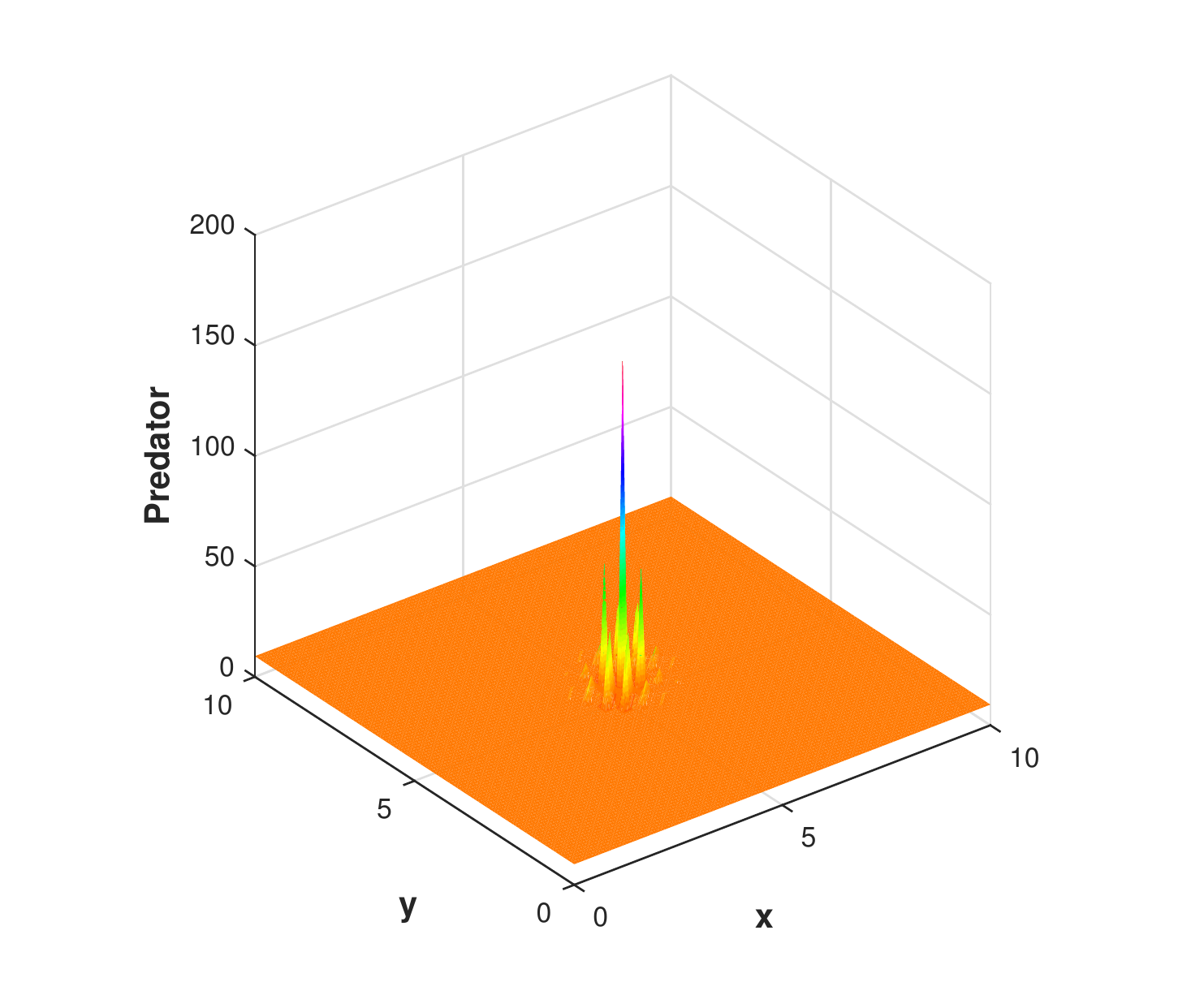}\includegraphics[width=0.35\textwidth]{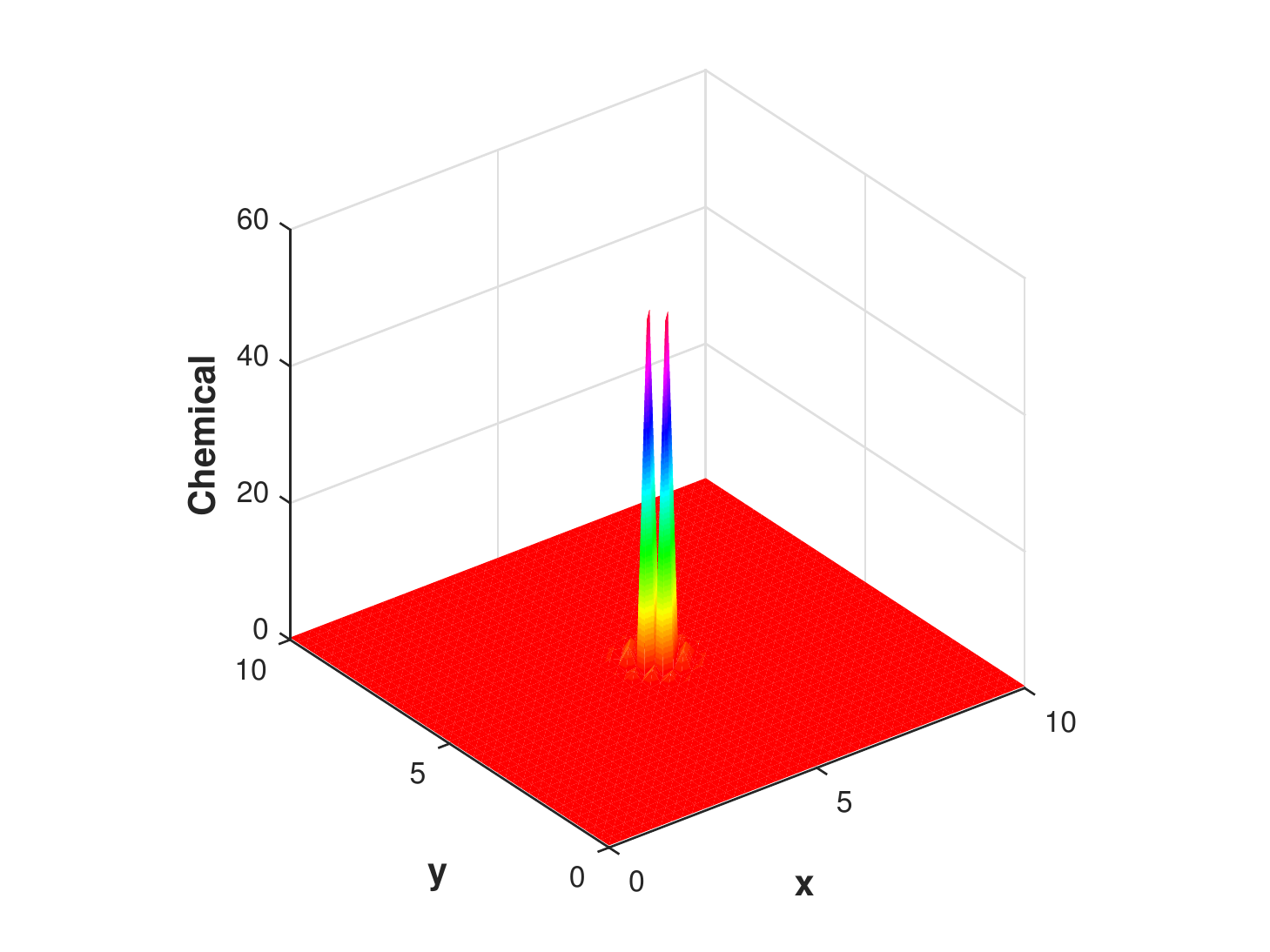}}	
\caption{Model A: numerical indication of finite-time blow-up for  $\chi=0.5$ and  $ \xi=10$  with remaining parameters the same as in (\ref{para1}) and Gaussian initial data for predator and prey centered in the middle of the square with constant initial data $\bar{W}$ for the chemical. Snapshots are presented at time steps (a) $t=10$ and (b) $t=134$.}
\label{blowup}
\end{figure}
\begin{figure*}[hbt!]  
\centering 
\subfloat[\label{chixi2}]{ \includegraphics[width=0.35\textwidth]{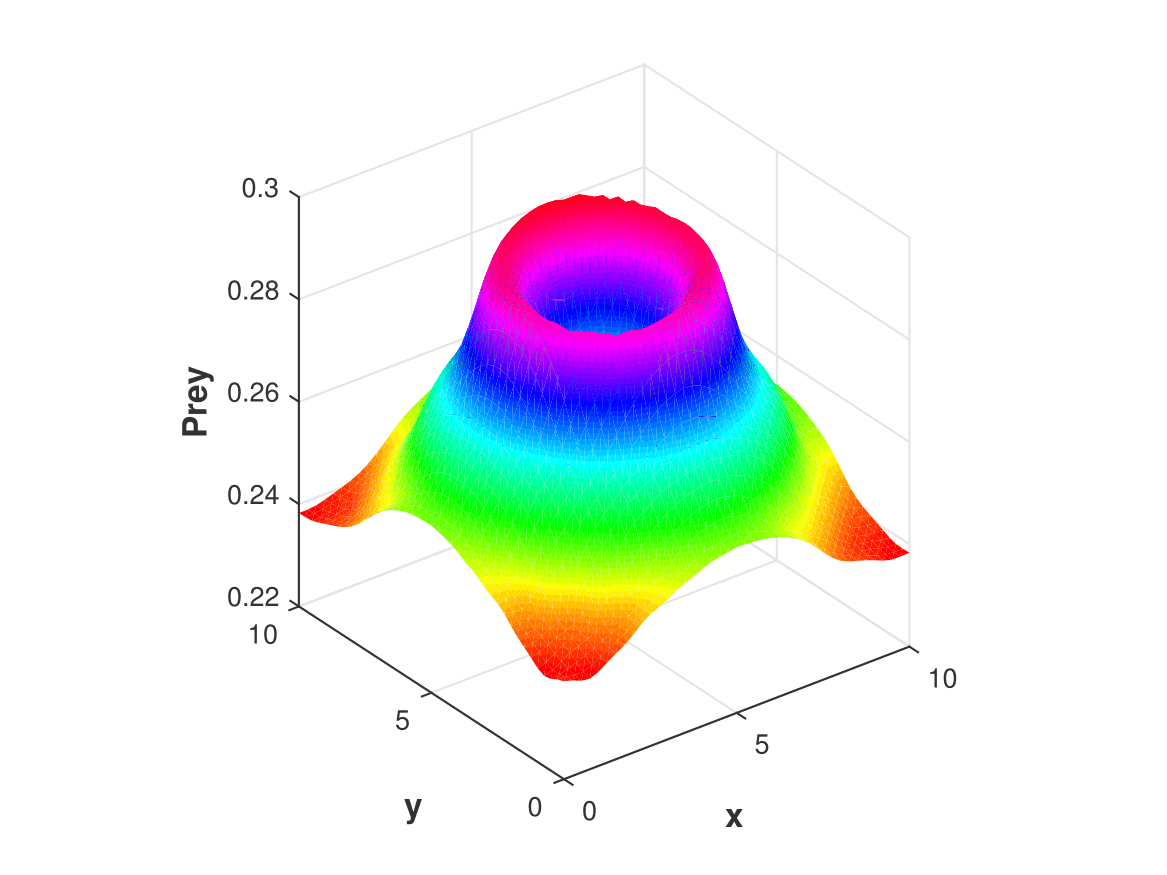} 
\includegraphics[width=0.35\textwidth]{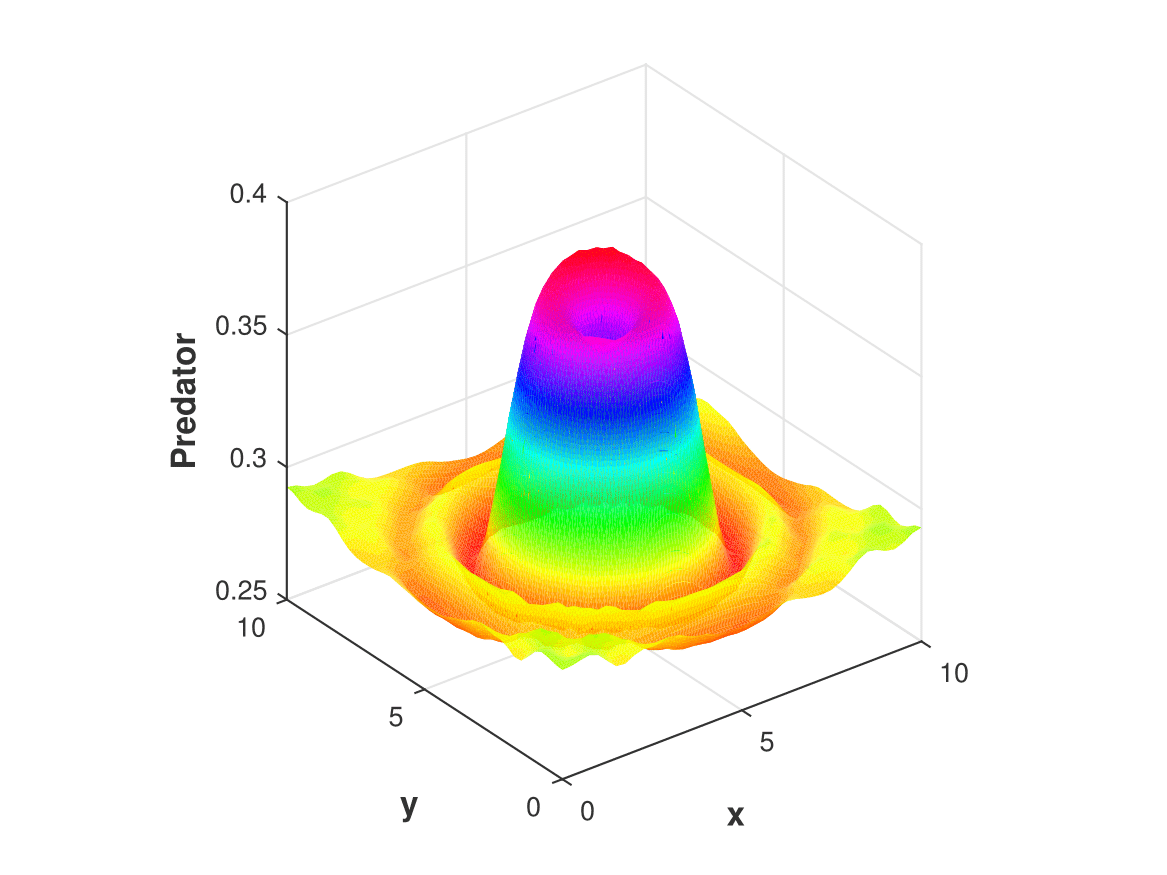}\includegraphics[width=0.35\textwidth]{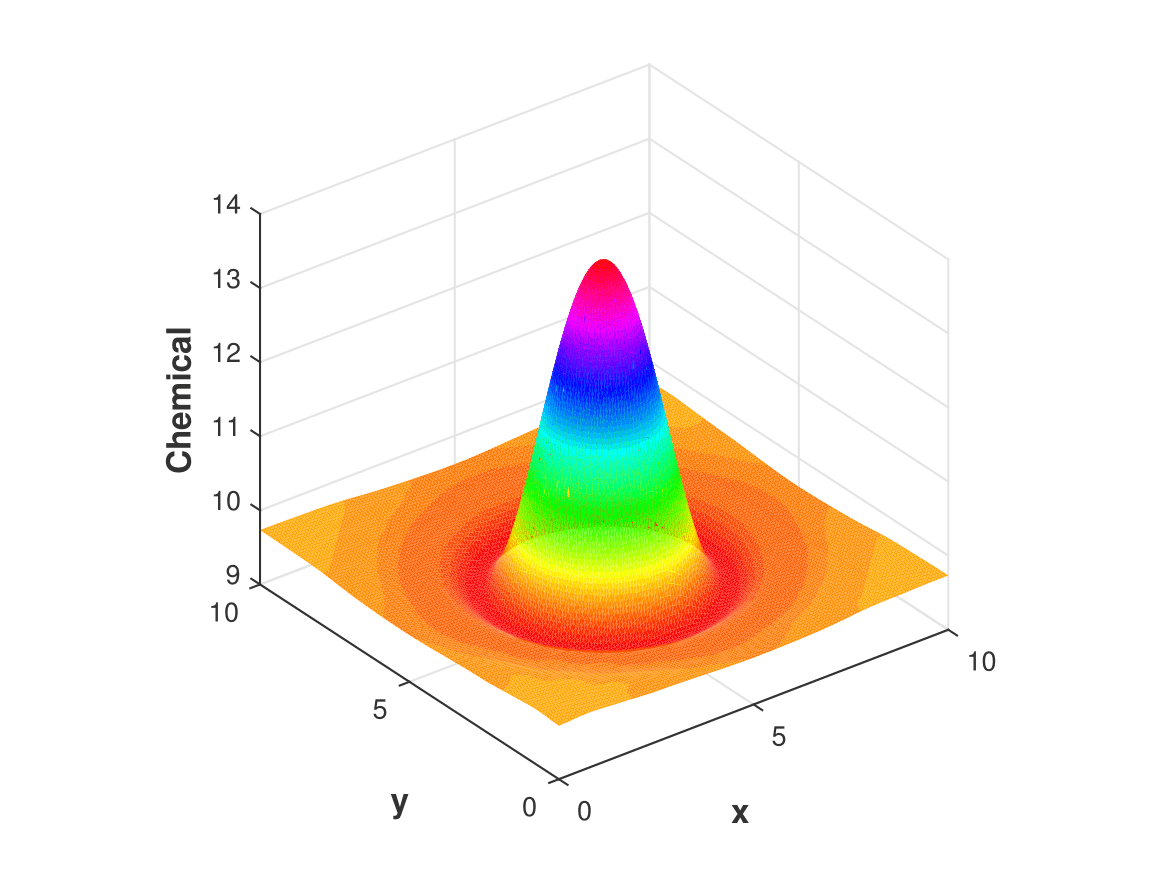}}
\caption{Model A: 2D seperation patterns  for $\chi=5.5$ and $\xi=5.5$ at time steps $t=70$ with with remaining parameters and initial data the same as in Fig. \ref{blowup}.}
\label{chixi}
\end{figure*}
\begin{figure*}[hbt!]  
\centering 
\subfloat[\label{chixi1}]{ \includegraphics[width=0.35\textwidth]{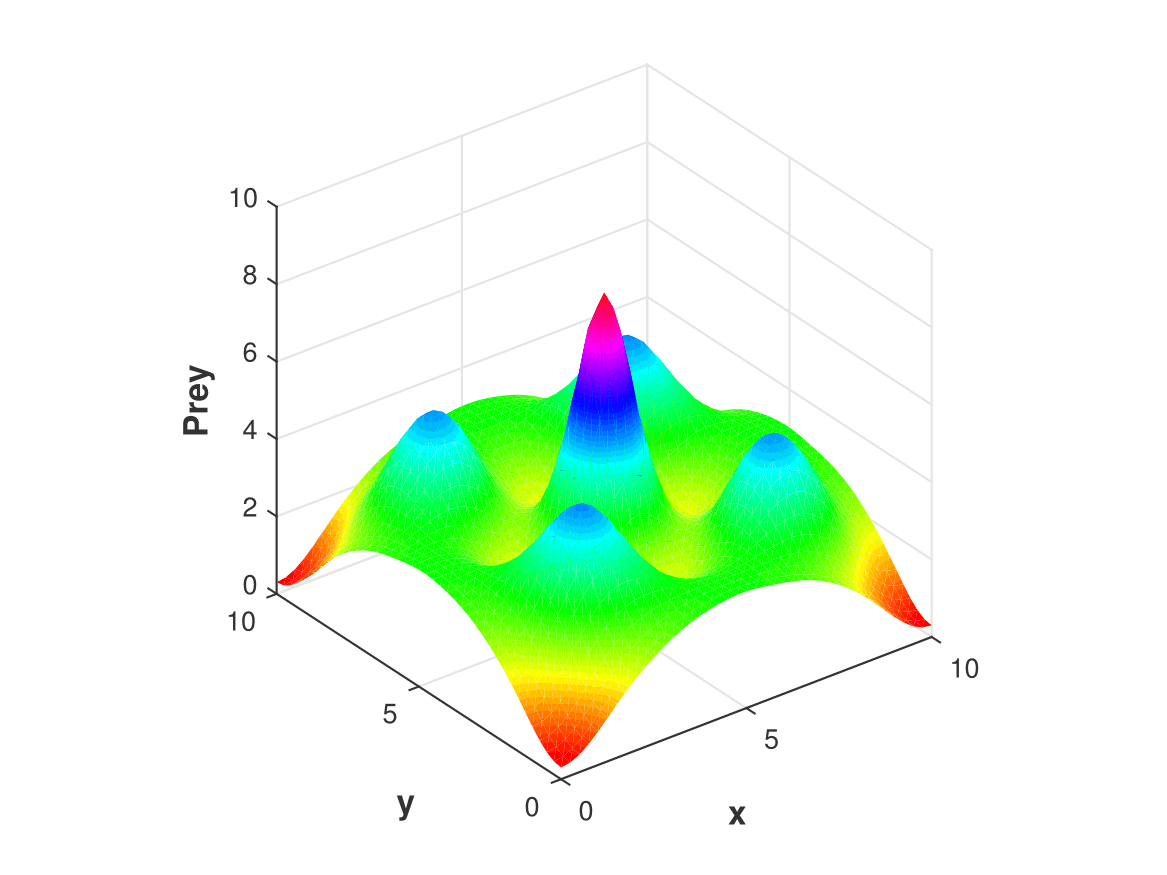} 
\includegraphics[width=0.35\textwidth]{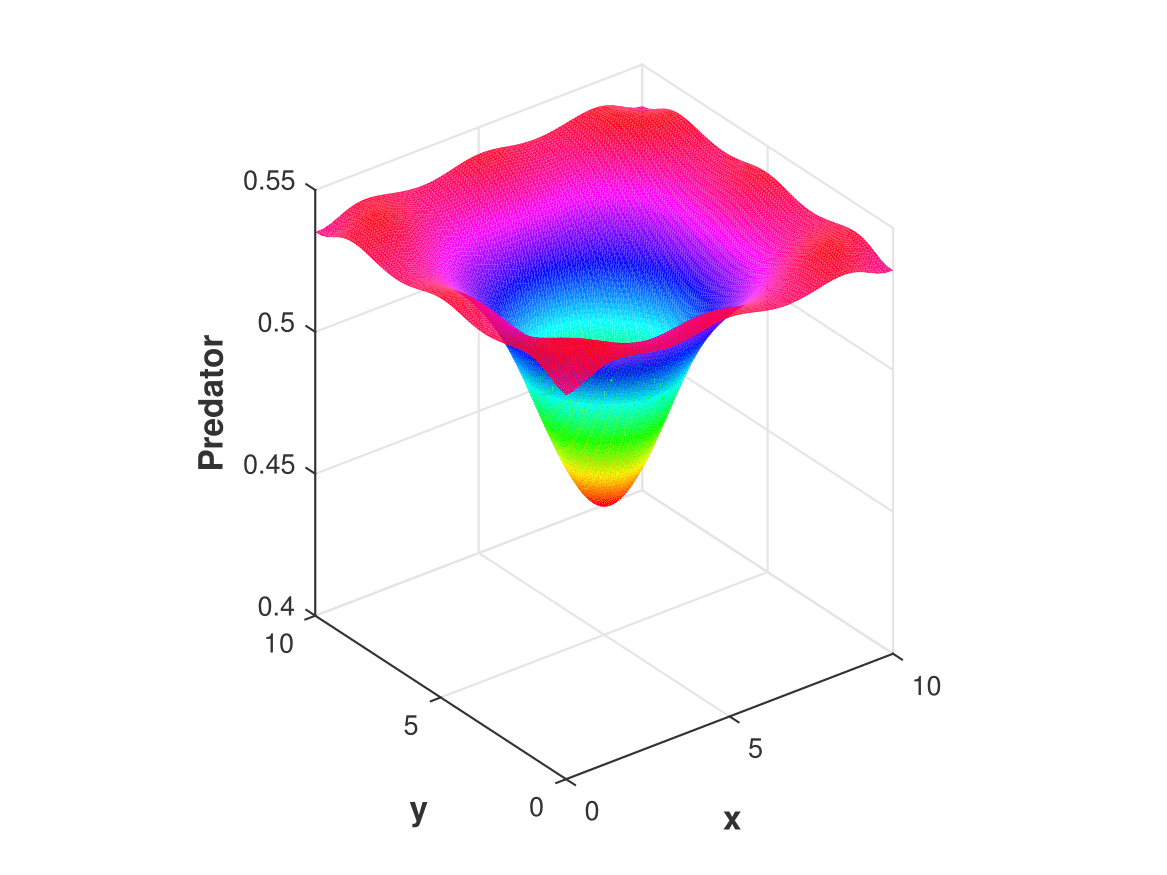}\includegraphics[width=0.35\textwidth]{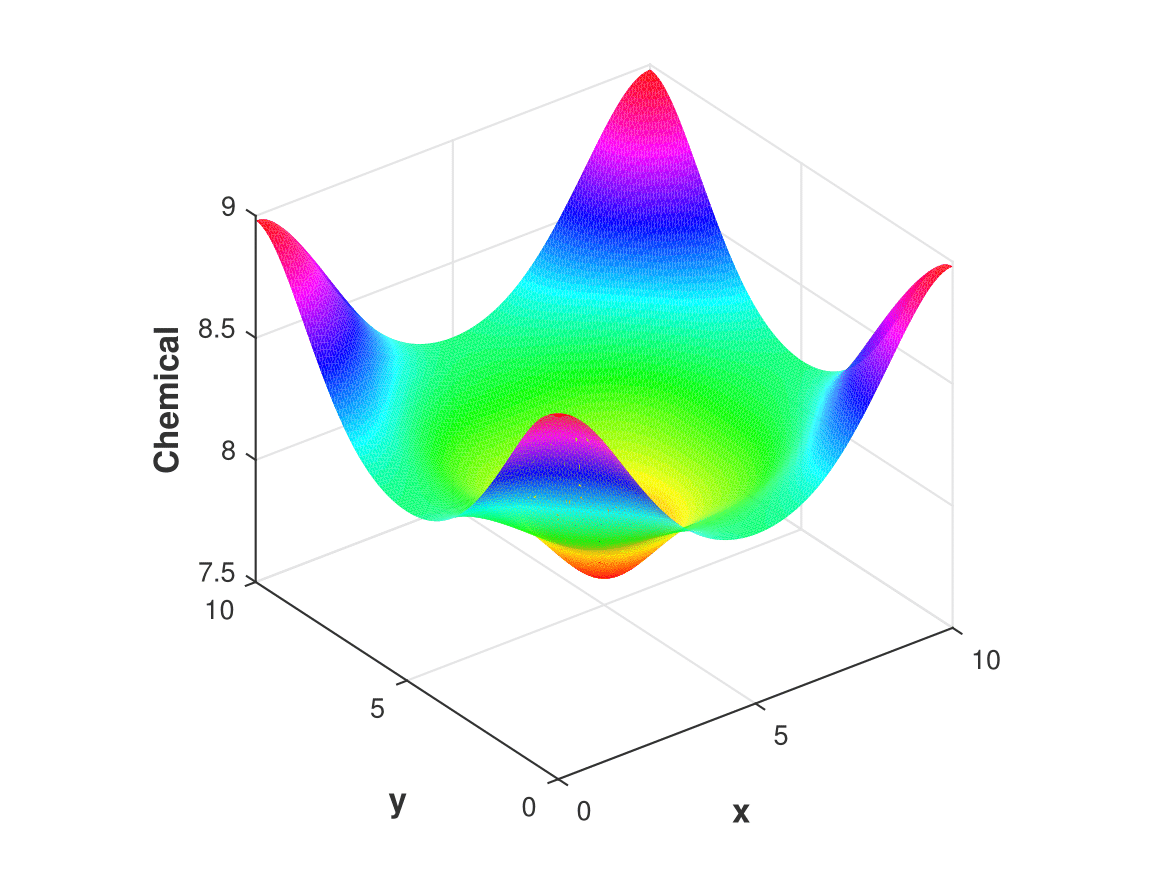}}
\caption{Model B: 2D seperation patterns  for $\chi=10$ ($\xi=0.0$) at time step $t=1500$ with remaining parameters and initial data the same as in Fig. \ref{blowup}.}
\label{chixi3}
\end{figure*}
\begin{figure}[hbt!]  
\centering 
\includegraphics[width=0.35\textwidth]{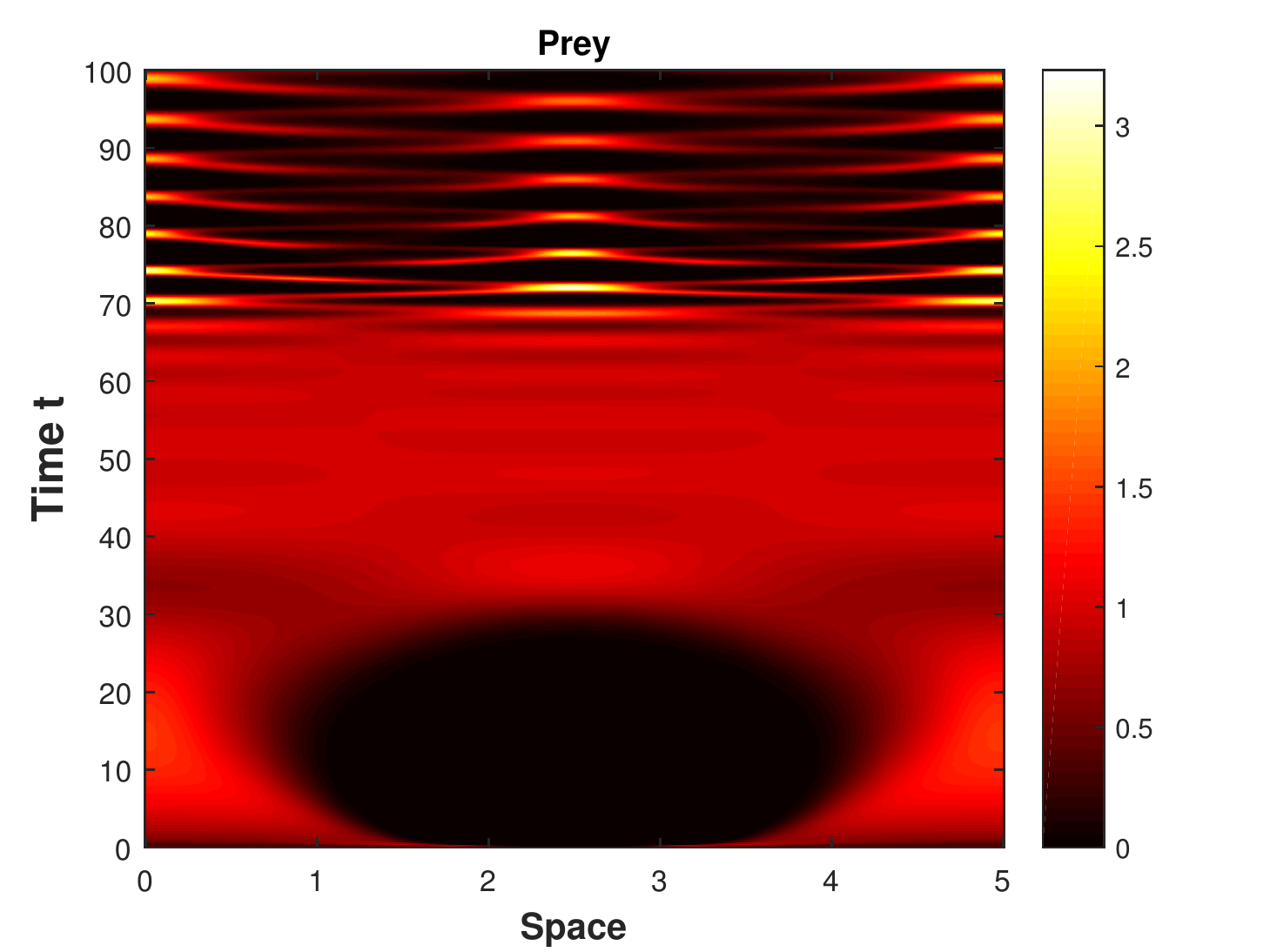} 
\includegraphics[width=0.35\textwidth]{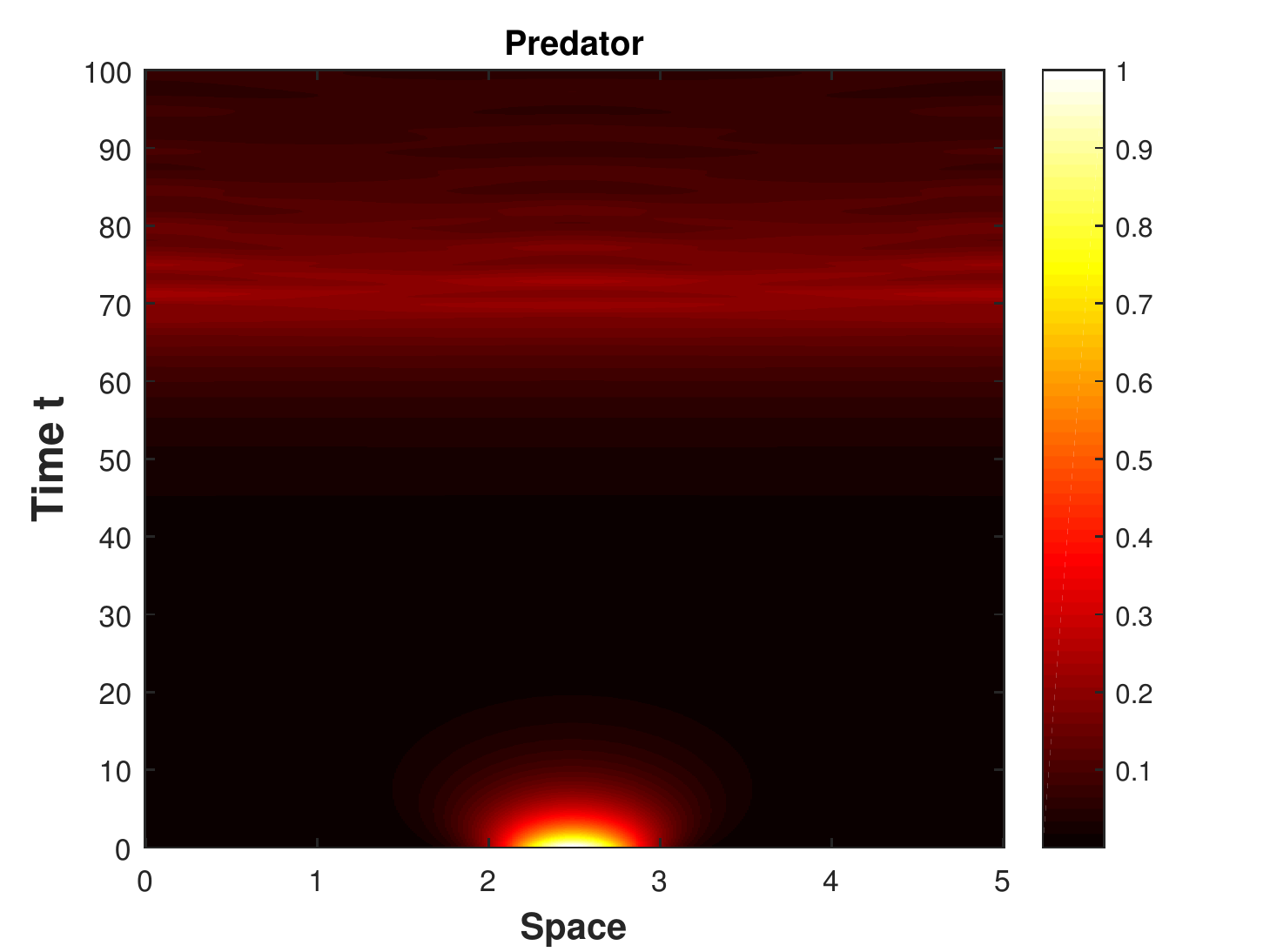}  	
\includegraphics[width=0.35\textwidth]{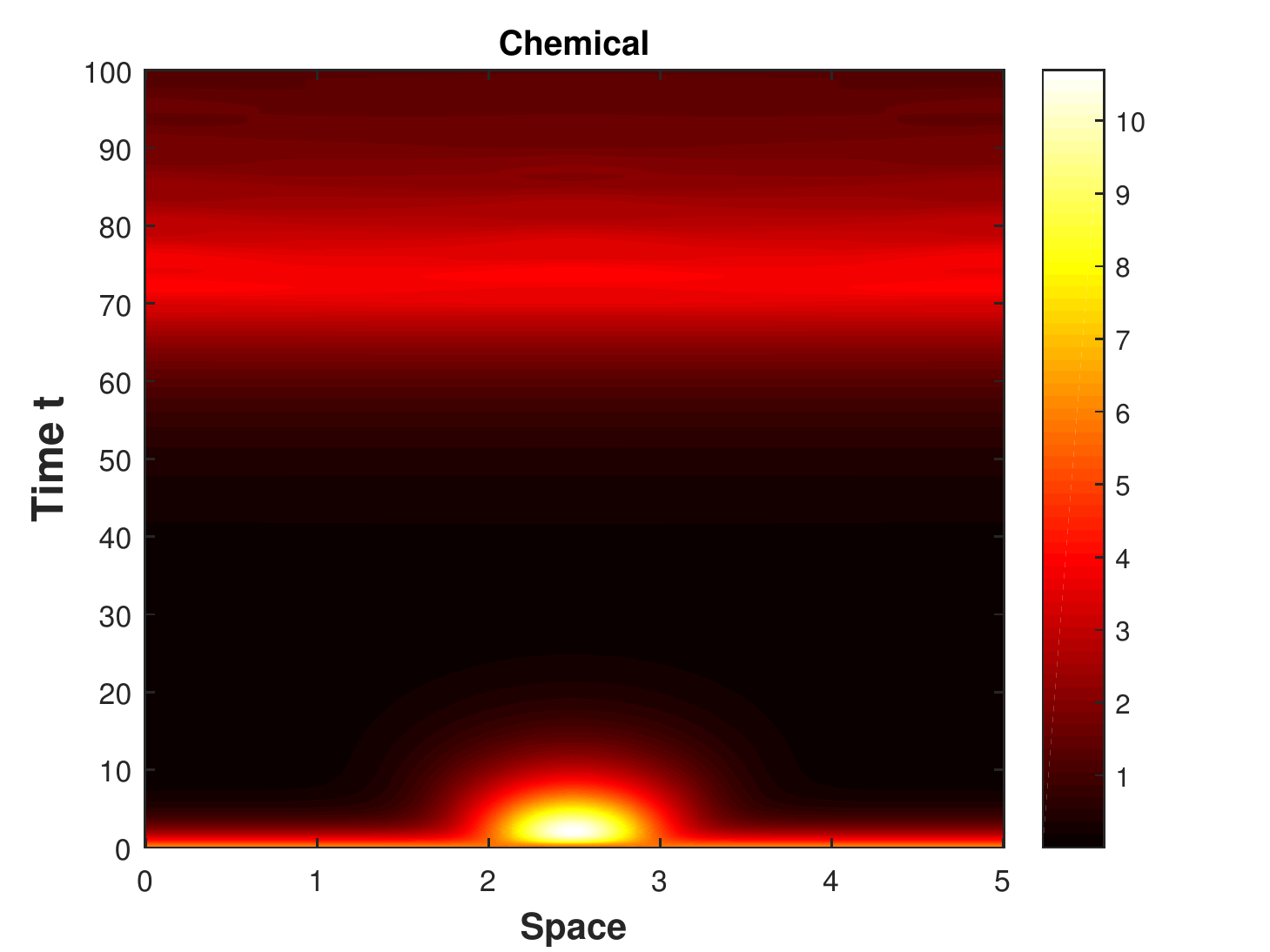} 
\caption{Model A: Spatio-temporal separation patterns emerged from initial data $N(0)=\bar{N},\ \bar{P}=\bar{P}+e^{-(x-2.5)^2},\ W(0)=\bar{W}$ and for the parameter set (\ref{para1}) with $\chi=8,\ \xi=0.27$.}
\label{atp}
\end{figure}

In Figure \ref{chemo} we  present 2D simulation  obtained for the square $(0\,,L)^2$ with $L=10$ with no-flux Neumann boundary conditions, initial data 
\[N(x,0)=\bar{N},\ P(x,0)=\bar{P}+e^{-((x-5)^2-(y-5)^2)},\ W(x,0)=\bar{W}\]
and  the set of parameters (\ref{para1}) with  $\chi>\chi^H$. We have observed complex almost periodic patterns illustrating essentially pursuit and evasion  of predators and prey in space which manifests itself by occurrence of spots of aggregation and  density depletion  varying  in time. It is worth noticing that all evolving patterns  keep the symmetry with respect to the middle of the square. Initially the prey and the chemical are homogeneously distributed in the 2D domain and predator initially dominates in the middle of the domain distributed according to the shifted Gaussian distribution. 

In Fig. \ref{chm1} at time step $t=10$ prey flees away from predator dominant area very quickly and makes a ring shape pattern. On the other hand, the predator and the chemical dominate  in the middle  of the domain. At time step $t=500$ we observe transition of predator central spot pattern into ring pattern which corresponds to the transition of prey  into the  ring pattern with relatively  larger radius (see Fig. \ref{chm2}). In Fig. \ref{chm3} we see  the  prey trying to escape form predator dominance area and forming aggregations  near the  centers of the square’s sides At the same time predator dominates inside rhombus alike structure. In Fig. \ref{chm4}, we observe the prey aggregations  in the corners of the square  while chemical and predator settle down in the centre of the domain. After some time we observe formation of prey aggregation in the centre of the domain while the  chemical leaves the centre of the domain . The numerical simulation suggests that Figs. \ref{chm3}-\ref{chm5} repeats in a fixed time period.

\subsection{Numerical simulation for Model A}
This subsection is devoted to studying model A with the Holling II  functional response (\ref{modelA}). The objective of this subsection is to examine numerically  the simultaneous effect of direct-prey taxis and chemo-repulsion on the  pattern formation. To this end, we first calculate the  critical  value $\xi^S$ (see (\ref{ksis}) numerically with the help of MATLAB. The coefficients of the characteristic polynomial of stability matrix (\ref{jacm2}) for the parameter set (\ref{para1}) are following
\begin{align*}
&\phi_j^1\approx0.55+1.02h_j>0,\ \phi_j^2\approx0.0201h_j^2+0.506h_j+0.1994+0.1776\xi h_j>0,\\
& \phi_j^3\approx0.0001h_j^3+0.005h_j^2+0.002h_j+0.0872+  \xi(0.1746\chi h_j+0.0850)h_j>0,\\
& \phi_j^1\phi_j^2-\phi_j^3\approx 0.0201h_j^3+0.506h_j^2+0.0546h_j+0.1994-\xi(\chi 0.9746h_j-0.0817)h_j.
\end{align*}
The stability threshold value $\xi^S$ from Theorem \ref{ThmstabA} is given by 
\begin{align*}
\xi^S\approx \min_{j\in \mathbb{N}_{+}}\Big\{ \frac{0.0201h_j^3+0.506h_j^2+0.0546h_j+0.1994}{(\chi 0.9746h_j-0.0817)h_j}\Big\}.
\end{align*}
For the set of parametrs (\ref{para1}) the minimum is attained at $h_j=h_1=\frac{\pi^2}{L^2}$. It is worth to mention that a positive  $\xi^S$ exists if and only if $\chi>\chi^S=\frac{0.0817}{0.974h_j}$ (see (\ref{RHA}))  which may give rise to Hopf type taxis-driven instability if chemo-repulsion taxis rate $\xi$ is big enough. One can easily see that $\xi^S$ depend upon the chemo-repulsive taxis rate $\chi$ and we numerically obtain that taxis-driven instability may emerge if  chemo-repulsive sensitivity coefficient  $\chi>\chi^S\approx 0.0267$ and $\xi$ is large enough (c.f. Theorem \ref{ThmstabA}).

First of all, we numerically show in accordance with Theorem \ref{ThmstabA}, that a small spatial perturbation of the  constant steady state $\bar{E}$ in model A (\ref{modelA}) does not affect the stability of the system for $\xi<\xi^S$. In Fig. \ref{fig9}, we observe that a spatial perturbation (\ref{idata}) at eigenmode $j=1$ in unit domain $L=1$ and $\chi=0.2$ converges to the constant steady state $\bar{E}$ if $\xi<\xi^S\approx  3.8144$. In the next figure \ref{fig99}, we keep all parameters and initial data the same to plot solution for a bigger $\xi$ i.e $\xi=10>\xi^S$. We observe then regular space-time pattern  and it  is worth  noticing  at this point that in the presence of chemotactic repulsion the prey-taxis is capable to destabilize  the coexistence steady state $\bar{E}$. This observation is worth underlining in the context of common opinion that the prey-taxis promotes the stability of coexistence steady state in predator-prey models  which didn't take into account the repulsive chemotaxis \cite{Bendamane,Lee,Tao}. One can see regular-space time pattern of small amplitude for prey in Fig. \ref{fig991}. However space-time patterns with larger amplitude are observed for predator and chemical (see Figs. \ref{fig992} \& \ref{fig992}) which are settled in boundary of domain. 

In simulation related to Fig. \ref{fig10} we intended to investigate the transition of pattern depending upon the strength of taxis. To this end, we run simulations for initial data (\ref{idata}) at eigenmode $j=1$ with $\chi=5$ and calculated the threshold $\xi^S=0.1464$. It is observed that if chemo-repulsion is stronger than prey-taxis rate then prey individuals flee to the corners and space-time separation pattern appears (see Fig. \ref{fig101}). However, predator also try to follow prey which gives rise to regular space-time pattern with small amplitude (see Fig. \ref{fig102}). The results obtained from Fig. \ref{fig99} \&  \ref{fig10} confer that large amplitude of space-time pattern depends upon the choice of taxis parameters. If chemo-repusion is higher than prey taxis rate then prey may have space-time separational patterns with lager amplitude and predator exhibits large amplitude space-time pattern if  prey taxis rate is higher than chemo-repulsion.

 
Next we show transient dynamics  for model A (\ref{modelA}) in the enlarged domain $L=10$. Presented numerical results show how prey-taxis  affects the pattern formation  in larger domain $L=10$ for fixed chemo-sensitivity coefficient ($\chi=2$) with corresponding  $\xi^S=0.082$ with the remaining parameters kept unchanged (\ref{para1}). Fig. \ref{fig131} corresponds to the  solution starting from initial data (\ref{idata}) with $j=4$ showing transition (at time $t=100$)  from a  regular space-time rhombous pattern to some other space inhomogeneous  structure with dominance  of prey at the ends of the domain interval. It is worth noticing that this transition is accompanied with the change in both period and amplitude of space-time fluctuations. As we increase prey-taxis sensitivity coefficient (see Fig. \ref{fig132}) it is observed that regular rhombus-alike structure resembling beehive appears immediately. This result reveals that prey-taxis is not only able to destabilize the predator-prey system but also has immense impact in the shaping of patterns. 

In 2D case  we observe more complex  behavior  of the solutions to model A (\ref{modelA}) in which  additionally  the prey taxis comes into play. We run 2D simulations in FreeFem++ package in order to investigate the simultaneous impact  of the chemo-repulsive taxis and the direct taxis on the behavior of solutions to model A.  Fig. \ref{blowup}   presents snapshots of surface plot observed at different time moments representing solutions starting from the initial data shaped as  shifted Gaussian distribution for prey and predator i.e $N_0=\bar{N}+e^{-((x-5)^2+(y-5)^2)},\ P_0=\bar{P}+e^{-((x-5)^2+(y-5)^2)}$ with  homogeneous distribution of the chemical $W_0=\bar{W}$. It has been observed that  prey and predator already at time step  $t=10$  exhibit  similar spiky structure (see Fig. \ref{bo1}) which is getting sharper and sharper over  time so that by obvious reasons any numerical approximation loses gradually its accuracy before reaching a sharp spike shape depicted  at time step $t=134$ (see Fig. \ref{bo3}). It is important to note that Fig.\ref{blowup} is presented for the situation when direct prey taxis is significantly  stronger than chemo-repulsive taxis (i.e. $\chi=0.5$ \& $\xi=10$) and all other parameters are the same as in (\ref{para1}). A possible interpretation of the singularity formation process is the following. At early stage of the process Fig.\ref{bo1} we may observe a rapid grow of the  density  function of the chemical    produced by the predator in the middle of the domain which  forms a steep spiky round hill of the chemical density surrounded   by  a valley. At the external valley slope  there is a gradient vector field directed outward the center. The opposite direction to this field is our chemorepulsion force forming a kind of barrier  which pushes the prey toward the center and stops  from  escaping the region limited by the round valley. At the same time the strong prey taxis directed toward the center of the domain results in both rapid shrinking and growth of the round spiky hill  and formation of high predator density in the middle of the domain. A closer look at this figure suggests that the prey try to escape from the predator dominant area but it is  less effective because prey-taxis is much stronger than  chemorepulsive taxis. 

Another scenario happens in Fig. \ref{chixi} when $\chi$ and $\xi$ are  equal each other. In this case due to relatively stronger chemorepulsion the prey is pushed out of the central region  with high chemical density and  then the predator density resembles a core surrounded by the density of prey which escapes outward the middle of the domain. This is a cumulative effect of both taxis mechanisms. It is interesting to see a dramatic difference between  the previous figure and Fig \ref{chixi3} when $\xi=0$  where we see a nice symmetric and periodic patterns which resemble those in Fig. \ref{chemo}. In particular in accordance with our theoretical results for larger time no singularity formation  takes place.

\section{Conclusions}
In this paper we  considered  two diffusive prey-predator models which take into account the reception of  chemical signals by prey which indicate  the location of predators. More precisely we investigated the avoidance of predator by prey upon  detection of chemical released by predator (e.g. predator odor) which stimulates migration outward  the gradient of the  chemical concentration-one of many possible antipredatory strategies observed in nature  \cite{Banks,Ferrari,Hay}. It is worth to notice that chemical signals with  various mechanisms of production may induce many other antipredatory adaptations in prey which demand further modeling efforts.  

The following remarks related to the results  obtained in this paper are worth underlining.
\begin{itemize}
\item  Classical diffusive  prey-predator  models enriched by terms accounting for  chemical signaling can describe the tendency to   spatio-temporal separation between prey and predators, by either avoiding areas inhabited by  predators or using those areas at different times than the predators. 
\item While trying to prove the existence of global  in time classical solutions to  model A which  contains two taxis terms we faced limitations in extending the proof to  higher space dimensions then $n=1$. Numerical solutions (see Fig. \ref{blowup})  indicate that  no classical solution is  expected in this case. It seems however, that a suitably defined weak solution to model A  exists for $n=2$. Interestingly, the formation of  blow-up solution in finite time is evidently related to the cumulative effect of both taxis mechanisms built-in to model A because each  of the two systems with a single taxis mechanism posses global classical solutions in space dimension $n=2$. The effect seems to be new and demands further studies. From the modeling view point it seems reasonable to consider a predator-prey model linking  the chemorepulsive evasion as response to an olfactory signal from predator  with negative  predator taxis corresponding to a visual detection of predators by prey. 
  
\item The most important feature stemming from the stability analysis of the coexistence steady state in model A and model B is the  destabilizing effect of the  repulsive chemotaxis which plays its role even in  the case when direct prey taxis is concerned. The latter is known to  stabilize  the coexistence steady state in  prey-predator models of reaction-diffusion type (at least when the Holling functional response is considered). Moreover, the stabilizing effect acts even when the chemosensitivity parameter $\chi$ exceeds the critical value $\chi^H$ from  model B provided the prey taxis effect measured in terms of the parameter $\xi$ is strong enough (c.f. Theorem \ref{ThmstabA}). 
\item Yet another consequence of the linear stability analysis is the type of bifurcation which may occur at the critical value of bifurcation parameter $\chi$. It turns out that any static bifurcation is precluded and only dynamic bifurcation of Hopf type may exist in the class of models studied in the present paper. 
\item Numerical simulations suggest  that evasive defense strategy of prey based on  chemical signaling may lead to the formation of complex space-time patterns of species  distribution. Solution patterns depicted in Section \ref{NumSim} for model A lead to interesting questions to be studied theoretically including abrupt in time change of patterns (see Fig \ref{figa13}) and blow-up solutions in 2D Fig. \ref{blowup}. Yet another effect worth further investigation is the transition of initial perturbation from one component of the system to another  as depicted in Figure 14 where initial perturbation only in predator population gives rise to strong regular pattern in  prey population with simultaneous  decay of fluctuation in the predator population. 
\end{itemize}

\noindent
{\bf Acknowledgments}

\noindent
Purnedu Mishra extends appreciation to ERCIM, the European Research Consortium for Informatics and Mathematics for
funding this work as well as  to the Faculty of Mathematics and Mechanics, University of Warsaw for hospitality during his post-doc stay. 
 
\end{document}